\documentclass[]{article}

\title{Kernel Mean Embedding of Probability Measures and its Applications to Functional Data Analysis}

\author{Saeed Hayati\\\href{mailto:s.hayati@sci.ui.ac.ir}{\texttt{s.hayati@sci.ui.ac.ir}}
	\and
	Kenji Fukumizu\\\href{mailto:fukumizu@ism.ac.jp}{\texttt{fukumizu@ism.ac.jp}}
	\and
	Afshin Parvardeh\\\href{mailto:a.parvardeh@sci.ui.ac.ir}{\texttt{a.parvardeh@sci.ui.ac.ir}}
}

\RequirePackage[OT1]{fontenc}
\RequirePackage{amsthm,amsmath}
\RequirePackage[numbers]{natbib}
\RequirePackage[colorlinks,citecolor=blue,urlcolor=blue]{hyperref}
\usepackage{graphicx}
\usepackage[latin9]{inputenc}
\usepackage{float}
\usepackage{mathrsfs}
\usepackage{amssymb}
\usepackage{units}
\usepackage{multirow}
\usepackage{subcaption}

\DeclareMathOperator*{\argmax}{arg\,max}

\makeatletter
\theoremstyle{plain}
\newtheorem{thm}{\protect\theoremname}
\theoremstyle{plain}
\newtheorem{lem}[thm]{\protect\lemmaname}
\theoremstyle{plain}

\theoremstyle{plain}
\newtheorem{cor}[thm]{\protect\corollaryname}
\theoremstyle{plain}
\newtheorem{prop}[thm]{\protect\propositionname}
\theoremstyle{definition}
\newtheorem{defn}[thm]{\protect\definitionname}

\makeatother

\providecommand{\corollaryname}{Corollary}
\providecommand{\definitionname}{Definition}
\providecommand{\lemmaname}{Lemma}
\providecommand{\propositionname}{Proposition}
\providecommand{\remarkname}{Remark}
\providecommand{\theoremname}{Theorem}

\def\code#1{\texttt{#1}}

\makeatletter
\renewenvironment{proof}[1][]{%
	\def\temp{#1}
	\ifx\temp\empty
	\par\pushQED{\qed}\normalfont%
	\topsep6\p@\@plus6\p@\relax
	\trivlist\item[\hskip\labelsep\bfseries#1\proofname\@addpunct{.}]%
	\ignorespaces
	\else
	
	\par\pushQED{\qed}\normalfont%
	\topsep6\p@\@plus6\p@\relax
	\trivlist\item[\hskip\labelsep\bfseries#1]%
	\ignorespaces
	
	\fi
}{%
	\popQED\endtrivlist\@endpefalse
}
\makeatother
\allowdisplaybreaks

\begin{document}

\maketitle

\begin{abstract}
This study intends to introduce kernel mean embedding of probability measures over infinite-dimensional separable Hilbert spaces induced by functional response statistical models. The embedded function represents the concentration of probability measures in small open neighborhoods, which identifies a pseudo-likelihood and fosters a rich framework for statistical inference. Utilizing Maximum Mean Discrepancy, we devise new tests in functional response models. The performance of new derived tests is evaluated against competitors in three major problems in functional data analysis including function-on-scalar regression, functional one-way ANOVA, and equality of covariance operators.
\end{abstract}

\section{Introduction}
Functional response models are among the
major problems in the context of Functional Data Analysis. A fundamental
issue in dealing with functional response statistical models arises due to the
lack of practical frameworks on characterizing probability measure
on function spaces. This is mainly a consequence of the tremendous gap on
how we present probability measures in finite-dimensional and infinite-dimensional
spaces. 

A useful property of finite-dimensional spaces is the existence of a locally finite, strictly positive, and translation
invariant measure like Lebesgue or counting measure, which makes it
easy to take advantage of probability
measures directly in the statistical inference. Fitting a statistical model, and estimating parameters,
hypothesis testing, deriving confidence regions and developing
goodness of fit indices, all can be applied by integrating distribution
or conditional distribution of response variables as a presumption
into statistical procedures.

Sporadic efforts have been gone into approximating or representing probability measures
on infinite-dimensional spaces. Let $\mathbb{H}$ be a separable infinite-dimensional Hilbert space and $X$ be a $\mathbb{H}$-valued random element with finite second moment and covariance operator $C$. \citet{delaigle2010defining}
approximated probability of $B_{r}\left(x\right)=\{\parallel X-x\parallel<r\}$
by the surrogate density of a finite-dimensional approximated version of $X$, obtained by projecting the random element
$X$ into a space spanned by first few eigenfunctions of $C$ with largest eigenvalues. The approximated small-ball probability
is on the basis of Karhunen-Lo\`eve expansion and putting an extra assumption that the
component scores are independent. The precision of this approximation depends on the volume of 
ball and probability measure itself.

%Their approximation of the
%small-ball probability is as what comes in follow.
Let $\mathcal{I}$ be a
compact subset of $\mathbb{R}$ such as closed interval $[0,1]$ and $X$ be a zero mean
$L^{2}\left[\mathcal{I}\right]$-valued random element with finite
second moment and Karhunen-Lo\`eve expansion $X=\sum_{j\geq1}\lambda_{j}^{1/2}X_{j}\psi_{j}$,
in which $X_j=\lambda_j^{-1/2}\langle X,\psi_j\rangle$ and $\left\{ \lambda_{j},\psi_{j}\right\} _{j\geq1}$ is the
eigensystem of covariance operator $C$. Suppose that the distribution of $X_j$ is
absolutely continuous with respect to the Lebesgue measure with
density $f_{j}$. Approximation of the logarithm of $p\left(x\,|\,r\right)=P\left(B_{r}\left(x\right)\right)=P\left(\{\parallel X-x\parallel<r\}\right)$
given by \citet{delaigle2010defining} is 
\[
\log p(x\,|\,r)=C_{1}(h,\left\{ \lambda_{j}\right\} _{j\geq1})+\sum_{j=1}^{h}\log f_{j}(x_{j})+o(h),
\]
in which $x_{j}=\left\langle x,\psi_{j}\right\rangle $, and $h$
is the number of components that depends on $r$ and tends to infinity
as $r$ declines to zero. 
%$C_{1}\left(\cdot\right)$ and $h$ depend neither on $x$ nor the probability measure $P$. 
$C_{1}\left(\cdot\right)$ 
depends only on size of the ball and sequence of eigenvalues, though the quantity $o(h)$ as the
precision of approximation depends on $P$.

The quantity $h^{-1}\sum_{j=1}^{h}\log f_{j}(x_{j})$ is called log-density
by \citet{delaigle2010defining}. A serious concern with this approximation
is its precision, which depends on the probability measure itself.
Accordingly, it can not be employed to compare small-ball probability in a family
of probability measures. For example, in the case
of estimating the parameters in a functional response regression model,
the induced probability measure varies with different choices of parameters.
Thus this approximation can not be employed for parameter estimation and comparing the
goodness of fit of different regression models.

Another work in representing probability measures on a general separable
Hilbert space $\mathbb{H}$ presented by \citet{lin2018mixture}. They constructed a
dense subspace of $\mathbb{H}$ called Mixture Inner Product Space (MIPS), which is the union of a countable collection of finite-dimensional subspaces of $\mathbb{H}$. An approximating version of the given $\mathbb{H}$-valued
random element lies in this subspace, which in consequence, lies
in a finite-dimensional subspace of $\mathbb{H}$ according
to a given discrete distribution. They defined a base measure on the MIPS, which is not translation-invariant, and introduced density functions for the MIPS-valued random elements.

%\textcolor{red}{
%In a practical viewpoint,
%both ideas are still more applicable to Gaussian Probability Measure
%models and the application of these methods are restricted mostly
%to problems such as clustering among those we could mention \citet{Delaigle2012,jacques2014model,Delaigle2019}. It can be shown as well, that in a %functional response linear regression model and with the Gaussian assumption, one can derive a sequence of estimators of
%location parameters with both approaches which converges to the ordinary least square estimator.}

Absence of a proper method
in representing probability measures over infinite-dimensional spaces
caused severe problems to statistical inference. To make it clear,
as an example \citet{greven2017general} 
developed a general framework for functional additive mixed-effect
regression models. They considered a log-likelihood function by summing
up the log-likelihood of response functions $Y_i$ at a grid of time-points $t_{id},~d=1,\ldots,D_i$, assuming $Y_i\left(t_{id}\right)$ to be independent within the grid of time-points.
A simulation study by \citet{kokoszka2017discussion} revealed the weak performance of the proposed framework in statistical hypothesis testing in a simple Gaussian Function-on-Scalar linear regression problem.

%In the lack of a Haar measure in infinite-dimensional spaces, 
Currently, MLE
and other density-based methods are 
%currently 
out of reach in the
context of functional response models. In this study, we follow a
different path by identifying probability measures with their kernel
mean functions and introduce a framework for statistical inference
in infinite-dimensional spaces. A promising fact about the kernel
mean functions, which is shown in this paper, is their ability to reflect the concentration of probability
measures in small open neighborhoods, where unlike the approach of \citet{delaigle2010defining}
is comparable among different probability measures. This property
of kernel mean function motivates us to make use of it in fitting statistical models and introducing new statistical tests in the context of functional data analysis. 

%If the embedding is injective, kernel mean embedding induce a metric on the space of probability measures over $\mathbb{H}$, namely Maximum Mean Discrepancy (MMD)  %\citep{muandet2017kernel}. \citet{pan2018} Offered a similar quantity called Ball divergence to quantify the difference between probability measure over separable Banach spaces. %However, for the case of infinite-dimensional spaces, Ball divergence could only distinguish two probability measures if at least on of them possess a full support i.e. %$\text{Supp}\left(P\rigt)=\mathbb{H}}$. They employed Ball divergence for a two sample test, where according to their simulation results both MMD and Ball divergence have a close performance.

This paper is organized as follows: In Section \ref{sec2}, kernel mean embedding
of probability measures over infinite-dimensional separable Hilbert spaces is discussed. In Section \ref{sec3}
the Maximum Kernel Mean estimation method is introduced and estimators
for Gaussian Response Regression models are derived.
In Section \ref{sec::aplications}, new statistical tests are developed for three major problems in functional data analysis and their performance evaluated using simulation studies. Section \ref{sec:discuss} has been devoted to discussion and conclusion. Major proofs are aggregated in the appendix.

\section{Kernel mean embedding of probability measures }\label{sec2}

We summarize the basics of kernel mean embedding.  See \citet{muandet2017kernel} for a general reference. Let $(\mathbb{H},B\left(\mathbb{H}\right),P)$ be a probability measure
space. Throughout this study $\mathbb{H}$ is an infinite-dimensional separable Hilbert space  equipped with inner product
$\langle\cdot,\cdot\rangle_{\mathbb{H}}$. 
A function $k$ : $\mathbb{H}\times\mathbb{H}\to\mathbb{R}$
is a \textit{positive definite kernel} if it is symmetric, i.e., $k(x, y) = k(y, x)$ and 
$\sum_{i=1}^{n}a_{i}a_{j}k(x_{i},x_{j})\geq0$
for all $n\in\mathbb{N}$ and $a_{i}\in\mathbb{R}$ and $x_{i}\in\mathbb{H}$. $k$ is \textit{strictly positive definite} if equality implies $a_{1}=a_{2}=\ldots=a_{n}=0$.
$k$ is said to be \textit{integrally strictly positive definite} if $\int k(x,y)\mu(dx)\mu(dy)>0$
for any non-zero finite signed measure $\mu$ defined over $(\mathbb{H},B\left(\mathbb{H}\right))$.
Any integrally strictly positive definite kernel is strictly positive
definite while the converse is not true \citep{sriperumbudur2010hilbert}. A positive definite kernel
induces a Hilbert space of functions over $\mathbb{H}$, which is called
Reproducing Kernel Hilbert Space (RKHS) and equals to $\mathcal{H}_{k}=\overline{span}\{k(x,\cdot);x\in\mathbb{H}\}$
with inner product  
\[
\langle\sum_{i\geq1}a_{i}k(x_{i},\cdot),\sum_{i\geq1}b_{i}k(y_{i},\cdot)\rangle_{\mathcal{H}_{k}}=\sum_{i\geq1}\sum_{j\geq1}a_{i}b_{j}k(x_{i},y_{j}).
\]
For each $f\in \mathcal{H}_{k}$ and $x\in\mathbb{H}$ we have $f(x)=\langle f,k(.,x)\rangle_{\mathcal{H}_k}$,
which is the reproducing property of kernel $k$. A strictly positive
definite kernel $k$ is said to be \textit{characteristic} for a family of
measures $\mathscr{P}$ if the map 
\[
m:\mathscr{P}\rightarrow \mathcal{H}_{k}\qquad P\mapsto\int k(x,.)P(dx)
\]
is injective. If $\mathbb{E}_{P}(\sqrt{k(X,X)})<\infty$ then $m_{P}(\cdot):=(m(P))(\cdot)$
exists in $\mathcal{H}_{k}$ \citep{muandet2017kernel}, and the function $m_{P}(\cdot)=\int k(x,\cdot)P(dx)$ is
called kernel mean function. Moreover, for any $f\in \mathcal{H}_{k}$
we have $\mathbb{E}_{P}[f(X)]=\langle f,m_{P}\rangle_{\mathcal{H}_{k}}$ \citep{smola2007hilbert}.
Thus, if kernel $k$ is characteristic then every probability measure
defined over $(\mathbb{H},\Sigma)$ is uniquely identified by an element
$m_{P}$ of $\mathcal{H}_{k}$ and \textit{Maximum Mean Discrepancy} (MMD) defined as
\begin{align}\label{formula::MMD_def}
\text{MMD}(\mathcal{H}_{k},\mathbb{P},\mathbb{Q}) & =\sup_{f\in \mathcal{H}_{k},\left\Vert f\right\Vert _{\mathcal{H}_{k}}\leq1}\left\{ \int f(x)\mathbb{P}(dx)-\int f(x)\mathbb{Q}(dx)\right\}\nonumber \\&=\sup_{f\in \mathcal{H}_{k},\left\Vert f\right\Vert _{\mathcal{H}_{k}}\leq1}\left\langle f,m_{\mathbb{P}}-m_{\mathbb{Q}}\right\rangle =\left\Vert m_{\mathbb{P}}-m_{\mathbb{Q}}\right\Vert _{\mathcal{H}_{k}},
\end{align}
is a metric on the family of
measures $\mathscr{P}$ over $\mathbb{H}$
\citep{muandet2017kernel}.
%If the embedding is injective, kernel mean embedding induce a metric on the space of probability measures over $\mathbb{H}$, namely Maximum Mean Discrepancy (MMD) \citep{muandet2017kernel}.

A similar quantity called Ball divergence is proposed by \citet{pan2018} to distinguish probability measures defined over separable Banach spaces. For the case of infinite-dimensional spaces, Ball divergence distinguishes two probability measures if at least one of them possesses a full support, that is, $\text{Supp}\left(P\right)=\mathbb{H}$. They employed Ball divergence for a two-sample test, which according to their simulation results, the performance of both MMD and Ball divergence are close and superior to other tests.

Kernel mean functions can also be used to reflect the concentration
of probability measures in small-balls, if the kernel function is translation-invariant.
A positive definite kernel $k$ is called translation-invariant if
$k(x,y)=\psi(x-y)$ for some positive definite function $\psi$. Gaussian
kernel $e^{-\sigma\parallel x-y\parallel_{\mathbb{H}}^{2}}$ and Laplace
kernel $e^{-\sigma\parallel x-y\parallel_{\mathbb{H}}}$ are such
kernels.
If we choose a continuous characteristic kernel that is bounded and translation-invariant, then the
kernel mean function $m_{P}$ can be employed to represent the concentration
of probability measure in different points of Hilbert space $\mathbb{H}$.
for example, consider 
\[
m_{p}(x)={\displaystyle \int\limits _{\mathbb{H}}e^{-\sigma\left\Vert x-y\right\Vert _{\mathbb{H}}^{2}}P(dy)}.
\]

If $m_{P}(\cdot)$ has an explicit form for a family of probability
measures then $m_{P}(\cdot)$ can be employed to study and compare different
probability measures. For example, if $m_{P}(x_{1})>m_{P}(x_{2})$
then it could be concluded that the concentration of probability measure
$P$ around the point $x_{1}$ is higher than $x_{2}$, and if for given
two probability measures $P_{1}$ and $P_{2}$ we had $m_{P_{1}}(x)>m_{P_{2}}(x)$
then we conclude that the concentration of probability measure
$P_{1}$ around the point $x$ is higher than that of probability measure
$P_{2}$. This property of kernel mean functions makes them a good
candidate to represent probability measures in infinite-dimensional
spaces.

The representation property of probability measures by kernel
mean functions is addressed in the next theorem and corollary. Proofs are provided in the appendix.
\begin{thm}
	\label{thm:thm1}Let $P_{1}$ and $P_{2}$ be two probability measures
	on a separable Hilbert space $\mathbb{H}$ over the field $\mathbb{R}$. Let
	$\psi:\mathbb{R^{+}}\rightarrow\left[0,1\right]$ be a bounded continuous, strictly decreasing
	and positive definite function e.g. $\psi(t)=e^{-t^{2}}$, such that
	$k(x,y)=\psi(\left\Vert x-y\right\Vert _{\mathbb{H}})$ is a translation-invariant
	characteristic kernel, and let $m_{P_{1}}(\cdot)$ and $m_{P_{2}}(\cdot)$
	be the kernel mean embedding of $P_{1}$ and $P_{2}$, respectively, for the kernel
	$k\left(\cdot,\cdot\right)$. If $m_{P_{2}}(y)>m_{P_{1}}(y)$ for a given $y\in\mathbb{H}$, then there exists an open ball
	$B_{r}(y)$ such that $P_{2}\left(B_{r}(y)\right)>P_{1}\left(B_{r}(y)\right)$,
	and $r$ depends only on difference $m_{P_{2}}(y)-m_{P_{1}}(y)$ and
	the characteristic kernel itself.
\end{thm}

\begin{cor}\label{cor:smallBallProb}
	\label{cor:repProperty}Let $P$ be a probability measure on a separable Hilbert
	space $\mathbb{H}$ over the field $\mathbb{R}$. Let $\psi:\mathbb{R^{+}}\rightarrow\left[0,1\right]$
	be a bounded continuous, strictly decreasing and positive definite function e.g. $\psi(t)=e^{-t^{2}}$,
	such that $k(x,y)=\psi(\left\Vert x-y\right\Vert _{\mathbb{H}})$
	is a translation-invariant characteristic kernel, and let $m_{p}(\cdot)$
	be the kernel mean embedding of $P$ for the kernel $k\left(\cdot,\cdot\right)$.
	If $m_{P}(y_{2})>m_{P}(y_{1})$ for some $y_{1},y_{2}\in\mathbb{H}$,
	then there exist open balls of the same size $B_{r}(y_{1})$ and
	$B_{r}(y_{2})$ such that $P\left(B_{r}(y_{2})\right)>P\left(B_{r}(y_{1})\right)$,
	and $r$ depends only on difference $m_{P}(y_{2})-m_{P}(y_{1})$ and
	the characteristic kernel itself.
\end{cor}
Kernel Mean Embedding of probability measures also has a connection
with kernel scoring rules. Proper Scoring Rules are well-established instruments with 
applications in assessing probability
models \citep{gneiting2007strictly}. The following definition is borrowed from \citet{steinwart2019strictly} and adapted to our context. In the following definition,  $c_{00}$ is the infinite-dimensional
inner product space of sequences vanishing at infinity, which is a dense subspace of $\ell_{2}$.
\begin{defn}
	Let $\mathbb{X}$ be an arbitrary measurable space. Here it may be considered to be either the separable
	Hilbert space $\ell_{2}$ or the separable inner product space $c_{00}$,
	and let $\mathcal{M}_{1}\left(\mathbb{X}\right)$ be the space of
	probability measures on $\mathbb{X}$. For $\mathcal{P}\subseteq\mathcal{M}_{1}\left(\mathbb{X}\right)$,
	a \textit{scoring rule} is defined as a function $S:\mathcal{P}\times\mathbb{X}\to\left[-\infty,\infty\right]$
	such that the integral $\int_{\mathbb{X}}S\left(P,x\right)Q\left(dx\right)$
	exists for all $P,Q\in\mathcal{P}$. The scoring rule is \textit{proper}
	if 
	\[
	\int_{\mathbb{X}}S\left(P,x\right)P\left(dx\right)\leq\int_{\mathbb{X}}S\left(Q,dx\right)P\left(dx\right),\quad\forall P,Q\in\mathcal{P}
	\]
	and is called \textit{strictly proper} if the  equality implies $P=Q$.
\end{defn}
\textit{Kernel scores} are a general class of proper scoring rules, in
which the scoring rule is generated by a symmetric positive definite
kernel $k:\mathbb{X}\times\mathbb{X}\to\mathbb{R}$ by 
\begin{align}
S_{k}\left(P,x\right) :=&-\int k\left(\omega,x\right)P\left(d\omega\right)+\frac{1}{2}\int\int k\left(\omega,\omega'\right)P\left(d\omega\right)P\left(d\omega'\right)\nonumber\\
=&-m_{p}\left(x\right)+\frac{1}{2}\left\Vert m_{P}\right\Vert ^{2}.\label{formula:properScoringRule}
\end{align}
The Maximum Mean Discrepancy distance between 
$P,Q\in\mathcal{P}$ satisfies 
\begin{equation}\label{formula::MMD1}
\left\Vert m_{P}-m_{Q}\right\Vert _{\mathcal{H}_{k}}^{2}=2\left(\int S_{k}\left(Q,x\right)P\left(dx\right)-\int S_{k}\left(P,x\right)P\left(dx\right)\right).
\end{equation}
If $k$ is bounded then $\mathcal{P}=\mathcal{M}_{1}\left(\mathbb{X}\right)$ \citep{steinwart2019strictly}.
In effect, a kernel score rule $S_{k}$ is a strictly proper scoring
rule if and only if kernel mean embedding is injective or $k$ to
be characteristic. 

There are a plethora  of studies on the different class of characteristic kernels over locally
compact spaces. 
For example, \citet{steinwart2001influence} proved that Gaussian
kernel is characteristic on compact sets,  \citet[Theorem 9]{sriperumbudur2010hilbert}
showed that Gaussian kernel is characteristic on the whole space $\mathbb{R}^{d}$, and 
\citet{Simon-Gabriel2018} studied the connection between various
concepts of kernels such as universality, characteristic and positive
definiteness of kernels. 
%As a general example Gaussian kernel is a prominent kernel which is characteristic for many finite-dimensional spaces.

Given a separable Hilbert space $\mathbb{H}$,
any integrally strictly positive definite kernel is characteristic
\citep[Theorem 7]{sriperumbudur2010hilbert}, however, it is not clear which kernels are integrally strictly positive definite over infinite-dimensional separable Hilbert spaces. To the best of our knowledge, there is no study on 
existence and construction of characteristic kernels for infinite-dimensional spaces.
The following two theorems, proofs of which are provided in Appendix, try to tackle this problem. In
Theorem \ref{thm:existanceOfCharKer_l2}, the result of \citet[Threorem 3.14]{steinwart2019strictly} is used 
to show the existence of a continuous characteristic kernel for infinite-dimensional
separable Hilbert spaces, and Theorem \ref{thm:GausCahr_c00} shows
that Gaussian kernel is characteristic for $c_{00}$, the infinite-dimensional
inner product space of sequences vanishing at infinity, which
is dense in $\ell_{2}$.
\begin{thm}
	\label{thm:existanceOfCharKer_l2}Let $\mathbb{H}$ be an infinite-dimensional
	separable Hilbert space. There exists a continuous characteristic
	kernel on $\mathbb{H}$.
\end{thm}

\begin{thm}
	\emph{\label{thm:GausCahr_c00}Let $c_{00}$ be the space of eventually
		zero sequences in $\mathbb{R}^{\infty}$. The Gaussian kernel
		defined as $k\left(x,y\right)=e^{-\sigma\left\Vert x-y\right\Vert _{2}^{2}}$
		is characteristic on $c_{00}$.}
\end{thm}
Beside what are presented in Theorem \ref{thm:existanceOfCharKer_l2} and Theorem \ref{thm:GausCahr_c00}, we show in Proposition \ref{prop:GausKMF} that Gaussian kernel is characteristic for the family of Gaussian probability measures over $\mathbb{H}$.% $\ell_{2}$.
\section{Maximum Kernel Mean Estimation}\label{sec3}
In the context of multivariate statistics, the density function is
considered as one of the most ubiquitous tools in statistical inference.
Density is a non-negative function, which represents the amount of probability
mass in a point or concentration of probability measure in a very small neighborhood.
% of the point relative to a Haar base measure.
Typically a nominated family of probability
measures is presented by the corresponding family of densities, and the aim is to choose a density from this family, which is the most likely one that generates a set of observations obtained by a probability-based survey sample. The aforementioned family of probability measures usually parameterized
by a parameter $\theta$ taking value either in a subset $\Theta$ of a finite-dimensional or infinite-dimensional
space.

Suppose that  $\left\{P_{\theta},~\theta\in\Theta\right\}$ is a nominated family of probability
measures indexed by $\theta$. The idea behind MLE is as follows: Suppose
we randomly survey the population according to a sampling method and the result is an observation
$y$. If $\theta$ is unknown, an estimation of $\theta$ is one for
which $P_{\theta}$ is the most likely generator of $y$. If the density function $f_{\theta}=dP_{\theta}/d\lambda$
exists, we seek for a $\theta$ for which $f_{\theta}(x)$ is of maximum value.

What makes a density function suitable for this kind of inference is
the base measure $\lambda$ where the density is defined relative
to it. A counting measure or Lebesgue measure are suitable options
in finite-dimensional spaces. These base measures are positive, locally
finite and translation invariant and a nontrivial measure
with these properties does not exist in an infinite-dimensional separable
Hilbert space \citep{Gill2014}.

Employing the kernel mean function, we can introduce a rather similar idea to likelihood-based estimation in infinite-dimensional spaces. Suppose $k$ is a bounded, continuous, and 
translation-invariant characteristic kernel as described in Theorem
\ref{thm:thm1}, such as Gaussian kernel $k(x,y)=e^{-\sigma\parallel x-y\parallel_{\mathbb{H}}^{2}}$
for a fixed $\sigma>0$. The kernel mean function $m_{P}(\cdot)$ is a bounded function over $\mathbb{H}$ which reflects the concentration of $P$ on a small neighborhood of $y\in\mathbb{H}$. Consider
the family of probability measures $M=\{P_{\theta},~\theta\in\Theta\}$
and its counterpart family of kernel mean functions $\{m_{P_{\theta}}(.);\:\theta\in\Theta\}$. Assume that  
$\theta$ is unknown and the endeavor is to estimate it through an observed
random sample $y$ from the population. Again the aim is to pick a $\widehat{\theta}\in\Theta$
such that $P_{\hat{\theta}}$ is the most likely generator of 
$y$. Thus, it seems possible to estimate $\theta$ by $\widehat{\theta}=\sup_{\theta\in\Theta}m_{P_{\theta}}(y)$.

In this section, we derive the kernel mean embedding of probability measures induced by Functional response models, and we show how  parameter estimation and hypothesis testing are capable in this framework.

\subsection{Kernel Mean Embedding of Gaussian Probability Measure}

The assumption of Gaussianity is prevalent and fundamental to many statistical problems in the context of functional data analysis, including functional response regressions, functional one-way ANOVA, and testing for homogeneity of covariance operators. In this regard, it is desirable to study the kernel mean embedding of Gaussian probability measures induced by these category of models.

Let $\mathbb{H}$ be an arbitrary infinite-dimensional separable Hilbert Space. An $\mathbb{H}$-valued
random element is said to be a Gaussian random element with the mean function
$\mu$ and covariance operator $C$, for any $a\in\mathbb{H}$,
we have $\left\langle a,X\right\rangle \sim N\left(\left\langle a,\mu\right\rangle ,\left\langle Ca,a\right\rangle \right)$.
A Gaussian random element has a finite second moment, and its covariance
operator is trace class \citep{maniglia2004gaussian}. Let $\left(\lambda_{i},\psi_{i}\right)_{i\geq1}$
be the eigensystem of $C$ and $\mathbb{H}$ be an arbitrary function space such as $L^{2}[0,1]$, then
the kernel of integral operator $C$ admits the decomposition
$k_{C}(s,t)=\sum_{j\geq1}\lambda_j\psi_{j}(s)\psi_{j}(t)=Cov\left[X(s),X(t)\right]$ \citep{hsing2015theoretical}.
Kernel mean function of the Gaussian family of probability measures with mean $\mu$
and covariance operator $C$ and its uniqueness is given in the following proposition and is proved in the appendix.
\begin{prop}\label{prop:GausKMF}
	Let $Y\sim\mathcal{N}\left(\mu,C\right)$ i.e. $\left\langle a,Y\right\rangle \sim\mathcal{N}\left(\left\langle a,\mu\right\rangle ,\left\langle Ca,a\right\rangle \right)$
	then for a Gaussian kernel, 
	\begin{align*}
	m_{P}(x) & =\int_{\mathbb{H}}\hspace{-0.2em}e^{-\sigma\left\Vert x-y\right\Vert _{\mathbb{H}}^{2}}\mathcal{N}\left(\mu,C\right)\left(dy\right)=\left|I+2\sigma C\right|^{-1/2}e^{-\sigma\left\langle \left(I+2\sigma C\right)^{-1}\left(x-\mu\right),\left(x-\mu\right)\right\rangle },
	\end{align*}
	the kernel mean embedding is injective and 
	\begin{align*}
	\left\Vert m_{P}\right\Vert _{\mathcal{H}_{k}}^{2} & =\left|I+4\sigma C\right|^{-{1}/{2}}.
	\end{align*}
\end{prop}
Consider a random sample $Y_{i}\in L^{2}[0,1],~i=1,\ldots,n$ of independent and identically random elements with distribution $\mathcal{N}\left(\mu,C\right)$ 
%with $k_{C}(s,t)=\sum_{i\geq1}\lambda_{i}\psi_{i}(s)\psi_{i}(t)$ in which $\left(\lambda_{i},\psi_{i}\right)_{i\geq1}$ are the eigenvalues and eigenfunctions of $C$
. By choosing a suitable characteristic kernel for the product space $\left(L^{2}[0,1]\right)^{n}$, kernel mean embedding of the induced probability measure by the random sample $\left\{Y_i\right\}$ i.e. $\otimes_{i=1}^n\mathcal{N}\left(\mu,C\right)$ can be computed. Let $k$ be the Gaussian kernel, then according to the following theorem, which is proved in the appendix, $\prod_{i=1}^{n}k(\cdot,\cdot)$ and $\sum_{i=1}^{n}k(\cdot,\cdot)$
are characteristic kernels for the family of product measures on $\left(L^{2}[0,1]\right)^{n}$.
\begin{prop}
	\label{thm:tensorProdCharKer}Let $k\left(\cdot,\cdot\right)$ be
	a characteristic kernel defined over a separable Hilbert space $\mathbb{H}$,
	then product-kernel $$k_P(\cdot,\cdot):\mathbb{H}^{n}\times\mathbb{H}^{n}\longrightarrow\mathbb{R}\quad\quad\left(\left\{ x_{i}\right\} ,\left\{ y_{i}\right\} \right)\mapsto\prod_{i=1}^{n}k(x_{i},y_{i})$$
	and sum-kernel
	\[
	k_S(\cdot,\cdot):\mathbb{H}^{n}\times\mathbb{H}^{n}\longrightarrow\mathbb{R}\quad\left(\left\{ x_{i}\right\} ,\left\{ y_{i}\right\} \right)\mapsto\sum_{i=1}^{n}k(x_{i},y_{i})
	\]
	are two characteristic kernels for the family of product probability measures
	$\mathscr{P}^{n}=\left\{ \otimes_{j=1}^{n}P_{j}\mid P_{j}\in\mathscr{P},~j=1,\ldots,n\right\} $
	on $\mathbb{H}^{n}$.
\end{prop}
For the case of Gaussian product-kernel, given a
simple random sample $Y_{1},\ldots,Y_{n}$ drawn from the Gaussian distribution, 
kernel mean function is $$m_{\otimes_{i=1}^{n}\mathcal{N}\left(\mu,C\right)}(y_{1},\ldots,y_{n})=\int e^{{-\sigma{\sum_{i=1}^{n}}\parallel y_{i}-z_{i}\parallel}^2}\otimes_{i=1}^{n}\mathcal{N}\left(\mu,C\right)\left(dz_{i}\right),$$
which its logarithm equals to 
\begin{align}
\log~& m_{\otimes_{i=1}^{n}\mathcal{N}\left(\mu,C\right)}(y_{1},\ldots,y_{n})\nonumber \\ &=\sum_{i=1}^{n}\sum_{j\geq1}\frac{-\sigma}{1+2\sigma\lambda_{j}}\langle y_{i}-\mu,\psi_{j}\rangle^{2}-\frac{n}{2}\sum_{j\geq1}\log\left(1+2\sigma\lambda_{j}\right),\label{eq:modelReg}
\end{align}
where  $y_{1},\ldots,y_{n}$ are the observation counterparts of $Y_{1},\ldots,Y_{n}$. By defining
sample mean function and sample covariance operator as $\bar{y}=\frac{1}{n}\sum_{i=1}^{n}y_{i}$
and $\hat{C}_{Y}=\frac{1}{n}\sum_{i=1}^{n}(y_{i}-\bar{y})\otimes(y_{i}-\bar{y})$ respectively, 
then, the logarithm of the kernel mean function for Gaussian product-kernel also is equal to 
\begin{align}
\log ~& m_{\otimes_{i=1}^{n}\mathcal{N}\left(\mu,C_{Y}\right)}(y_{1},\ldots,y_{n})\nonumber\\
&=  \sum_{j\geq1}\frac{-n\sigma}{1+2\sigma\lambda_{j}}\left[\langle\hat{C}_{Y}\psi_{j},\psi_{j}\rangle+\langle\bar{y}-\mu,\psi_{j}\rangle^{2}\right]-\frac{n}{2}\sum_{j\geq1}\log\left(1+2\sigma\lambda_{j}\right).\label{eq:sufficientGaussian}
\end{align}
We can see that the kernel mean function is dependent on  $\{y_{1},y_{2},\ldots,y_{n}\}$ only through $\bar{y}$ and $\hat{C}_{Y}$. Since Gaussian-product Kernel is characteristic, Equation (\ref{eq:sufficientGaussian})
shows that for the family of Gaussian probability measures, $\left(\bar{y},\hat{C}_{Y}\right)$
is a typical joint sufficient statistic for parameters $\left(\mu,C_{Y}\right)$.
The possibility to identify sufficient statistics
through kernel mean functions, alongside Theorem 1 and Corollary 2, reveals how a kernel mean embedding of probability measure behave akin to density function over finite-dimensional spaces.

The location and covariance parameters of the distribution can
be estimated by maximizing $\log m_{\otimes_{i=1}^{n}\mathcal{N}\left(\mu,C\right)}(y_{1},\ldots,y_{n})$. The resulting estimator which  is slightly different in weights of components from the estimators
%which is slightly different in weights of components with the limiting
%estimator of parameters of $\mu_{i}$
one may obtain by the small
ball probability approximation proposed by \citet{delaigle2010defining}, or OLS approach. As it is highlighted in Proposition \ref{rem::deligel_limitingEstimator}, the Maximum Kernel Mean (MKM) estimator of the location parameters 
converges to OLS, and the limiting estimator one may obtain by the small-ball
probability approach, as $\sigma$ tends to infinity. It is also worth noting that although there is no estimation for covariance parameters by the small-ball probability approximation approach, there is an estimation for them by MKM. 
%There are also other advantages of MKM over small-ball probability approximation. The method proposed by \citet{delaigle2010defining} approximates the concentration of probability in a neighborhood of $y$ by a finite-dimensional approximation of the random element. The precision of this approximation depends on both the size of the neighborhood and also the probability measure itself. Thus, the approximated density is not comparable among different probability measures, though it is possible to compare the concentration of different probability measures by kernel mean function.

%Theorem \ref{thm:thm1} and Corollary \ref{cor:smallBallProb} reveal how  a kernel mean function is akin to a density function. The following proposition shows another aspect of this similarity and shows that it is possible to identify sufficient statistics 
%in the space of infinite-dimensional separable Hilbert spaces through characteristic kernels and kernel mean functions. 

%This way, it could be seen how a kernel mean embedding of probability measure behave akin to density function over finite-dimensional spaces.

In the the context of functional regression, as it is addressed in Section \ref{sec:functionalReg}, we may substitute a linear
model for $\mu$ in (\ref{eq:modelReg}), and estimate the parameters of the model either 
%by maximizing $\log m_{\otimes P(dY_{i})}(y_{1},\ldots,y_{n})$ or
$m_P\left(y_{1},\ldots,y_{n}\right)-\frac{1}{2}\left\Vert m_{P}\right\Vert_{\mathcal{H}_{k}}^{2}$ as in (\ref{formula:properScoringRule}), or only choose to maximize $\log m_{\otimes P(dY_{i})}(y_{1},\ldots,y_{n})$ for estimating the location parameters seeing as $\left\Vert m_{P}\right\Vert_{\mathcal{H}_{k}}^{2}$ does not depend on the location parameters.

The Kernel Mean approach also provides
a rich toolbox of kernel methods developed by the machine learning community, which can be used in statistical inference. To give just a few examples,
we can name Kernel Bayes Rule for Bayesian inference and Latent variable
modeling. The Maximum Mean Discrepancy (MMD) for hypothesis testing and
developing Goodness of Fit indices, and Hilbert Schmidt Independence
Criterion (HSIC) for measuring dependency between random elements \citep[see][]{muandet2017kernel,gretton2012kernel,Harchaoui2009ChangePoint,Tang2017Bernoulli}.
In section \ref{sec::aplications}, MMD is used to derive and introduce new tests for three main problems in functional data analysis, including Function-on-Scalar regression, one-way ANOVA, and testing for homogeneity of covariance operators. The power of these tests is studied and compared with competitors by simulation.
\subsection{MKM Estimation of Parameters in Function-on-Scalar Regression}\label{sec:functionalReg}
Let $Y$ be a Gaussian random element taking value in $\mathbb{H}$. Given a random sample of $Y$, we can employ the kernel mean function to %develop a Pseudo-likelihood to 
estimate the location and covariance parameters. Let $y_{1},y_{2},\ldots,y_{n}$ be $n$ independent realizations of $Y$ %with the common mean function $\mu$ and common covariance operator $C$.
% {\color{red}we maximize either $m_P\left(y_{1},\ldots,y_{n}\right)$ or minimizing strictly proper scoring rule in (\ref{formula:properScoringRule}), that is maximizing $m_P\left(y_{1},\ldots,y_{n}\right)-\frac{1}{2}\left\Vert m_{P}\right\Vert ^{2}$ where $P$ is the probability measure induced by $n$ independent copies of $Y$. Since $\left\Vert m_{P}\right\Vert ^{2}$ depends only on the eigenvalues of covariance operator, given the covariance operator, it is enough to maximize $m_P\left(\cdot\right)$ to estimate the location parameters.}  Consider 
according to the following Function-on-Scalar regression model: 
\begin{equation}\label{model:f-o-s1}
Y_i\left(t\right)=x_i^T\beta\left(t\right)+\varepsilon_i\left(t\right),~~~i=1,\ldots,n
\end{equation}
where $x_i$ is the vector of scalar covariates and $\beta$ is the vector of $p$ functional parameters. Residual functions $\varepsilon_i$ are $n$ independent copies of a Gaussian random element with mean function zero and covariance operator $C$. The following two propositions can be employed to obtain the MKM estimation of location and covariance parameters. 

\begin{prop}\label{prop:RMKM_loc}
	Let $y_{1},y_{2},\ldots,y_{n}$ be $n$ independent realizations of model (\ref{model:f-o-s1}), where $\varepsilon_i$ is an $\mathbb{H}$-valued Gaussian random element with mean function zero and covariance operator $C$. The MKM estimation of functional regression parameters coincide with the ordinary least square estimation.
\end{prop}
\begin{proof}
	The logarithm of kernel mean function by (\ref{eq:modelReg}) equals to
	\begin{align}
	\log m_{\beta}(y_{1},\ldots,y_{n}):&=\log m_{\otimes_{i=1}^{n}\mathcal{N}\left(\mu_{i},C\right)}(y_{1},\ldots,y_{n})\nonumber\\
	&=-\sigma\sum_{i=1}^{n}\left\langle \left(I+2\sigma C\right)^{-1}\left(y_{i}-x_{i}^{T}\beta\right),y_{i}-x_{i}^{T}\beta\right\rangle\nonumber\\
	&\phantom{:=}-\frac{n}{2}\sum_{j\geq1}\log\left(1+2\sigma\lambda_{j}\right).\label{eq:fos1}
	\end{align}
	Fr\'echet derivation of (\ref{eq:fos1}) with respect to $\beta$
	is an operator from $\mathbb{H}^{p}$ to $\mathbb{R}$, i.e.
	\[
	\frac{\partial}{\partial\beta}\log m_{\beta}(y_{1},\ldots,y_{n}):\mathbb{H}^{p}\rightarrow\mathbb{R}.
	\]
	$\hat{\beta}$
	is a local extremum of $\log m_{\otimes_{i=1}^{n}\mathcal{N}\left(\mu,C\right)}(y_{1},\ldots,y_{n})$,
	if 
	\[
	\left[\frac{\partial}{\partial\beta}\log m_{\hat{\beta}}(y_{1},\ldots,y_{n})\right]\left(h\right)=0\qquad\forall h\in\mathbb{H}^{p}.
	\]
	Taking Fr\'echet derivation of (\ref{eq:fos1}) with respect to $\beta$, for an  arbitrary $h\in\mathbb{H}^{p}$ we have
	\begin{align*}
	\left[\frac{\partial}{\partial\beta}\log m_{\beta}(y_{1},\ldots,y_{n})\right]\hspace{-0.2em}\left(h\right) & =\left[\frac{\partial}{\partial\beta}\sum_{i=1}^{n}\left\langle \left(I+2\sigma C\right)^{-1}\hspace{-0.2em}\left(y_{i}-x_{i}^{T}\beta\right),y_{i}-x_{i}^{T}\beta\right\rangle\right]\hspace{-0.2em}\left(h\right)\\
	& =2\sigma\sum_{i=1}^{n}\left\langle \left(I+2\sigma C\right)^{-1}x_{i}^{T}h,y_{i}-x_{i}^{T}\beta\right\rangle \\
	& =2\sigma\sum_{k=1}^{p}\left\langle h_{k},\left(I+2\sigma C\right)^{-1}\sum_{i=1}^{n}x_{ik}\left(y_{i}-x_{i}^{T}\beta\right)\right\rangle 
	\end{align*}
	So if $\hat{\beta}$ is a local extremum of (\ref{eq:fos1}), for each
	$1\leq k\leq p$, we must have $$\sum_{i=1}^{n}x_{ik}\left(y_{i}-x_{i}^{T}\beta\right)=0,$$
	so $X^{T}\left(Y-X\beta\right)=0$, and consequently $\hat{\beta}=\left(X^{T}X\right)^{-1}X^{T}Y$.
	The remaining question arises here is that if $\hat{\beta}$ maximizes
	(\ref{eq:fos1}) or not. Let $\beta=\hat{\beta}+\nu$, then
	\begin{align*}
	\log m_{\beta}(y_{1},\ldots,y_{n})&=\log m_{\hat{\beta}}(y_{1},\ldots,y_{n})-\sigma\sum_{i=1}^{n}\left\langle \left(I+2\sigma C\right)^{-1}\left(x_{i}^{T}\nu\right),x_{i}^{T}\nu\right\rangle\\ &\leq\log m_{\hat{\beta}}(y_{1},\ldots,y_{n}),
	\end{align*}
	which completes the proof.
\end{proof}
From the last proposition, $\hat{\beta}=\left(X^{T}X\right)^{-1}X^{T}Y$ is the MKM estimator
of functional regression coefficients. It is also possible to derive a restricted MKM estimation of covariance operator with a similar approach to restricted ML. Let $A=\left[u_1,\ldots,u_{n-k}\right]$ be the first $n-k$ eigenvectors of $I-X\left(X^TX\right)^{-1}X^T$. Let $Y=\left[Y_i\right]_{i=1,\ldots,n}$ be a $n\times 1$ matrix, then $Y_i^*=u_i^TY$ is called the error contrast vector and is a sequence of $n-k$ independent and identically distributed random elements with mean function zero and common covariance operator $C$. We can then use the sequence $y_1^*,\ldots,y_{n-k}^*$ and employ Proposition \ref{prop:RMKM_cov} to estimate the covariance operator by $\hat{C}=\frac{1}{n-k}\sum_{i=1}^{n-k}y_i^*\otimes y_i^*$. 
\begin{prop}\label{prop:RMKM_cov}
	Let $y_{1},y_{2},\ldots,y_{n}$ be $n$ independent realizations of  $\mathbb{H}$-valued Gaussian random element with mean function zero, and covariance operator $C=\sum_{j\geq1} \lambda_j\psi_j\otimes\psi_j$.
	Let $\hat{C}=\frac{1}{n}\sum_{i=1}^{n} y_{i}\otimes y_{i}$, 
	then as $\sigma\to\infty$ the MKM estimator of $\left\{ \lambda_{j},\psi_{j}\right\} _{j\geq1}$
	converges to 
	\begin{enumerate}
		\item $\hat{\psi}_{k}=$ The $k$'th eigenfunction of $\hat{C}$
		\item $\hat{\lambda}_{k}=\frac{1}{n}\sum_{i=1}^{n}\left\langle y_{i},\hat{\psi}_{k}\right\rangle ^{2}$
	\end{enumerate}
\end{prop}

\begin{proof}
	The logarithm of the kernel mean function of the product measure $\otimes_{i=1}^{n}\mathcal{N}\left(\mu,C\right)$
	is presented in (\ref{eq:modelReg}), while we set $\mu=0$. 
	Parameter estimation is obtained by taking Fr\'echet derivation of kernel mean function with respect
	to $\psi_{k}$ and usual derivation of kernel mean function with respect to $\lambda_k$. In each case, it is shown that
	the local extremum is the global maximum of kernel
	mean function.\\
	\newline
	1) $\psi_{k}$: First we obtain the estimation of $\psi_{1}$. Taking
	Fr\'echet derivation of kernel mean function with respect to 
	$\psi_{1}$, we have
	\begin{align*}
	\left[\frac{\partial}{\partial\psi_{1}}\log m_{\otimes\mathcal{N}}(y_{1},\ldots,y_{n})\right]\left(h\right) & =\left[\frac{\partial}{\partial\psi_{1}}\sum_{i=1}^{n}\sum_{j\geq1}\frac{-\sigma}{1+2\sigma\lambda_{j}}\left\langle y_{i},\psi_{j}\right\rangle ^{2}\right]\left(h\right)\\
	& =\frac{-2\sigma}{1+2\sigma\lambda_{1}}\sum_{i=1}^{n}\left\langle y_{i},\psi_{1}\right\rangle \left\langle y_{i},h\right\rangle \\
	& =\frac{-2\sigma n}{1+2\sigma\lambda_{1}}\left\langle \hat{C}\psi_{1},h\right\rangle. 
	\end{align*}
	Consider that $\psi_{1}$ lies in a sphere of radius 1, thus $\tilde{\psi}_{1}$
	is an extremum point of $\log m_{\otimes\mathcal{N}}(y_{1},\ldots,y_{n})$,
	if $\left\langle \hat{C}\psi_{1},h\right\rangle=0$ for any arbitrary
	$h$ in the tangent space of unit sphere at point $\tilde{\psi}_{1}$,
	i.e. 
	\begin{equation}
	\forall h\in\{\tilde{\psi}_{1}\}^{\perp}\quad\Rightarrow\qquad\left\langle \hat{C}\tilde{\psi}_{1},h\right\rangle =0.\label{eq:Gaus_Psi_est}
	\end{equation}
	In addition, for the case of identifiability $\tilde{\psi}_{1}$ must associates to the largest eigenvalue of $\hat{C}$
	%, i.e. $\tilde{\lambda}_{1}=\frac{1}{n}\sum_{i=1}^{n}\left\langle y_{i},\tilde{\psi}_{1}\right\rangle ^{2}$
	. This way, MKM estimation of $\psi_{1}$ is the solution to the following optimization problem:
	\begin{equation}
	\hat{\psi}_{1}=\underset{\tilde{\psi}\in\mathbb{H}}{\argmax}\frac{1}{n}\sum_{i=1}^{n}\left\langle y_{i},\tilde{\psi}\right\rangle ^{2}\qquad s.t.\quad\left\langle \hat{C}\tilde{\psi}_{k},h\right\rangle =0\quad\forall h\in\{\tilde{\psi}\}^{\perp},\label{eq:Gaus_Psi_est_prog}
	\end{equation}
	which immediately follows that MKM estimation of $\psi_{1}$ is the
	first eigenfunction of $\hat{C}$, and is independent of kernel parameter $\sigma$. The remaining question to answer is
	that if $\hat{\psi}_{1}$ maximizes (\ref{eq:modelReg}) or not.
	Consider that for any arbitrary $h\in\left\{ \hat{\psi}_{1}\right\} ^{\perp}$
	and $\tilde{\psi}_{1}=\frac{\hat{\psi}_{1}+h}{\left\Vert \hat{\psi}_{1}+h\right\Vert }$,
	%	\begin{align*}
	%	\log~&m_{\tilde{\psi}_{1}}(y_{1},\ldots,y_{n})\\
	%	 &=\log m_{\hat{\psi}_{1}}(y_{1},\ldots,y_{n})-\frac{n\sigma}{\left\Vert \hat{\psi}_{1}+h\right\Vert ^{2}\left(1+2\sigma\lambda_{1}\right)}\left\langle \hat{C}h,h\right\rangle -\frac{2n\sigma}{\left\Vert \hat{\psi}_{1}+h\right\Vert ^{2}\left(1+2\sigma\lambda_{1}\right)}\left\langle \hat{C}\hat{\psi}_{1},h\right\rangle \\
	%	& \leq\log m_{\hat{\psi}_{1}}(y_{1},\ldots,y_{n}).
	%	\end{align*}
	\begin{align*}
	\log m_{\tilde{\psi}_{1}}(y_{1},\ldots,y_{n}) &=\log m_{\hat{\psi}_{1}}(y_{1},\ldots,y_{n})-\frac{n\sigma}{\left\Vert \hat{\psi}_{1}+h\right\Vert ^{2}\left(1+2\sigma\lambda_{1}\right)}\left\langle \hat{C}h,h\right\rangle  \\
	& \leq\log m_{\hat{\psi}_{1}}(y_{1},\ldots,y_{n}).
	\end{align*}
	So $\hat{\psi}_{1}$ is the MKM estimation of $\psi_{1}$. For the
	MKM estimation of $\psi_{k}$, $k\geq2$, the following
	constraint should be added to the optimization problem (\ref{eq:Gaus_Psi_est_prog})
	\[
	\left\langle \hat{\psi}_{k},\hat{\psi}_{j}\right\rangle =0\qquad\forall1\leq j<k,
	\]
	which shows that the MKM estimation of all eigenfunctions is the same
	as the set of eigenfunctions of $\hat{C}$.  
	\newline\newline
	2) $\lambda_{k}$: Taking derivation of kernel mean function with
	respect to $\lambda_{k}$, yields
	\begin{align}
	\frac{\partial}{\partial\lambda_{k}}\log m_{\otimes\mathcal{N}}(y_{1},\ldots,y_{n}) & =\frac{\partial}{\partial\lambda_{k}}\hspace{-.2em}\left[\sum_{i=1}^{n}\sum_{j\geq1}\hspace{-0.2em}\frac{-\sigma}{1\hspace{-.2em}+\hspace{-.2em}2\sigma\lambda_{j}}\hspace{-0.2em}\left\langle y_{i},\psi_{j}\right\rangle ^{2}\hspace{-.2em}-\hspace{-.2em}\frac{n}{2}\sum_{j\geq1}\log\left(1\hspace{-.2em}+\hspace{-.2em}2\sigma\lambda_{j}\right)\right]\nonumber\\
	& =\sum_{i=1}^{n}\frac{2\sigma^{2}}{\left(1+2\sigma\lambda_{k}\right)^{2}}\left\langle y_{i},\psi_{k}\right\rangle ^{2}-\frac{n}{2}\frac{2\sigma}{1+2\sigma\lambda_{k}}\nonumber\\
	& =\frac{\sigma}{\left(1+2\sigma\lambda_{k}\right)^{2}}\left[2\sigma\sum_{i=1}^{n}\left\langle y_{i},\psi_{k}\right\rangle ^{2}-n\left(1+2\sigma\lambda_{k}\right)\right]\label{eq124}.
	\end{align}
	Equating (\ref{eq124}) to zero and given $\psi_k$, the value of $\lambda_k$ which maximizes (\ref{eq:modelReg}) while we set $\mu=0$ is given by  $\hat{\lambda}_{k}=\frac{1}{n}\sum_{i=1}^{n}\left\langle y_{i},\psi_{k}\right\rangle ^{2}-\frac{1}{2\sigma}$, in consequence by putting $\psi_k$ to be the $k$'th eigenfunction of $\hat{C}$, we obtain $\hat{\lambda}_{k}=\frac{1}{n}\sum_{i=1}^{n}\left\langle y_{i},\hat{\psi}_{k}\right\rangle ^{2}-\frac{1}{2\sigma}$,
	which is a biased estimator of $\lambda_{k}$ and converges to $\frac{1}{n}\sum_{i=1}^{n}\left\langle y_{i},\hat{\psi}_{k}\right\rangle ^{2}$
	as $\sigma$ tends to infinity.
\end{proof}
%\begin{rem}\color{red}
%	If we maximize $m_P\left(y_{1},\ldots,y_{n}\right)-\frac{1}{2}\left\Vert m_{P}\right\Vert ^{2}$ instead of $m_P\left(y_{1},\ldots,y_{n}\right)$,  same results  yield as those presented in Proposition \ref{prop:RMKM_loc} and \ref{prop:RMKM_cov}.
%\end{rem}
In the following proposition, we obtained an estimation of the functional regression coefficients of model (\ref{model:f-o-s1}) by employing small-ball (SB) probability approximation. 
\begin{prop}\label{rem::deligel_limitingEstimator}
	In the case of Function-on-Scalar regression with the assumption of normality as in the model (\ref{model:f-o-s1}), in estimating the location parameters, the MKM or OLS estimator is the same as the one obtained by the small-ball probability approximation proposed by \citet{delaigle2010defining}.
\end{prop}
\begin{proof}
	Let $Y_{1},\ldots,Y_{n}$ be a simple random
	sample generated by model (\ref{model:f-o-s1}), where $\varepsilon_i$ is a sequence of $n$ independent copies of a $\mathbb{H}$-valued Gaussian random element with mean function zero, and covariance operator $C=\sum_{j\geq1} \lambda_j\psi_j\otimes\psi_j$.
	Let functions $\beta_{k}$ admits the Fourier decomposition
	$\beta_{k}=\sum_{j\geq1}\theta_{kj}\psi_{j}$ and define $\theta_{j}=\left(\theta_{kj}\right)_{k=1,\ldots,p}$, 
	$x_{i}=\left(x_{ik}\right)_{k=1,\ldots,p}$ and $\beta=\left(\beta_{k}\right)_{k=1,\ldots,p}$.
	
	The identity
	$\epsilon_{i}=Y_{i}-x_{i}^{T}\beta\overset{\text{iid}}{\sim}\mathcal{N}(0,C)$
	is equivalent to the situation where component scores $\langle\epsilon,\psi_{j}\rangle/\sqrt{\lambda_{j}}$
	are independent and identically distributed according to the standard
	normal distribution for each $j\in\mathbb{N}$. Fix $r>0$ and let $h=\underset{j}{\text{argmax}}~r^{2}\leq\lambda_{j}$. By method of
	\citet{delaigle2010defining}, the log-density
	%, that is the approximated log-likelihood of small-ball probabilities 
	with radius	$r$ equals to 
	\begin{equation}\label{eq:SB-objfun}
	\ln P(\epsilon_{i}|r)\propto\sum_{j=1}^{h}\ln f_{j}(\sqrt{\lambda_{j}}\langle\epsilon_{i},\psi_{j}\rangle)\propto-\sum_{j=1}^{h}\frac{\lambda_{j}}{2}\left(\langle Y_{i},\psi_{j}\rangle-x_{i}^{T}\theta_{j}\right)^{2},
	\end{equation}
	and thus 
	\[
	\sum_{i=1}^{n}\ln P(\epsilon_{i}|r)\propto\hspace{-.2em}-\hspace{-.2em}\sum_{i=1}^{n}\sum_{j=1}^{h}\frac{\lambda_{j}}{2}\left(\langle Y_{i},\psi_{j}\rangle-x_{i}^{T}\theta_{j}\right)^{2}\hspace{-.2em}=\hspace{-.2em}\sum_{j=1}^{h}\lambda_{j}\Big(-\frac{1}{2}\theta_{j}^{T}(X^{T}X)\theta_{j}+B_{j}^{T}\theta_{j}\Big)
	\]
	in which $B_{j}=\sum_{i=1}^{n}\langle Y_{i},\psi_{j}\rangle x_{i}$
	and $X$ is $n\times p$ model matrix and $Y$ is an $n\times1$
	column vector containing functions $Y_{i}(\cdot)$. Estimate of $\theta_{j}$
	can be obtained by solving the equation $\frac{\partial}{\partial\theta_{j}}\sum_{i=1}^{n}\ln P(\epsilon_{i}|r)=0$
	thus 
	\[
	\hat{\theta}_{j}=\left(X^{T}X\right)^{-1}B_{j}.
	\]
	For a given $r>0$ or its coupled quantity $h\in\mathbb{N}$, estimation
	of $\beta$ is 
	\begin{align*}
	\hat{\beta}^{r}=\left(\sum_{j=1}^{h}\hat{\theta}_{kj}\psi_{j}\right)_{k=1,\ldots,p}&=\left(X^{T}X\right)^{-1}\left(\sum_{i=1}^{n}x_{ik}\sum_{j=1}^{h}\langle Y_{i},\psi_{j}\rangle\psi_{j}\right)_{k=1,\ldots,p}\\
	&=\left(X^{T}X\right)^{-1}X^{T}\left(\sum_{j=1}^{h}\langle Y_{i},\psi_{j}\rangle\psi_{j}\right).
	\end{align*}
	Considering the limit of $\hat{\beta}^{r}$ as $r$ tends to zero, the limiting
	estimation ends up with $\lim_{r\to0}\hat{\beta}^{r}=\big(X^{T}X\big)^{-1}X^{T}Y$,
	which is the OLS estimation of $\beta$.
\end{proof}
%Both MKM and SB estimators depend on other parameters, with the limiting estimators coincide, when the $\sigma$ and $h$ tends to infinity. %It can also be notice the difference in the components weights of the MKM objective function (\ref{eq:modelReg}) and SB objective function (\ref{eq:SB-objfun}).
\section{Applications}\label{sec::aplications}
In the context of machine learning, the Maximum Mean Discrepancy (MMD) 
is a useful vehicle for hypothesis testing. If the kernel $k$ is characteristic, then MMD is a metric on the space of probability measures. The induced distance by this metric can be employed to derive different statistical tests that can be hard to handle in the context of Functional data analysis, especially simultaneous hypothesis tests such as one-way ANOVA and testing for equality of covariance operators in more than two groups. To develop new tests, the probability measures induced by null and alternative hypotheses are embedded in a Hilbert space and their distance is computed by the MMD. 
%For example consider the Gaussian kernel as a characteristic kernel over the space of sequences with finitely many non-zero elements or the family of Elliptical probability measures over an arbitrary infinite-dimensional Hilbert space. Let $\mathcal{H}_\sigma$ be the RKHS induced by Gaussian kernel with parameter $\sigma$ and $\text{MMD}_\sigma$ be the corresponding MMD distance. \textcolor{red}{It is known that if $\sigma_2>\sigma_1$ then $\mathcal{H}_{\sigma_2}\supset\mathcal{H}_{\sigma_1}$}, and consequently $\text{MMD}_{\sigma_2}$ will be stronger than $\text{MMD}_{\sigma_1}$.

Kernel-based methods such as kernel mean and covariance embedding of probability measures have a wide range of applications in analyzing structured and non-structured data and also developing non-parametric tests in finite-dimensional spaces such as testing for homogeneity of location and variance parameters, change-point detection, and test of independence. See for example \citet{Harchaoui2009ChangePoint}, \citet{gretton2012kernel} and \citet{Tang2017Bernoulli} among others to get some insight. 
In this section, we employ MMD to develop new tests for three major problems in the context of Functional response models, including Function-on-Scalar regression, Functional one-way ANOVA, and testing for equality of covariance operators. The performance of new tests is compared to some state-of-the-art methods.

Before proceeding, it seems indispensable to notice 
that in the methods developed in this section, we assumed that the sampled random functions are observed completely. However, in practice, a function is observed only in a sparse or dense subset of the domain. Accordingly, at the first stage of analysis, we may operate a smoothing procedure to construct  functions. To scrutinize the effect of smoothing in the results, the first simulation in this section has been done with different number of points sampled per curve. The results seem to be acceptable despite using smoothed functions rather than complete functions.

\subsection{Hypothesis Testing in Function-on-Scalar Regression Model}
Let $\mathbb{H}$ be the space of square-integrable functions $L^2[0,1]$, and consider the following simple Function-on-Scalar regression problem: 
\begin{equation}
y_{i}\left(t\right)=\alpha\left(t\right)+x_{i}\beta\left(t\right)+\varepsilon_{i}\left(t\right),\quad\varepsilon_{i}\stackrel{\text{iid}}{\sim}\mathcal{N}\left(0,C\right);\quad i=1,\ldots,n\quad t\in\left[0,1\right]\label{eq:func-on-scalRegSimple}
\end{equation}
where $\mathcal{N}\left(0,C\right)$ is the Gaussian distribution over the space $L^2[0,1]$, with  mean function zero and covariance operator $C$.
In proposition \ref{prop:RMKM_loc} it is shown that the maximum
kernel mean estimation of intercept and slope functions coincide with
OLS estimation. To assess the uncertainty of estimation and testing
for $H_{0}:\beta=0$, we run a simulation study proposed by \citet{kokoszka2017discussion}
to compare type-I error and power of new test devised from MMD with the current developed
tests. 

Let $\alpha\left(t\right)=2t$ and $\beta\left(t\right)=-c_{0}\cos\left(\pi t\right)$,
in which the parameter $c_{0}$ is used to switch between the null 
and alternative hypotheses. For the covariance operator $C$ in (\ref{eq:func-on-scalRegSimple}) we use
the Mat\'ern family of covariance operators, once with an infinite smoothness
parameter and once with the smoothness parameter set to $\nicefrac{1}{2}$, 
\[
C_{\infty}\left(s,t\right)=\exp\left\{ -\frac{\left|s-t\right|^{2}}{\rho}\right\} \quad\text{and}\quad C_{\nicefrac{1}{2}}\left(s,t\right)=\exp\left\{ -\frac{\left|s-t\right|}{\rho}\right\} .
\]

The kernel function $C_{\infty}$ referred to as squared-exponential
covariance and $C_{\nicefrac{1}{2}}$ referred to as exponential
covariance function. To test $H_{0}:\beta=0$, we devise a new test
using MMD statistic by employing Gaussian kernel as a characteristic
kernel on $\mathbb{H}^{n}$. We use either the Gaussian sum-kernel
\[
k(\cdot,\cdot):\mathbb{H}^{n}\times\mathbb{H}^{n}\longrightarrow\mathbb{R}\quad\left(\left\{ x_{i}\right\} ,\left\{ y_{i}\right\} \right)\mapsto\sum_{i=1}^{n}e^{-\sigma\left\Vert x_{i}-y_{i}\right\Vert _{\mathbb{H}}^{2}}
\]
or Gaussian product-kernel 
\[
k(\cdot,\cdot):\mathbb{H}^{n}\times\mathbb{H}^{n}\longrightarrow\mathbb{R}\quad\left(\left\{ x_{i}\right\} ,\left\{ y_{i}\right\} \right)\mapsto\prod_{i=1}^{n}e^{-\sigma\left\Vert x_{i}-y_{i}\right\Vert _{\mathbb{H}}^{2}}.
\]
Let $P_{0}$ be the induced probability measure of model (\ref{eq:func-on-scalRegSimple})
under the null hypothesis, and $P_{1}$ be the induced probability measure
when the parameter $\beta$ considered to be free. Let $\hat{C}_{0}$
and $\hat{\alpha}_{0}$ be the estimation of covariance and intercept
function under the null hypothesis and $\hat{C}_{1}$, $\hat{\alpha}_{1}$
and $\hat{\beta}_{1}$ be the estimation of covariance, intercept
and slope function under the alternative hypothesis as described in Proposition \ref{prop:RMKM_loc} and Proposition \ref{prop:RMKM_cov}, then plugin estimators
$\hat{P}_{0}$ and $\hat{P}_{1}$ are $\prod_{i=1}^{n}\mathcal{N}\left(\hat{\alpha}_{0},\hat{C}_{0}\right)$
and $\prod_{i=1}^{n}\mathcal{N}\left(\hat{\alpha}_{1}+x_{i}\hat{\beta}_{1},\hat{C}_{1}\right)$
respectively. The MMD statistic can then be defined as the maximum mean discrepancy
distance between $\hat{P}_{0}$ and $\hat{P}_{1}$. Using the Gaussian
Sum-Kernel, MMD statistic equals to:
%\begin{align*}
%\text{MMDS} & =\left\Vert m_{\hat{P}_{0}}-m_{\hat{P}_{1}}\right\Vert _{\mathcal{H}_{k}}=\left[\left\Vert m_{\hat{P}_{0}}\right\Vert _{\mathcal{H}_{k}}^{2}+\left\Vert m_{\hat{P}_{1}}\right\Vert _{\mathcal{H}_{k}}^{2}-2\left\langle m_{\hat{P}_{0}},m_{\hat{P}_{1}}\right\rangle _{\mathcal{H}_{k}}\right]^{\nicefrac{1}{2}}\\
%& =\left[\vphantom{\sum_{i=1}^{n}e^{-\sigma\left\langle \left(I+2\sigma\left(\hat{C}_{0}+\hat{C}_{1}\right)\right)^{-1}\left(\hat{\alpha}_{0}-\hat{\alpha}_{1}-x_{i}\hat{\beta}_{1}\right),\left(\hat{\alpha}_{0}-\hat{\alpha}_{1}-x_{i}\hat{\beta}_{1}\right)\right\rangle }}n\left|I+4\sigma\hat{C}_{0}\right|^{-\nicefrac{1}{2}}+n\left|I+4\sigma\hat{C}_{1}\right|^{-\nicefrac{1}{2}}\right.\\
%& \left.\qquad-2\left|I+2\sigma\left(\hat{C}_{0}+\hat{C}_{1}\right)\right|^{-\nicefrac{1}{2}}\sum_{i=1}^{n}e^{-\sigma\left\langle \left(I+2\sigma\left(\hat{C}_{0}+\hat{C}_{1}\right)\right)^{-1}\left(\hat{\alpha}_{0}-\hat{\alpha}_{1}-x_{i}\hat{\beta}_{1}\right),\left(\hat{\alpha}_{0}-\hat{\alpha}_{1}-x_{i}\hat{\beta}_{1}\right)\right\rangle }\right]^{\nicefrac{1}{2}},
%\end{align*}
\begin{align*}
\text{MMDS} & =\left\Vert m_{\hat{P}_{0}}-m_{\hat{P}_{1}}\right\Vert _{\mathcal{H}_{k}}=\left[\left\Vert m_{\hat{P}_{0}}\right\Vert _{\mathcal{H}_{k}}^{2}+\left\Vert m_{\hat{P}_{1}}\right\Vert _{\mathcal{H}_{k}}^{2}-2\left\langle m_{\hat{P}_{0}},m_{\hat{P}_{1}}\right\rangle _{\mathcal{H}_{k}}\right]^{\nicefrac{1}{2}}\\
& =\left[\vphantom{\sum_{i=1}^{n}e^{-\sigma\left\langle \left(I+2\sigma\left(\hat{C}_{0}+\hat{C}_{1}\right)\right)^{-1}\left(\hat{\alpha}_{0}-\hat{\alpha}_{1}-x_{i}\hat{\beta}_{1}\right),\left(\hat{\alpha}_{0}-\hat{\alpha}_{1}-x_{i}\hat{\beta}_{1}\right)\right\rangle }}n\left|I+4\sigma\hat{C}_{0}\right|^{-\nicefrac{1}{2}}+n\left|I+4\sigma\hat{C}_{1}\right|^{-\nicefrac{1}{2}}-2\left|I+2\sigma\left(\hat{C}_{0}+\hat{C}_{1}\right)\right|^{-\nicefrac{1}{2}}\right.\\
& \left.\qquad\qquad~\times\sum_{i=1}^{n}e^{-\sigma\left\langle \left(I+2\sigma\left(\hat{C}_{0}+\hat{C}_{1}\right)\right)^{-1}\left(\hat{\alpha}_{0}-\hat{\alpha}_{1}-x_{i}\hat{\beta}_{1}\right),\left(\hat{\alpha}_{0}-\hat{\alpha}_{1}-x_{i}\hat{\beta}_{1}\right)\right\rangle }\right]^{\nicefrac{1}{2}},
\end{align*}
and the Gaussian product-kernel yields:
\begin{align*}
\text{MMDP} & =\left\Vert m_{\hat{P}_{0}}-m_{\hat{P}_{1}}\right\Vert _{\mathcal{H}_{k}}=\left[\left\Vert m_{\hat{P}_{0}}\right\Vert _{\mathcal{H}_{k}}^{2}+\left\Vert m_{\hat{P}_{1}}\right\Vert _{\mathcal{H}_{k}}^{2}-2\left\langle m_{\hat{P}_{0}},m_{\hat{P}_{1}}\right\rangle _{\mathcal{H}_{k}}\right]^{\nicefrac{1}{2}}\\
& =\left[\vphantom{e^{-\sigma\sum\limits _{i=1}^{n}\left\langle \left(I+2\sigma\left(\hat{C}_{0}+\hat{C}_{1}\right)\right)^{-1}\left(\hat{\alpha}_{0}-\hat{\alpha}_{1}-x_{i}\hat{\beta}_{1}\right),\left(\hat{\alpha}_{0}-\hat{\alpha}_{1}-x_{i}\hat{\beta}_{1}\right)\right\rangle }}\left|I+4\sigma\hat{C}_{0}\right|^{-\nicefrac{n}{2}}+\left|I+4\sigma\hat{C}_{1}\right|^{-\nicefrac{n}{2}}-2\left|I+2\sigma\left(\hat{C}_{0}+\hat{C}_{1}\right)\right|^{-\nicefrac{n}{2}}\right.\\
& \left.\qquad\qquad\qquad\times e^{-\sigma\sum\limits _{i=1}^{n}\left\langle \left(I+2\sigma\left(\hat{C}_{0}+\hat{C}_{1}\right)\right)^{-1}\left(\hat{\alpha}_{0}-\hat{\alpha}_{1}-x_{i}\hat{\beta}_{1}\right),\left(\hat{\alpha}_{0}-\hat{\alpha}_{1}-x_{i}\hat{\beta}_{1}\right)\right\rangle }\right]^{\nicefrac{1}{2}}.
\end{align*}
Significance level is put at $0.05$ and type-I error and
power of test are compared with the test proposed by \citet{greven2017general} and implemented in the \code{pffr} function of \code{refund}
package. The rate of rejection
for different number of points sampled per curve ($m$) and
two different covariance operators, computed using a Monte Carlo simulation
study, and the results are presented in Tables
\ref{tbl:func_on_scal_sqExp} and  \ref{tbl:func_on_scal_Exp}. 
Note that the distribution of MMD statistic under null hypothesis is approximated using the random permutation method.
It could be realized that type-I error of \code{pffr} is inflated and is higher than the significance level. This problem is much worse with increasing  number of sampling points per curve. \citet{kokoszka2017discussion}
proposed a partial fix to this problem. They suggest to ignore the
uncertainty estimates from \code{pffr}. Instead, they use the estimated residual functions given by \code{pffr} and combine them with a classic estimate of uncertainty. To this end, let $\hat{\beta}_{1}\left(t\right)$
and $\hat{\varepsilon}_{i}\left(t\right)$ be the slope function and
residual functions estimated by \code{pffr} respectively. Then an estimation
of the uncertainty for $\hat{\beta}_{1}\left(t\right)$ is 
\begin{align}
\text{Cov}\left(\hat{\beta}_{1}\left(s\right),\hat{\beta}_{1}\left(t\right)\right) & =\text{Cov}\left(\varepsilon\left(s\right),\varepsilon\left(t\right)\right)\left[X^{T}X\right]_{\left(2,2\right)}^{-1}\nonumber \\
& \approxeq\left(n\sum_{i=1}^{n}\left(x_{i}-\bar{x}\right)^{2}\right)^{-1}\sum_{i=1}^{n}\hat{\varepsilon}_{i}\left(s\right)\hat{\varepsilon}_{i}\left(t\right),\label{eq:covbetaSlope}
\end{align}
in which $X$ is an $n\times2$ data matrix with vector of ones in the
first column and vector of scalar covariate $\left(x_{i}\right)$ in
the second column. We can run a variety of hypothesis tests by plugging $\hat{\beta}_{1}\left(t\right)$ obtained
from \code{pffr} and estimation of uncertainty obtained by
(\ref{eq:covbetaSlope}) into the \code{fregion.test} function in \code{fregion} package.

\begin{table}[tbh]
	\centering{}\caption{Type-I errors and empirical powers for $H_{0}:\beta=0$ with square-exponential
		covariance function.}
	{\footnotesize{
			\begin{tabular}{c|c|r@{\extracolsep{0pt}.}lr@{\extracolsep{0pt}.}lr@{\extracolsep{0pt}.}lr@{\extracolsep{0pt}.}l|r@{\extracolsep{0pt}.}lr@{\extracolsep{0pt}.}lr@{\extracolsep{0pt}.}lr@{\extracolsep{0pt}.}l}
				\hline 
				& n & \multicolumn{8}{c|}{30} & \multicolumn{8}{c}{70}\tabularnewline
				\hline 
				\hline 
				& $c_{0}$ & \multicolumn{2}{c}{\textbf{0}} & 0&2 & 0&4 & 0&6 & \multicolumn{2}{c}{\textbf{0}} & 0&1 & 0&2 & 0&3\tabularnewline
				\hline 
				\multirow{5}{*}{$m=10$} & pffr & \textbf{50}&\textbf{3} & 71&4 & 95&7 & 99&8 & \textbf{51}&\textbf{3} & 63&2 & 89&8 & 99&0\tabularnewline
				& Norm & \textbf{5}&\textbf{2} & 10&7 & 44&3 & 88&2 & \textbf{4}&\textbf{7} & 8&4 & 24&8 & 61&6\tabularnewline
				& Ellipse & \textbf{3}&\textbf{7} & 20&5 & 76&2 & 97&5 & \textbf{2}&\textbf{1} & 11&0 & 51&1 & 86&4\tabularnewline
				& MMDP & \textbf{4}&\textbf{8} & 23&3 & 77&2 & 97&5 & \textbf{4}&\textbf{4} & 16&1 & 59&0 & 89&9\tabularnewline
				& MMDS & \textbf{3}&\textbf{6} & 27&4 & 79&7 & 97&6 & \textbf{4}&\textbf{8} & 20&6 & 61&9 & 90&8\tabularnewline
				\hline 
				\multirow{5}{*}{$m=20$} & pffr & \textbf{69}&\textbf{6} & 86&6 & 98&5 & 99&9 & \textbf{66}&\textbf{7} & 80&6 & 96&8 & 99&7\tabularnewline
				& Norm & \textbf{5}&\textbf{1} & 10&6 & 45&8 & 87&6 & \textbf{2}&\textbf{7} & 8&5 & 23&4 & 60&4\tabularnewline
				& Ellipse & \textbf{4}&\textbf{6} & 25&6 & 73&7 & 98&6 & \textbf{2}&\textbf{5} & 14&3 & 49&1 & 87&0\tabularnewline
				& MMDP & \textbf{5}&\textbf{1} & 23&8 & 74&1 & 98&7 & \textbf{4}&\textbf{5} & 17&7 & 52&8 & 88&2\tabularnewline
				& MMDS & \textbf{5}&\textbf{5} & 26&9 & 75&4 & 98&7 & \textbf{4}&\textbf{8} & 22&2 & 62&5 & 92&0\tabularnewline
				\hline 
				\multirow{5}{*}{$m=50$} & pffr & \textbf{87}&\textbf{0} & 96&6 & \multicolumn{2}{c}{100} & \multicolumn{2}{c|}{100} & \textbf{85}&\textbf{0} & 93&2 & 99&1 & \multicolumn{2}{c}{100}\tabularnewline
				& Norm & \textbf{5}&\textbf{2} & 14&2 & 47&4 & 86&4 & \textbf{7}&\textbf{2} & 8&6 & 26&0 & 62&3\tabularnewline
				& Ellipse & \textbf{5}&\textbf{1} & 31&3 & 81&5 & 98&1 & \textbf{3}&\textbf{4} & 14&5 & 58&2 & 90&6\tabularnewline
				& MMDP & \textbf{4}&\textbf{0} & 25&4 & 76&5 & 97&3 & \textbf{4}&\textbf{5} & 15&9 & 56&5 & 90&5\tabularnewline
				& MMDS & \textbf{4}&\textbf{0} & 28&8 & 79&6 & 98&0 & \textbf{4}&\textbf{8} & 20&4 & 65&2 & 93&3\tabularnewline
				\hline 
			\end{tabular}	
	}}\label{tbl:func_on_scal_sqExp}
\end{table}
\begin{table}[bh!]
	\centering{}
	\caption{Type-I errors and empirical powers for $H_{0}:\beta=0$ with exponential
		covariance function.}
	{\footnotesize{
			\begin{tabular}{c|c|r@{\extracolsep{0pt}.}lr@{\extracolsep{0pt}.}lr@{\extracolsep{0pt}.}lr@{\extracolsep{0pt}.}l|r@{\extracolsep{0pt}.}lr@{\extracolsep{0pt}.}lr@{\extracolsep{0pt}.}lr@{\extracolsep{0pt}.}l}
				\hline 
				& n & \multicolumn{8}{c|}{30} & \multicolumn{8}{c}{70}\tabularnewline
				\hline 
				\hline 
				& $c_{0}$ & \multicolumn{2}{c}{\textbf{0}} & 0&2 & 0&4 & 0&6 & \multicolumn{2}{c}{\textbf{0}} & 0&1 & 0&2 & 0&3\tabularnewline
				\hline 
				\multirow{5}{*}{$m=10$} & pffr & \textbf{43}&\textbf{9} & 68&2 & 95&9 & \multicolumn{2}{c|}{100} & \textbf{43}&\textbf{0} & 57&8 & 85&9 & 98&3\tabularnewline
				& Norm & \textbf{3}&\textbf{0} & 11&9 & 55&5 & 91&1 & \textbf{2}&\textbf{7} & 8&5 & 29&4 & 70&9\tabularnewline
				& Ellipse & \textbf{0}&\textbf{7} & 10&8 & 56&4 & 92&0 & \textbf{0}&\textbf{1} & 1&5 & 18&7 & 60&0\tabularnewline
				& MMDP & \textbf{5}&\textbf{1} & 15&8 & 58&1 & 91&1 & \textbf{4}&\textbf{7} & 12&4 & 41&65 & 78&2\tabularnewline
				& MMDS & \textbf{5}&\textbf{0} & 8&5 & 23&9 & 52&2 & \textbf{4}&\textbf{6} & 5&5 & 9&3 & 15&3\tabularnewline
				\hline 
				\multirow{5}{*}{$m=20$} & pffr & \textbf{67}&\textbf{6} & 86&3 & 98&7 & \multicolumn{2}{c|}{100} & \textbf{63}&\textbf{6} & 79&5 & 96&1 & 99&8\tabularnewline
				& Norm & \textbf{3}&\textbf{3} & 11&9 & 54&2 & 92&2 & \textbf{3}&\textbf{8} & 7&4 & 29&5 & 71&7\tabularnewline
				& Ellipse & \textbf{1}&\textbf{1} & 10&7 & 54&3 & 93&0 & \textbf{0}&\textbf{1} & 0&1 & 6&4 & 40&9\tabularnewline
				& MMDP & \textbf{4}&\textbf{6} & 18&2 & 61&3 & 93&2 & \textbf{5}&\textbf{3} & 12&5 & 43&3 & 82&5\tabularnewline
				& MMDS & \textbf{4}&\textbf{6} & 20&8 & 61&1 & 91&9 & \textbf{5}&\textbf{4} & 17&4 & 54&7 & 87&9\tabularnewline
				\hline 
				\multirow{5}{*}{$m=50$} & pffr & \textbf{86}&\textbf{1} & 95&3 & 99&8 & \multicolumn{2}{c|}{100} & \textbf{87}&\textbf{0} & 93&8 & 99&3 & \multicolumn{2}{c}{100}\tabularnewline
				& Norm & \textbf{4}&\textbf{6} & 12&0 & 54&6 & 91&4 & \textbf{4}&\textbf{6} & 8&2 & 32&0 & 73&5\tabularnewline
				& Ellipse & \textbf{0}&\textbf{5} & 6&8 & 41&5 & 83&9 & \textbf{0}&\textbf{1} & 0&1 & 3&1 & 23&3\tabularnewline
				& MMDP & \textbf{4}&\textbf{4} & 18&7 & 61&6 & 92&6 & \textbf{4}&\textbf{9} & 11&7 & 44&9 & 82&3\tabularnewline
				& MMDS & \textbf{3}&\textbf{8} & 21&2 & 66&2 & 93&5 & \textbf{5}&\textbf{1} & 19&9 & 61&5 & 91&6\tabularnewline
				\hline 
			\end{tabular}
	}}
	\par
	\label{tbl:func_on_scal_Exp}
\end{table}

Here we compare the proposed test with two existing ones; one is based
on the $L^{2}\left[0,1\right]$ norm and the other test based on hyper-ellipsoid confidence regions
proposed by \citet{Choi2018}. The simulation results reported in Tables
\ref{tbl:func_on_scal_sqExp} and  \ref{tbl:func_on_scal_Exp}. It can be noticed 
that MMD tests have type-I errors close to the significance level,
and their powers are superior to the norm and hyper-ellipse tests. In all situations when the number of sampling points per curve is moderate and high ($m=20$ or $m=50$), MMDS outperforms MMDP, though, in the case of exponential covariance function where residual functions are less smooth, MMDP has  higher power than MMDS when the number of sampling points is small ($m=10$). 

Performance of the kernel methods depends on the choice of kernel's parameters. As it is shown in the proof of Theorem \ref{thm:thm1}, the precision of the kernel mean function in representing probability measures depends on the kernel bandwidth. A large enough kernel parameter $\sigma$ works well in our simulation studies.
In this simulation study, the kernel parameter has been set to $\sigma=5e4$ for MMDS and $\sigma=5e1$ for MMDP test. Fourier basis also has been used in the smoothing procedure. The number of components for the smoothing procedure is considered to be fixed and equals 41. The results presented in Tables
\ref{tbl:func_on_scal_sqExp} and  \ref{tbl:func_on_scal_Exp} are reported by 5000 iterations. 

\subsection{Functional One-way ANOVA}

The one-way ANOVA is a fundamental problem in statistical inference. 
%Consider the probability measures induced by this model with functional response. 
Assume that $\mathbb{H}$ is a separable Hilbert space, $Y_{ij}$ for $i=1,\ldots,k$ and $j=1,\ldots,n_{i}$ are independent
random samples taking values from  $\mathbb{H}$, and  $y_{ij}$ are their observations counterparts.
For a typical functional one-way ANOVA problem with the assumption
of homogeneity of covariance operators, the random elements  $Y_{ij}$
are assumed to be generated according to the following model: 
\begin{equation}
Y_{ij}=\mu_{i}+\varepsilon_{ij},\quad\varepsilon_{ij}\stackrel{iid}{\sim}\mathcal{N}\left(0,C\right);\quad i=1,\ldots,k,\;j=1,\ldots,n_{i}\label{eq:onewayANOVAModel}
\end{equation}
where $\mu_{i}$ is the mean function of the $i$'th group and the covariance
operator is equal in all the $k$ groups:  $C=E\left[\varepsilon_{ij}\otimes\varepsilon_{ij}\right]$ for all $i=1,\ldots,k$ and $j=1,\ldots,n_{i}$.
It is of interest to test the equality of $k$ mean functions, i.e.
$H_{0}:\mu_{1}=\cdots=\mu_{k}$. Let $P_{1}$ be the probability
measure of samples generated by model (\ref{eq:onewayANOVAModel})
under $H_{0}$ and $P_{2}$ be
the  probability measure induced by model (\ref{eq:onewayANOVAModel}),
under the alternative hypothesis, that is, when parameters $\mu_{i}$ are considered
to be free. With the assumption of homogeneity of covariance operators,
the squared MMD with Gaussian product-kernel equals to:
\begin{align*}
\text{MMD}^{2}&=\left\Vert m_{P_{1}}-m_{P_{2}}\right\Vert _{\mathcal{H}_{k}}^{2}\\
&=\left|I+4\sigma C\right|^{-\nicefrac{n}{2}}+\left|I+4\sigma C\right|^{-\nicefrac{n}{2}}\\
&~\hphantom{=\left|I+4\sigma C\right|^{-\nicefrac{n}{2}}}-2\left|I+4\sigma C\right|^{\nicefrac{-n}{2}}e^{-\sigma\sum_{i=1}^{k}n_{i}\left\langle \left(I+4\sigma C\right)^{-1}\left(\mu_{i}-\mu\right),\left(\mu_{i}-\mu\right)\right\rangle }\\
&=2\left|I+4\sigma C\right|^{-\nicefrac{n}{2}}\left(1-e^{-\sigma\sum_{i=1}^{k}n_{i}\left\langle \left(I+4\sigma C\right)^{-1}\left(\mu_{i}-\mu\right),\left(\mu_{i}-\mu\right)\right\rangle }\right).
\end{align*}
Covariance operator is assumed to be equal between the groups, so the
new MMD test statistic can be simplified as 
\begin{align}
\text{MMD}_{0}&=\sum_{i=1}^{k}n_{i}\left\langle \left(I+4\sigma C\right)^{-1}\left(\mu_{i}-\mu\right),\left(\mu_{i}-\mu\right)\right\rangle\nonumber\\ &=\sum_{i=1}^{k}n_{i}\left\Vert \left(I+4\sigma C\right)^{-\nicefrac{1}{2}}\left(\mu_{i}-\mu\right)\right\Vert _{\mathbb{H}}^{2}.\label{eq:MMD0}
\end{align}
Accordingly, the new test statistic is the weighted sum of the distance of group mean functions
$\mu_{i}$ from total mean function $\mu$. By plugging the usual estimation
of mean functions (which are also MKM estimations of mean functions)
into (\ref{eq:MMD0}), the new test statistic yields  
\begin{align*}
\hat{\text{MMD}}_{0}&=\sum_{i=1}^{k}n_{i}\left\langle \left(I+4\sigma\hat{C}\right)^{-1}\left(\hat{\mu}_{i}-\hat{\mu}\right),\left(\hat{\mu}_{i}-\hat{\mu}\right)\right\rangle\nonumber\\
&=\sum_{i=1}^{k}n_{i}\left\Vert \left(I+4\sigma\hat{C}\right)^{-\nicefrac{1}{2}}\left(\hat{\mu}_{i}-\hat{\mu}\right)\right\Vert _{\mathbb{H}}^{2}.
\end{align*}
This test statistic is similar to the kernel Fisher discriminant analysis (KFDA) test statistic proposed by \citet{harchaoui2013kernel} for a two-sample kernel-based non-parametric test when the underlying space is finite-dimensional.

Let $\mathbb{H}$ again be the space of square-integrable functions $L^{2}[0,1]$. Motivated by \citet{zhang2019new}, a simulation study was run to
evaluate the performance of the new MMD test against four other competitors developed for the space of square-integrable functions: a
$L^{2}$-norm based test proposed in \citet{zhang2007statistical}, an $F$-type
test proposed by \citet{shen2004f}, a Global Point-wise $F$-test offered in \citet{zhang2014one}
and the $F_{\text{max}}$ test developed and proposed by \citet{zhang2019new}. We used the same setup as \citet{zhang2019new} for data generating procedure in our simulation study. Hence it is assumed that the functional samples in (\ref{eq:onewayANOVAModel})
are generated from the following one-way ANOVA model: 
\begin{equation}
y_{ij}\left(t\right)=\mu_{i}\left(t\right)+\varepsilon_{ij}\left(t\right),\quad\mu_{i}\left(t\right)=\boldsymbol{c}_{i}^{T}\left[1,t,t^{2},t^{3}\right],\quad\varepsilon_{ij}\left(t\right)=\sum_{r\geq1}\sqrt{\lambda_{r}}z_{ijr}\psi_{r}\left(t\right).\label{eq:onwayANOVAgeneration}
\end{equation}
\[
i=1,\ldots,k;\quad j=1,\ldots,n_{i};\quad t\in\left[0,1\right].
\]
While our method works without the need to put any restriction on the
number of components, we follow the same setup as in \citet{zhang2019new}
and assume a finite number of $q$ nonzero eigenvalues in (\ref{eq:onwayANOVAgeneration}).

The parameter $n_{i}$ denotes the size of each group and the set of $\psi_{r}$ is the eigenfunctions. For all $i$,$j$, the design time points are considered
to be balanced and equally spaced, thus all sampled curves are measured in  the common grid of time points $t_{j}=j/\left(T+1\right)$,
$j=1,\ldots,T$. 

The eigenvalues are assumed to follow the pattern
$\lambda_{r}=a\rho^{r-1}$ for fixed $a>0$ and $\rho\in\left(0,1\right)$.
The tuning parameter $\rho$ determines the decay rate of eigenvalues.
For $\rho$ close to zero (resp. close to one), eigenvalues decay fast
(resp. slowly) and residual functions are more (resp. less) smooth. We put
$\boldsymbol{c}_{1}=\left[1,2.3,3.4,1.5\right]^{T}$ and $\boldsymbol{u}=\left[1,2,3,4\right]^{T}/\sqrt{30}$.
The vector $\boldsymbol{c}_{i}=\boldsymbol{c}_{1}+\left(i-1\right)\delta\boldsymbol{u}$
for different values of $\delta$ represents the mean functions of the $k$ groups. The parameter $\delta$ switches between null and alternative hypotheses.

We fix $q=11$, $a=1.5$, $T=80$ where $T$ is the number of time
points where each curve is observed. For the eigenfunctions, we
put $\psi_{1}\left(t\right)=1$, $\psi_{2r}\left(t\right)=\sqrt{2}\sin\left(2\pi rt\right)$
and $\psi_{2r+1}\left(t\right)=\sqrt{2}\cos\left(2\pi rt\right)$
for $r=1,\ldots,q$. Different setups for data generating procedure
is a combination of the following set of parameters:
\begin{itemize}
	\item $z_{ijr}\stackrel{iid}{\sim}N\left(0,1\right)$ or $z_{ijr}\stackrel{iid}{\sim}t_{4}/\sqrt{2}$.
	\item $\rho=0.1, 0.5, 0.9$ for different level of smoothness of residual
	functions.
	\item $\left(n_{i}\right)=\left(20, 30, 30\right)$ for the small sample
	%,	$\left(n_{i}\right)=\left(30, 40, 70\right)$ for the medium sample
	and $\left(n_{i}\right)=\left(70, 80, 100\right)$ for the large sample
	cases.
\end{itemize}
The parameter $\delta$ for each pair of $\rho$ and $n_{i}$ is selected
in a way that the difference between the performance
of five tests can be distinguished. For all the four test statistics, $L^{2}$-norm based,
$F$-type, GPF, and $F_{\max}$, the authors proposed bootstrap methods
to estimate the null distribution of the test statistic. Consult
\citet[Section 4.5.5]{zhang2013analysis} for the implementation
of the $L^{2}$-norm based and $F$-type test,  \citet[Section 2.4]{zhang2014one}
for the implementation of the GPF test, and \citet[Section S.1]{zhang2019new}
for the implementation of the $F_{\text{max}}$ test. 

Mention should be made that 
%It worth noting that 
although in this paper the
kernel mean embedding of probability measures and the MMD
statistic are derived for the family of Gaussian probability distributions, the new MMD test dominates all the four other
tests in all the situations even in non-Gaussian scenarios. 
Simulation results are shown in Tables \ref{tblCovNormal} and
\ref{tblCovT}. 

In our simulation study, we take $\sigma=1e3$ and B-spline basis has been used in smoothing procedure.
%, to have a different basis with the eigenfunctions of covariance operator.
The number of components for the smoothing procedure is considered to be fixed and equals 41.

According to the results, the empirical power of the MMD test
is higher than the other two tests in all situations. The results presented in these tables are produced and reported by 2000 iterations. 

\begin{table}[tbh]
	\begin{centering}
		\label{tbl:ANOVAnormal}\caption{Empirical powers (in percent) of $L^{2}$, $F$, GPF, $F_{\text{max}}$
			and $\text{MMD}_{0}$ for one-way ANOVA problem when $z_{ijr}\stackrel{iid}{\sim}N\left(0,1\right)$.}
		{\footnotesize{
				\begin{tabular}{cc|ccccc|ccccc}
					\hline 
					$\rho$ &  & \multicolumn{5}{c|}{$\left(n_{i}\right)=\left(20,30,30\right)$} & \multicolumn{5}{c}{$\left(n_{i}\right)=\left(70,80,100\right)$}\tabularnewline
					\hline 
					\multirow{6}{*}{0.1} & $\delta$ & \textbf{0} & 0.015 & 0.03 & 0.05 & 0.065 & \textbf{0} & 0.01 & 0.02 & 0.03 & 0.04\tabularnewline
					& $L^{2}$ & \textbf{7.2} & 8.1 & 16.2 & 43.8 & 70.9 & \textbf{4.6} & 9.8 & 24.9 & 55.9 & 86.2\tabularnewline
					& $F$ & \textbf{6.9} & 6.9 & 13.4 & 39.6 & 66.5 & \textbf{3.8} & 9.5 & 24.1 & 54.6 & 85.6\tabularnewline
					& GPF & \textbf{6.3} & 7.7 & 15.9 & 44.1 & 71.2 & \textbf{4.4} & 9.9 & 24.7 & 56.0 & 86.4\tabularnewline
					& $F_{\text{max}}$ & \textbf{5.9} & 15.8 & 56.8 & 95.1 & 100 & \textbf{4.4} & 26.9 & 82.1 & 99.4 & 100\tabularnewline
					& MMD & \textbf{5.8} & 31.8 & 99.9 & 100 & 100 & \textbf{4.2} & 68.7 & 100 & 100 & 100\tabularnewline
					\hline 
					\multirow{6}{*}{0.5} & $\delta$ & \textbf{0} & 0.05 & 0.1 & 0.15 & 0.2 & \textbf{0} & 0.04 & 0.08 & 0.12 & 0.16\tabularnewline
					& $L^{2}$ & \textbf{4.1} & 5.9 & 10.3 & 15.4 & 25.0 & \textbf{5.3} & 8.5 & 18.4 & 35.7 & 62.5\tabularnewline
					& $F$ & \textbf{3.3} & 4.6 & 9.1 & 13.2 & 21.8 & \textbf{4.7} & 7.9 & 17.2 & 34.5 & 60.7\tabularnewline
					& GPF & \textbf{4.3} & 6.0 & 11.2 & 16.3 & 26.9 & \textbf{5.6} & 8.6 & 18.5 & 37.3 & 64.3\tabularnewline
					& $F_{\text{max}}$ & \textbf{3.3} & 6.5 & 17.5 & 35.8 & 64.1 & \textbf{4.7} & 10.8 & 40.7 & 81.0 & 98.3\tabularnewline
					& MMD & \textbf{4.8} & 19.7 & 92.7 & 100 & 100 & \textbf{4.9} & 65.7 & 100 & 100 & 100\tabularnewline
					\hline 
					\multirow{6}{*}{0.9} & $\delta$ & \textbf{0} & 0.15 & 0.3 & 0.45 & 0.6 & \textbf{0} & 0.1 & 0.2 & 0.3 & 0.4\tabularnewline
					& $L^{2}$ & \textbf{4.9} & 5.8 & 12.5 & 25.3 & 48.1 & \textbf{5.1} & 8.3 & 20.5 & 46.9 & 74.4\tabularnewline
					& $F$ & \textbf{3.3} & 4.0 & 10.3 & 20.1 & 41.9 & \textbf{4.8} & 7.5 & 20.0 & 45.0 & 73.4\tabularnewline
					& GPF & \textbf{6.2} & 7.7 & 16.2 & 29.7 & 52.1 & \textbf{5.6} & 9.1 & 21.9 & 49.2 & 76.4\tabularnewline
					& $F_{\text{max}}$ & \textbf{6.5} & 6.5 & 12.3 & 23.5 & 44.3 & \textbf{5.1} & 7.6 & 18.2 & 42.2 & 72.0\tabularnewline
					& MMD & \textbf{5.3} & 9.3 & 31.1 & 80.0 & 100 & \textbf{4.4} & 14.9 & 74.6 & 100 & 100\tabularnewline
					\hline 
		\end{tabular}}}
		\par\end{centering}
\end{table}
\begin{table}[tbh]
	\centering{}\label{tbl:ANOVAt}\caption{Empirical powers (in percent) of $L^{2}$, $F$, GPF, $F_{\text{max}}$
		and $\text{MMD}_0$ for one-way ANOVA problem when $z_{ijr}\stackrel{iid}{\sim}t_{4}/\sqrt{2}$.}
	{\footnotesize{
			\begin{tabular}{cc|ccccc|ccccc}
				\hline 
				$\rho$ &  & \multicolumn{5}{c|}{$\left(n_{i}\right)=\left(20,30,30\right)$} & \multicolumn{5}{c}{$\left(n_{i}\right)=\left(70,80,100\right)$}\tabularnewline
				\hline 
				\multirow{6}{*}{0.1} & $\delta$ & \textbf{0} & 0.015 & 0.03 & 0.05 & 0.065 & \textbf{0} & 0.01 & 0.02 & 0.03 & 0.04\tabularnewline
				& $L^{2}$ & \textbf{7.4} & 7.7 & 17.6 & 42.0 & 69.1 & \textbf{6.1} & 9.3 & 24.6 & 53.9 & 86.6\tabularnewline
				& $F$ & \textbf{5.8} & 6.0 & 14.7 & 35.6 & 64.1 & \textbf{5.8} & 8.6 & 23.2 & 52.5 & 85.7\tabularnewline
				& GPF & \textbf{7.0} & 7.3 & 17.5 & 41.3 & 68.9 & \textbf{6.3} & 9.2 & 24.9 & 53.8 & 86.1\tabularnewline
				& $F_{\text{max}}$ & \textbf{6.3} & 16.9 & 54.8 & 96.9 & 99.9 & \textbf{6.0} & 24.8 & 80.3 & 99.5 & 100\tabularnewline
				& MMD & \textbf{6.3} & 34.2 & 100 & 100 & 100 & \textbf{6.0} & 71.0 & 100 & 100 & 100\tabularnewline
				\hline 
				\multirow{6}{*}{0.5} & $\delta$ & \textbf{0} & 0.05 & 0.1 & 0.15 & 0.2 & \textbf{0} & 0.04 & 0.08 & 0.12 & 0.16\tabularnewline
				& $L^{2}$ & \textbf{4.4} & 6.1 & 9.8 & 16.0 & 25.8 & \textbf{4.9} & 8.0 & 17.4 & 36.7 & 59.6\tabularnewline
				& $F$ & \textbf{4.1} & 4.6 & 8.4 & 13.4 & 22.6 & \textbf{4.6} & 7.6 & 16.3 & 35.6 & 57.9\tabularnewline
				& GPF & \textbf{4.6} & 6.6 & 10.1 & 16.6 & 25.8 & \textbf{5.0} & 8.3 & 18.2 & 37.8 & 62.0\tabularnewline
				& $F_{\text{max}}$ & \textbf{4.2} & 7.0 & 14.6 & 35.9 & 62.6 & \textbf{4.8} & 9.6 & 40.8 & 82.6 & 98.5\tabularnewline
				& MMD & \textbf{5.4} & 23.1 & 94.2 & 100 & 100 & \textbf{4.6} & 68.6 & 100 & 100 & 100\tabularnewline
				\hline 
				\multirow{6}{*}{0.9} & $\delta$ & \textbf{0} & 0.15 & 0.3 & 0.45 & 0.6 & \textbf{0} & 0.1 & 0.2 & 0.4 & 0.4\tabularnewline
				& $L^{2}$ & \textbf{3.9} & 4.5 & 11.7 & 27.7 & 47.2 & \textbf{5.0} & 7.3 & 20.8 & 46.2 & 74.3\tabularnewline
				& $F$ & \textbf{2.8} & 3.2 & 8.4 & 23.4 & 40.9 & \textbf{4.6} & 6.7 & 19.5 & 44.6 & 73.2\tabularnewline
				& GPF & \textbf{5.4} & 5.9 & 15.2 & 31.0 & 52.3 & \textbf{5.2} & 7.9 & 22.9 & 48.1 & 75.8\tabularnewline
				& $F_{\text{max}}$ & \textbf{4.9} & 5.2 & 11.9 & 26.4 & 46.5 & \textbf{5.3} & 7.1 & 18.7 & 42.7 & 74.2\tabularnewline
				& MMD & \textbf{5.0} & 8.9 & 30.9 & 78.7 & 100 & \textbf{5.3} & 13.1 & 75.7 & 100 & 100\tabularnewline
				\hline 
	\end{tabular}}}
\end{table}
\subsection{Testing for Equality of Covariance Operators}\label{sec:homCovTest}
Let $\mathbb{H}$ be a separable Hilbert space. Assume that
$Y_{ij}$ for $i=1,\ldots,k$ and $j=1,\ldots,n_{i}$ are independent
$\mathbb{H}$-valued Gaussian random elements, and $y_{ij}$ are their observation counterparts generated from the following model:
\begin{equation}
Y_{ij}=\mu_{i}+\varepsilon_{ij},\quad\varepsilon_{ij}\stackrel{ind}{\sim}\mathcal{N}\left(0,C_{i}\right);\quad i=1,\ldots,k,\;j=1,\ldots,n_{i},\label{eq:CovModel}
\end{equation}
where $\mu_{i}$ is the unknown mean function of group $i$,
and $\varepsilon_{ij}$ accounts for subject-effect functions with
mean zero and covariance operator $C_{i}=E\left[\varepsilon_{ij}\otimes\varepsilon_{ij}\right]$.
It is of interest to test the equality of $k$ covariance operators,
i.e. $H_{0}:C_{1}=\cdots=C_{k}$. Based on the proof of Proposition \ref{prop:GausKMF}, the squared MMD for comparing two Gaussian probability measures
$\mathcal{N}\left(\mu_{1},C_{1}\right)$ and $\mathcal{N}\left(\mu_{2},C_{2}\right)$
equal to
\begin{align*}
\left\Vert m_{P_{1}}-m_{P_{2}}\right\Vert _{\mathcal{H}_{k}}^{2} & =\left\Vert m_{P_{1}}\right\Vert _{\mathcal{H}_{k}}^{2}+\left\Vert m_{P_{2}}\right\Vert _{\mathcal{H}_{k}}^{2}-2\left\langle m_{P_{1}},m_{P_{2}}\right\rangle _{\mathcal{H}_{k}}\\
& =\left|I+4\sigma C_{1}\right|^{-\nicefrac{1}{2}}+\left|I+4\sigma C_{2}\right|^{-\nicefrac{1}{2}}\\
&\phantom{=}-2\left|I+2\sigma\left(C_{1}+C_{2}\right)\right|^{\nicefrac{-1}{2}}e^{-\sigma\left\langle \left(I+2\sigma\left(C_{1}\hspace{-0.2em}+\hspace{-0.2em}C_{2}\right)\right)^{-1}\left(\mu_{1}-\mu_{2}\right),\left(\mu_{1}-\mu_{2}\right)\right\rangle }.
\end{align*}
To develop the MMD statistic based on a simple random sample from the
model (\ref{eq:CovModel}), first, we have to choose a proper characteristic kernel
for the space $\mathbb{H}^{\sum\limits _{i=1}^{k}n_{i}}$. By Proposition
\ref{prop:GausKMF}, 
Gaussian kernel $k\left(x,y\right)=e^{-\sigma\left\Vert x-y\right\Vert _{\mathbb{H}}^{2}}$
is characteristic for the family of Gaussian probability measures
on $\mathbb{H}$, and from Theorem \ref{thm:tensorProdCharKer}, $\left(\left\{ x_{i}\right\} ,\left\{ y_{i}\right\} \right)\mapsto\sum_{i=1}^{n}k\left(x_{i},y_{i}\right)$
is a characteristic kernel for the family of finite product of Gaussian probability
measures on $\mathbb{H}^{n}$. Let $P_{1}$ be the probability measure
of samples generated by the model (\ref{eq:CovModel}) under $H_{0}$,
i.e. $C_{1}=\ldots=C_{k}=C$ and $P_{2}$ be the probability
measure induced by the model (\ref{eq:CovModel}), when parameters $C_{i}$
considered to be free. For the centered version of the model (\ref{eq:CovModel}),
that is, $\mu_{1}=\cdots=\mu_{k}=0$, the squared MMD with Gaussian Sum-Kernel equals 
\begin{align*}
\text{MMD}^{2}&=\left\Vert m_{P_{1}}-m_{P_{2}}\right\Vert _{\mathcal{H}_{k}}^{2}\\
& =n\left|I+4\sigma C\right|^{-\nicefrac{1}{2}}+\sum_{i=1}^{k}n_{i}\left|I+4\sigma C_{i}\right|^{-\nicefrac{1}{2}}-2\sum_{i=1}^{k}n_{i}\left|I+2\sigma\left(C+C_{i}\right)\right|^{\nicefrac{-1}{2}}\\
& =\sum_{i=1}^{k}n_{i}\left(\left|I+4\sigma C\right|^{-\nicefrac{1}{2}}+\left|I+4\sigma C_{i}\right|^{-\nicefrac{1}{2}}-2\left|I+2\sigma\left(C+C_{i}\right)\right|^{\nicefrac{-1}{2}}\right).
\end{align*}
There currently developed tests for the $k$-sample equality of covariance functions
problem, if $\mathbb{H}$ considered being the space of square-integrable
functions over a compact set like $L^{2}\left[0,1\right]$. We address two recent successfully developed tests for homogeneity of covariance functions, namely quasi-GPF
and quasi-$F_{max}$ introduced by \citet{guo2019new}. There
are other tests, which are shown to be of less powerful in different settings
relative to the currently mentioned tests. See \citet{guo2019new}
and references therein for more information and simulation studies. 
%Let $\gamma_{i}\left(x,t\right)$ be
%the kernel of covariance operators $C_{i}$, then quasi-GPF and quasi-$F_{max}$
%is defined as a quasi point-wise F statistic by 
%\[
%\text{GPF}=\int\limits _{0}^{1}\int\limits _{0}^{1}F_{n}\left(s,t\right)dsdt,\qquad F_{\text{max}}=\sup_{s,t\in\left[0,1\right]}F_{n}\left(s,t\right)
%\]
%in which 
%\[
%F_{n}\left(s,t\right)=\frac{\nicefrac{\text{SSB}\left(s,t\right)}{\left(k-1\right)}}{\nicefrac{\text{SSE}\left(s,t\right)}{\left(n-k\right)}},
%\]
%and 
%\[
%\text{SSB}\left(s,t\right)=\sum_{i=1}^{k}\left(n_{i}-1\right)\left[\hat{\gamma}_{i}\left(s,t\right)-\hat{\gamma}\left(s,t\right)\right]^{2},
%\]
%\[
%\text{SSE}\left(s,t\right)=\sum_{i=1}^{k}\sum_{j=1}^{n_{i}}\left[\hat{\nu}_{ij}\left(s\right)\hat{\nu}_{ij}\left(t\right)-\hat{\gamma}_{i}\left(s,t\right)\right]^{2},
%\]
%where 
%\[
%\hat{\gamma}_{i}\left(s,t\right)=\left(n_{i}-1\right)^{-1}\sum_{j=1}^{n}\left[y_{ij}\left(s\right)-\bar{y}_{i}\left(s\right)\right]\left[y_{ij}\left(t\right)-\bar{y}_{i}\left(t\right)\right],
%\]
%and 
%\[
%\hat{\gamma}\left(s,t\right)=\left(n-k\right)^{-1}\sum_{i=1}^{k}\left(n_{i}-1\right)\hat{\gamma}_{i}\left(s,t\right),
%\]
%in which $\bar{y}_{i}\left(t\right)=n_{i}^{-1}\sum_{j=1}^{n_{i}}y_{ij}\left(t\right)$
%and $\hat{\nu}_{ij}\left(t\right)=y_{ij}\left(t\right)-\bar{y}_{i}\left(t\right)$.
We compared the new MMD test against quasi-GPF and quasi-$F_{\max}$ in a simulation study. Our simulation study is motivated by \citet{guo2019new}, and we used the same setup for the data generating procedure. Assume that the mean function is zero and data is
generated in a $k-$regime scheme according to the following model:
\[
y_{ij}\left(t\right)=\varepsilon_{ij}\left(t\right),\quad\varepsilon_{ij}\left(t\right)=h\left(t\right)\sum_{r\geq1}\sqrt{\lambda_{r}}z_{ijr}\psi_{ir}\left(t\right)
\]
\[
i=1,\ldots,k;\quad j=1,\ldots,n_{i};\quad t\in\left[0,1\right].
\]
where $h\left(t\right)$
is common for all the groups and $n_{i}$ denotes the size of each group. 
The set $\{\psi_{ir}\}_{r\geq 1}$, for each $i$, is a set of basis functions, and we set 
the eigenvalues $\lambda_{r}=a\rho^{r-1}$ for fixed $a>0$
and $\rho\in\left(0,1\right)$. The tuning parameter $\rho$ determines
the decay rate of eigenvalues. For a $\rho$ close to zero,
eigenvalues decay fast and functional data is more correlated
and more smooth, however, for a $\rho$ close to one, eigenvalues
decay slowly and realization of functional data is less correlated across its domain
and thus less smooth.

\begin{table}[tbh]
	\caption{Empirical powers (in percent) of $L^2$, $T_{\max}$, GPF, $F_{\max}$ and MMD when
		$z_{ijr}\stackrel{iid}{\sim}N\left(0,1\right)$.}
	
	\centering{}{\footnotesize{
			\begin{tabular}{cc|ccccc|ccccc}
				\hline 
				$\rho$ &  & \multicolumn{5}{c|}{$\left(n_{i}\right)=\left(20,30,30\right)$} & \multicolumn{5}{c}{$\left(n_{i}\right)=\left(70,80,100\right)$}\tabularnewline
				\hline 
				\multirow{6}{*}{0.1} & $\omega$ & \textbf{0} & 0.5 & 1 & 2.5 & 5 & \textbf{0} & 0.5 & 1 & 2.5 & 4\tabularnewline
				& $L^{2}$ & \textbf{5.3} & 5.4 & 6.3 & 7.1 & 8.1 & \textbf{4.8} & 4.7 & 6.3 & 6.0 & 7.2\tabularnewline
				& $T_{\max}$ & \textbf{4.8} & 5.0 & 5.7 & 6.7 & 7.6 & \textbf{4.5} & 4.6 & 6.4 & 6.2 & 8.1\tabularnewline
				& GPF & \textbf{4.4} & 5.3 & 7.1 & 19.4 & 72.9 & \textbf{5.2} & 6.4 & 7.6 & 70.1 & 100\tabularnewline
				& $F_{\text{max}}$ & \textbf{4.9} & 5.7 & 9.7 & 48.7 & 89.1 & \textbf{5.2} & 8.1 & 21.2 & 100 & 100\tabularnewline
				& MMD & \textbf{4.4} & 100 & 100 & 100 & 100 & \textbf{4.7} & 100 & 100 & 100 & 100\tabularnewline
				\hline 
				\multirow{6}{*}{0.5} & $\omega$ & \textbf{0} & 0.5 & 1 & 1.5 & 3 & \textbf{0} & 0.5 & 0.8 & 1.1 & 1.4\tabularnewline
				& $L^{2}$ & \textbf{4.6} & 5.0 & 6.3 & 6.7 & 5.3 & \textbf{4.8} & 4.6 & 4.4 & 6.2 & 5.3\tabularnewline
				& $T_{\max}$ & \textbf{4.7} & 4.5 & 6.1 & 4.0 & 5.4 & \textbf{4.6} & 4.9 & 5.1 & 5.8 & 5.4\tabularnewline
				& GPF & \textbf{5.1} & 9.5 & 27.3 & 46.7 & 95.1 & \textbf{5.6} & 28.7 & 65.7 & 91.5 & 99.3\tabularnewline
				& $F_{\text{max}}$ & \textbf{4.2} & 12.0 & 30.1 & 55.4 & 97.3 & \textbf{5.3} & 35.4 & 77.8 & 97.4 & 99.6\tabularnewline
				& MMD & \textbf{4.8} & 99.7 & 100 & 100 & 100 & \textbf{4.4} & 100 & 100 & 100 & 100\tabularnewline
				\hline 
				\multirow{6}{*}{0.9} & $\omega$ & \textbf{0} & 0.5 & 0.8 & 1.2 & 1.5 & \textbf{0} & 0.4 & 0.6 & 0.8 & 1\tabularnewline
				& $L^{2}$ & \textbf{5.7} & 6.6 & 5.7 & 5.4 & 6.7 & \textbf{5.9} & 4.8 & 5.6 & 5.1 & 5.8\tabularnewline
				& $T_{\max}$ & \textbf{4.7} & 6.7 & 6.0 & 5.3 & 6.8 & \textbf{4.6} & 4.7 & 6.0 & 5.4 & 5.7\tabularnewline
				& GPF & \textbf{5.5} & 16.4 & 29.4 & 60.8 & 75.7 & \textbf{5.2} & 25.2 & 55.1 & 84.5 & 98.7\tabularnewline
				& $F_{\text{max}}$ & \textbf{5.1} & 8.3 & 14.5 & 19.5 & 33.2 & \textbf{5.4} & 14.2 & 24.5 & 43.4 & 65.3\tabularnewline
				& MMD & \textbf{4.4} & 100 & 100 & 100 & 100 & \textbf{4.3} & 100 & 100 & 100 & 100\tabularnewline
				\hline 
	\end{tabular}}}%{\footnotesize\par}
	\label{tblCovNormal}
\end{table}

\begin{table}[tbh]
	\caption{Empirical powers (in percent) of $L^2$, $T_{\max}$, GPF, $F_{\max}$ and MMD when $z_{ijr}\stackrel{iid}{\sim}\sqrt{\nicefrac{3}{5}}~t_{5}$.}
	
	\centering{}{\footnotesize{
			\begin{tabular}{cc|ccccc|ccccc}
				\hline 
				$\rho$ &  & \multicolumn{5}{c|}{$\left(n_{i}\right)=\left(20,30,30\right)$} & \multicolumn{5}{c}{$\left(n_{i}\right)=\left(70,80,100\right)$}\tabularnewline
				\hline 
				\multirow{6}{*}{0.1} & $\omega$ & \textbf{0} & 0.5 & 1 & 2.5 & 5 & \textbf{0} & 0.5 & 1 & 2.5 & 4\tabularnewline
				& $L^{2}$ & \textbf{5.3} & 4.4 & 5.2 & 4.3 & 5.6 & \textbf{5.6} & 7.2 & 7.6 & 6.4 & 5.6\tabularnewline
				& $T_{\max}$ & \textbf{4.8} & 5.2 & 5.6 & 5.7 & 7.1 & \textbf{4.9} & 5.7 & 5.8 & 6.9 & 7.1\tabularnewline
				& GPF & \textbf{4.9} & 4.5 & 10.0 & 12.9 & 58.0 & \textbf{4.8} & 8.5 & 10.1 & 30.6 & 80.3\tabularnewline
				& $F_{\text{max}}$ & \textbf{5.2} & 8.5 & 13.1 & 34.1 & 70.1 & \textbf{5.7} & 12.8 & 20.4 & 88.5 & 98.5\tabularnewline
				& MMD & \textbf{4.2} & 100 & 100 & 100 & 100 & \textbf{4.4} & 100 & 100 & 100 & 100\tabularnewline
				\hline 
				\multirow{6}{*}{0.5} & $\omega$ & \textbf{0} & 0.5 & 1 & 1.5 & 3 & \textbf{0} & 0.5 & 0.8 & 1.1 & 1.4\tabularnewline
				& $L^{2}$ & \textbf{5.0} & 4.3 & 4.3 & 5.1 & 5.6 & \textbf{5.1} & 5.6 & 4.3 & 6.9 & 6.3\tabularnewline
				& $T_{\max}$ & \textbf{4.7} & 5.5 & 4.8 & 5.7 & 7.2 & \textbf{4.7} & 5.4 & 4.4 & 5.7 & 6.8\tabularnewline
				& GPF & \textbf{5.4} & 14.2 & 21.4 & 34.2 & 72.8 & \textbf{4.9} & 15.7 & 44.2 & 65.7 & 84.2\tabularnewline
				& $F_{\text{max}}$ & \textbf{4.6} & 12.8 & 22.6 & 40.3 & 68.5 & \textbf{5.4} & 24.2 & 61.4 & 84.2 & 91.4\tabularnewline
				& MMD & \textbf{4.8} & 84.2 & 100 & 100 & 100 & \textbf{5.0} & 100 & 100 & 100 & 100\tabularnewline
				\hline 
				\multirow{6}{*}{0.9} & $\omega$ & \textbf{0} & 0.5 & 0.8 & 1.2 & 1.5 & \textbf{0} & 0.4 & 0.6 & 0.8 & 1\tabularnewline
				& $L^{2}$ & \textbf{4.1} & 4.3 & 4.2 & 3.8 & 4.9 & \textbf{5.2} & 6.4 & 5.3 & 6.7 & 6.5\tabularnewline
				& $T_{\max}$ & \textbf{5.9} & 4.2 & 5.5 & 5.0 & 5.8 & \textbf{4.9} & 6.9 & 7.1 & 6.9 & 6.8\tabularnewline
				& GPF & \textbf{5.2} & 15.7 & 18.5 & 31.4 & 54.9 & \textbf{4.4} & 10.3 & 44.3 & 67.1 & 72.8\tabularnewline
				& $F_{\text{max}}$ & \textbf{4.1} & 11.4 & 10.3 & 12.8 & 30.3 & \textbf{5.1} & 7.1 & 12.8 & 23.8 & 50.3\tabularnewline
				& MMD & \textbf{5.4} & 91.4 & 100 & 100 & 100 & \textbf{4.8} & 100 & 100 & 100 & 100\tabularnewline
				\hline 
	\end{tabular}}}%
	\label{tblCovT}
\end{table}

Although \citet{guo2019new} assumed a finite
number of $q$ nonzero eigenvalues in the simulation process, our test works well without the need to put any restriction
on the number of components. Here we follow \citet{guo2019new} and 
fix $q=41$, $a=1.5$, $T=80$, $h\left(t\right)=\frac{T}{t+1}$ where
$T$ is the number of time points that each curve is observed at.
Different setups for data generating procedure is a combination
of the following choice of parameters:
\begin{itemize}
	\item $z_{ijr}\stackrel{iid}{\sim}N\left(0,1\right)$ or $z_{ijr}\stackrel{iid}{\sim}\sqrt{\nicefrac{3}{5}}~t_{5}$.
	\item $\rho=0.1, 0.5, 0.9$ for three class of high, moderate and low correlations.
	\item $\left(n_{i}\right)=\left(20, 30, 30\right)$ for the small sample
	%, $\left(n_{i}\right)=\left(30, 40, 50\right)$ for the medium sample
	and $\left(n_{i}\right)=\left(70, 80, 100\right)$ for the large sample
	cases.
	\item $\psi_{ir}\left(t\right)=\phi_{r}\left(t\right)$ for $r=1, 3, 4,\ldots,q$
	and $\psi_{i2}\left(t\right)=\phi_{2}\left(t\right)+\left(i-1\right)\omega/h\left(t\right)$
	for different choice of $\omega$ to reflect between group difference
	of covariance operators, and we can take $\left\{ \phi_{r}\right\} _{r\geq1}$
	either the set of Fourier or B-spline basis for $L^{2}\left[0,1\right]$.
\end{itemize}

\citet{guo2019new} compared quasi-GPF and quasi-$F_{\text{max}}$
with few other tests including two other tests $L^2$ and $T_{\max}$ \citep{gou2018}. According
to their results, quasi-GPF is superior in low correlation schemes, and
quasi-$F_{\text{max}}$ is superior in the high correlation schemes.
Although In this paper we derived kernel mean embedding and MMD statistic for the family of Gaussian probability
distributions, it can be noticed from Tables \ref{tblCovNormal} and \ref{tblCovT} that the new MMD based test dominates the quasi-GPF and quasi-$F_{\text{max}}$ in all situations including non-Gaussian scenarios.

Let $\hat{C}$ be the usual estimation of the covariance operator under
$H_{0}$ and $\hat{C}_{i}$ the usual covariance operator estimation
of group $i$. Then our MMD test statistic equals  
\[
\hat{\text{MMD}}=\left[\sum_{i=1}^{k}n_{i}\left(\left|I+4\sigma\hat{C}\right|^{-\nicefrac{1}{2}}\hspace{-0.45em}+\left|I+4\sigma\hat{C}_{i}\right|^{-\nicefrac{1}{2}}\hspace{-0.45em}-2\left|I+2\sigma\left(\hat{C}+\hat{C}_{i}\right)\right|^{\nicefrac{-1}{2}}\right)\right]^{\nicefrac{1}{2}}.
\]
In this simulation study, we take $\sigma=1e3$. The null distribution of all the test statistics $L^2$, $T_{\max}$, $F_{\max}$, GPF and MMD are approximated by
the random permutation method. The empirical powers of the five test statistics are calculated in a simulation study. Results for $\alpha=0.05$ and $\left\{ \phi_{r}\right\} _{r\geq1}$ selected to be the set of
Fourier basis are presented in Tables \ref{tblCovNormal} and \ref{tblCovT}. In this simulation study, we used the B-spline basis for the smoothing procedure. % to have a different basis other than the eigenfunctions of covariance operator, and 
The number of components for smoothing procedure is considered to be fixed and equals 41. 

According to the results,  the empirical  powers of MMD test are uniformly
higher than the other four tests in all the situations. The results presented in these tables were produced and reported by 2000 iterations.

\subsubsection{Medfly Data}
In this section, we apply MMD and the other four tests introduced in Section \ref{sec:homCovTest},  $L^2$, $T_{\max}$, $\text{GPF}$ and $F_{\max}$, to test for homogeneity of covariance operators in a real data example, according to the model (\ref{eq:CovModel}) with $k=4$.

Medfly data set is a functional data of mortality rate of medflies. 
Approximately, 7,200 medflies of a given size were maintained in aluminum cages. Adults were given either a diet of sugar and water, or a diet of sugar, water and ad libitum. Each day, dead flies were removed, counted, and their sex determined \citep{Carey1992}. The number and rate of alive  medflies were recorded over a period of 101 days. In effect, the aim is to assess the effects of nutrition and gender on survival or mortality of medlies.

Cohorts of medflies consist of four groups, (a) Females on a sugar diet, (b) Females on a protein plus sugar diet, (c) Males on a sugar diet, and (d) Males on a protein plus sugar diet. The effect of gender and nutrition is studied before \citep[see for example][]{Koenker2001,Muller2003data}, and it is known that there is an interaction between gender and nutrition on the survival of medflies \citep{Muller1998data}. Survival functions of cohorts of medflies during a period of 30 days (days 2-31) are illustrated in Figure \ref{fig:rfig}. 
\begin{figure}[th]
	\begin{subfigure}{.24\textwidth}
		\centering
		\includegraphics[width=.98\linewidth]{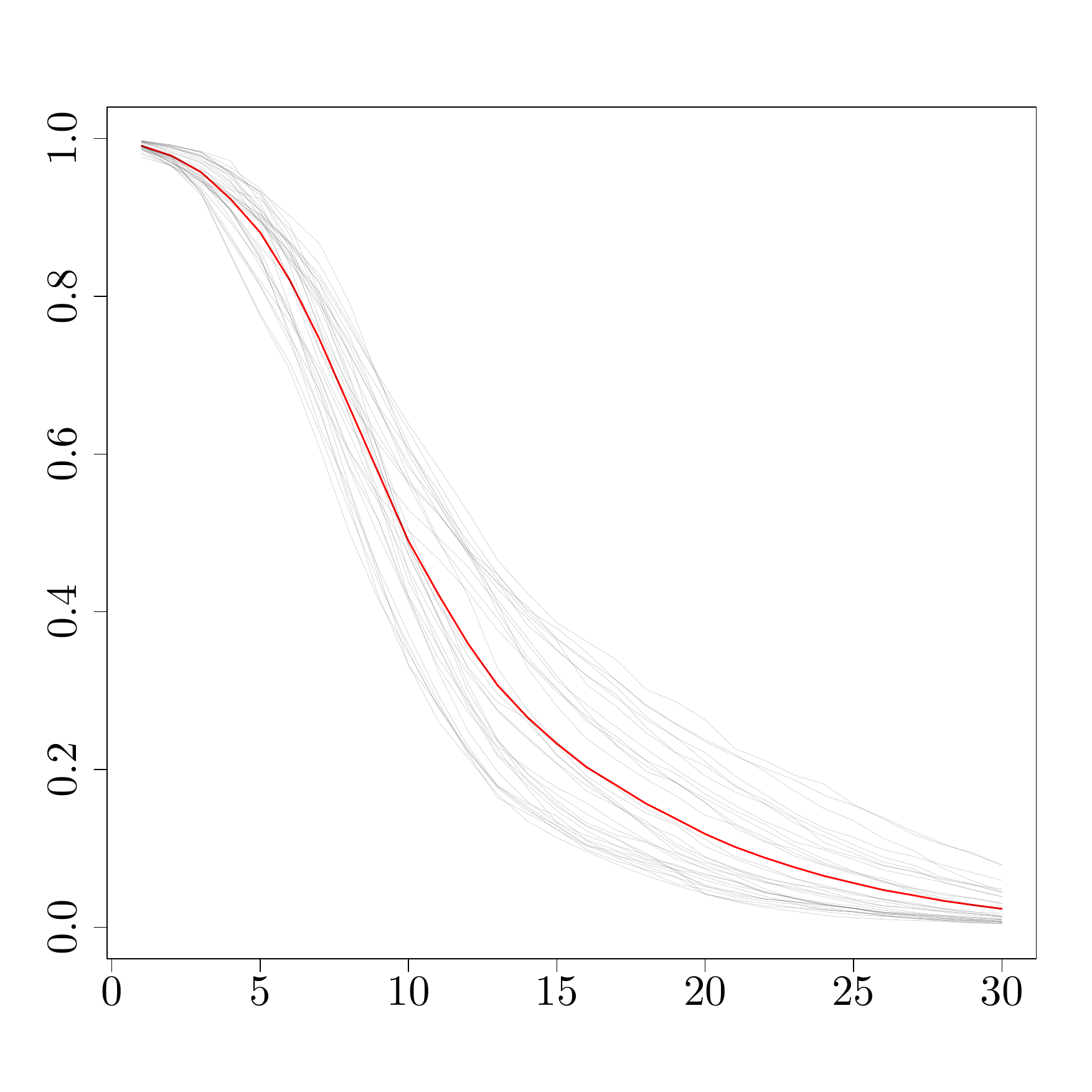}
		%		\caption{}
		%		\label{fig:ffig3}
		
	\end{subfigure}%
	\begin{subfigure}{.24\textwidth}
		\centering
		\includegraphics[width=.98\linewidth]{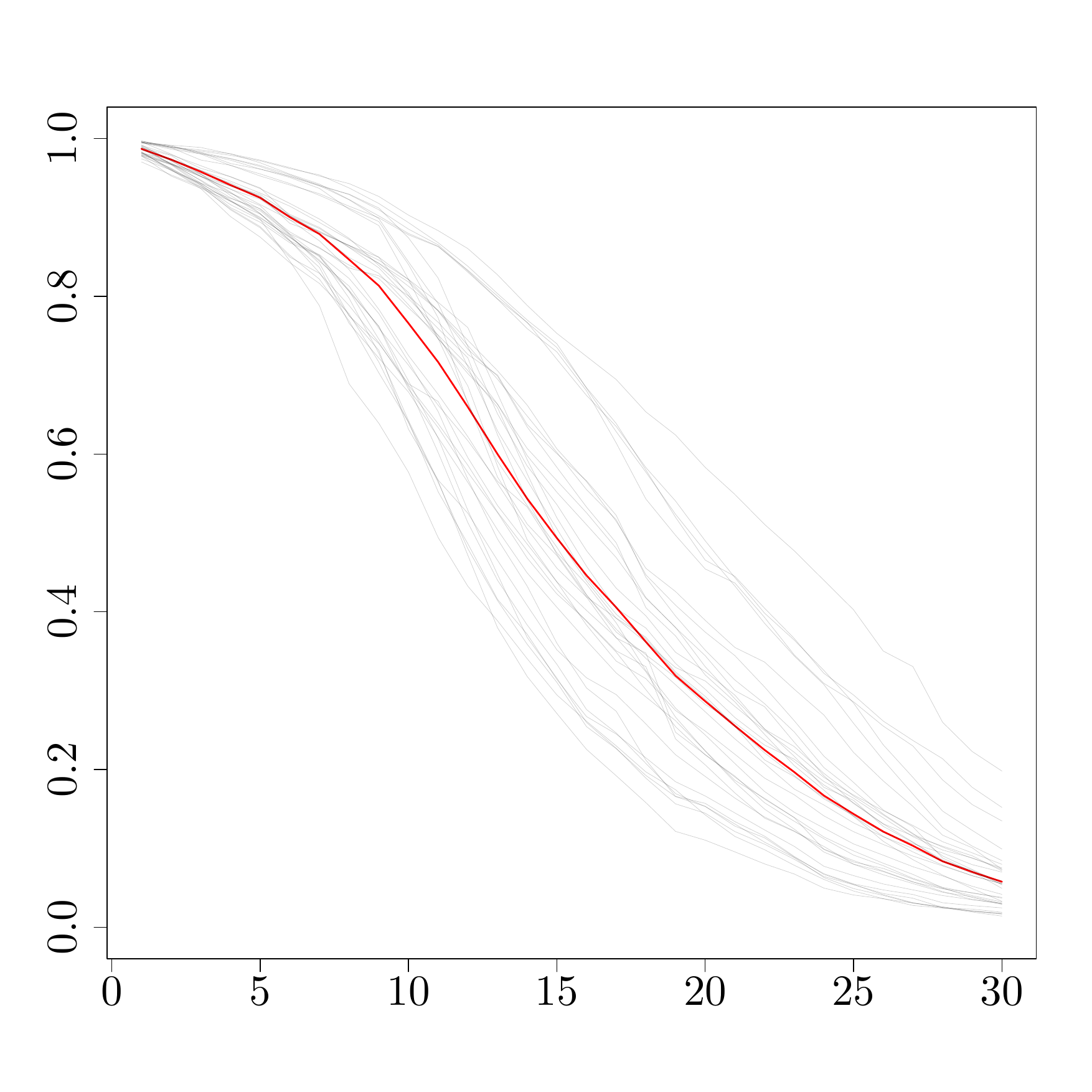}
		%\caption{}
		%\label{fig:ffig4}
	\end{subfigure}
	\begin{subfigure}{.24\textwidth}
		\centering
		\includegraphics[width=.98\linewidth]{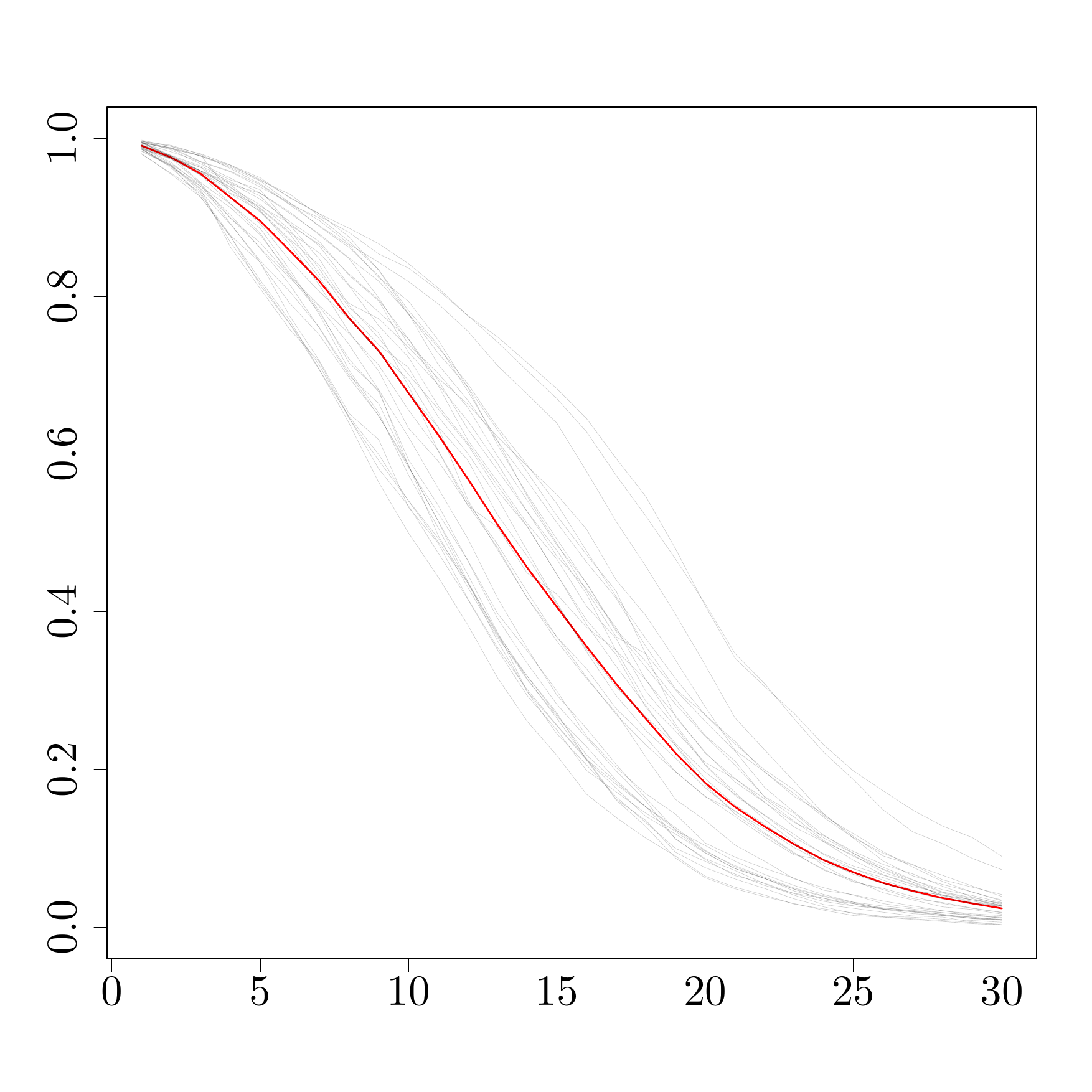}
		%\caption{}
		%\label{fig:ffig1}
	\end{subfigure}
	\begin{subfigure}{.24\textwidth}
		\centering
		\includegraphics[width=.98\linewidth]{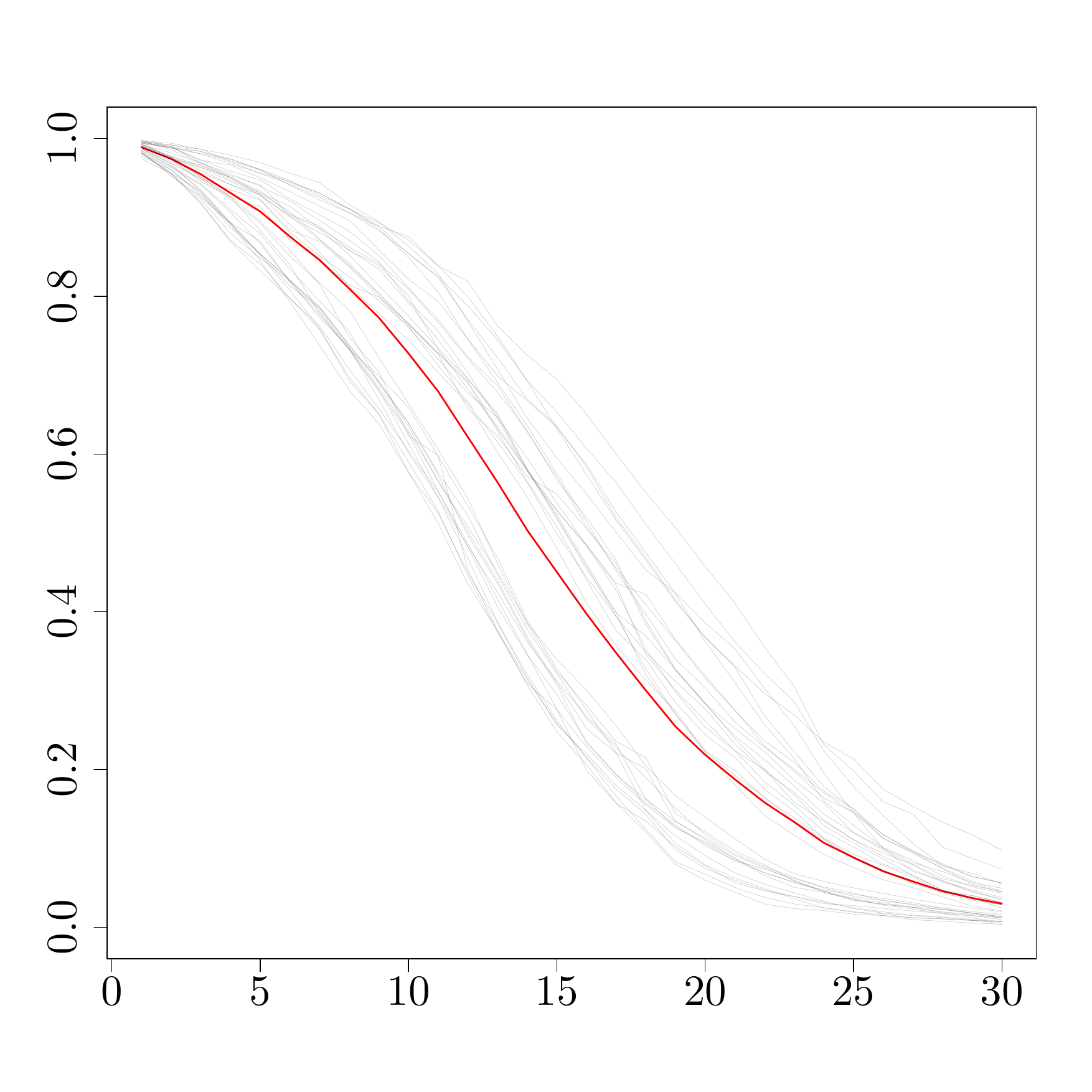}
		%\caption{}
		%\label{fig:ffig2}
	\end{subfigure}
	
	\begin{subfigure}{.24\textwidth}
		\centering
		\includegraphics[width=.98\linewidth]{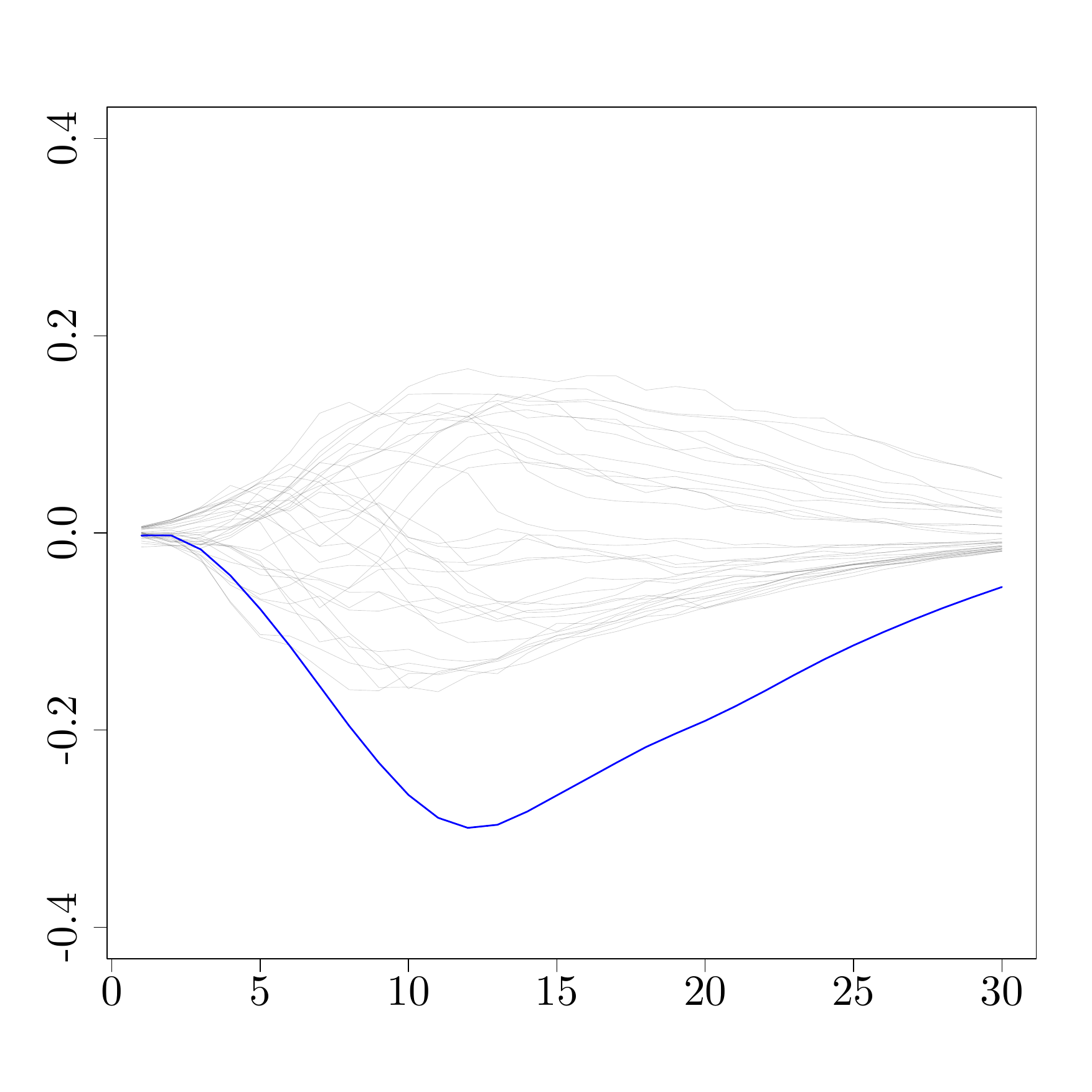}
		\caption{}
		%\label{fig:rfig3}
	\end{subfigure}%
	\begin{subfigure}{.24\textwidth}
		\centering
		\includegraphics[width=.98\linewidth]{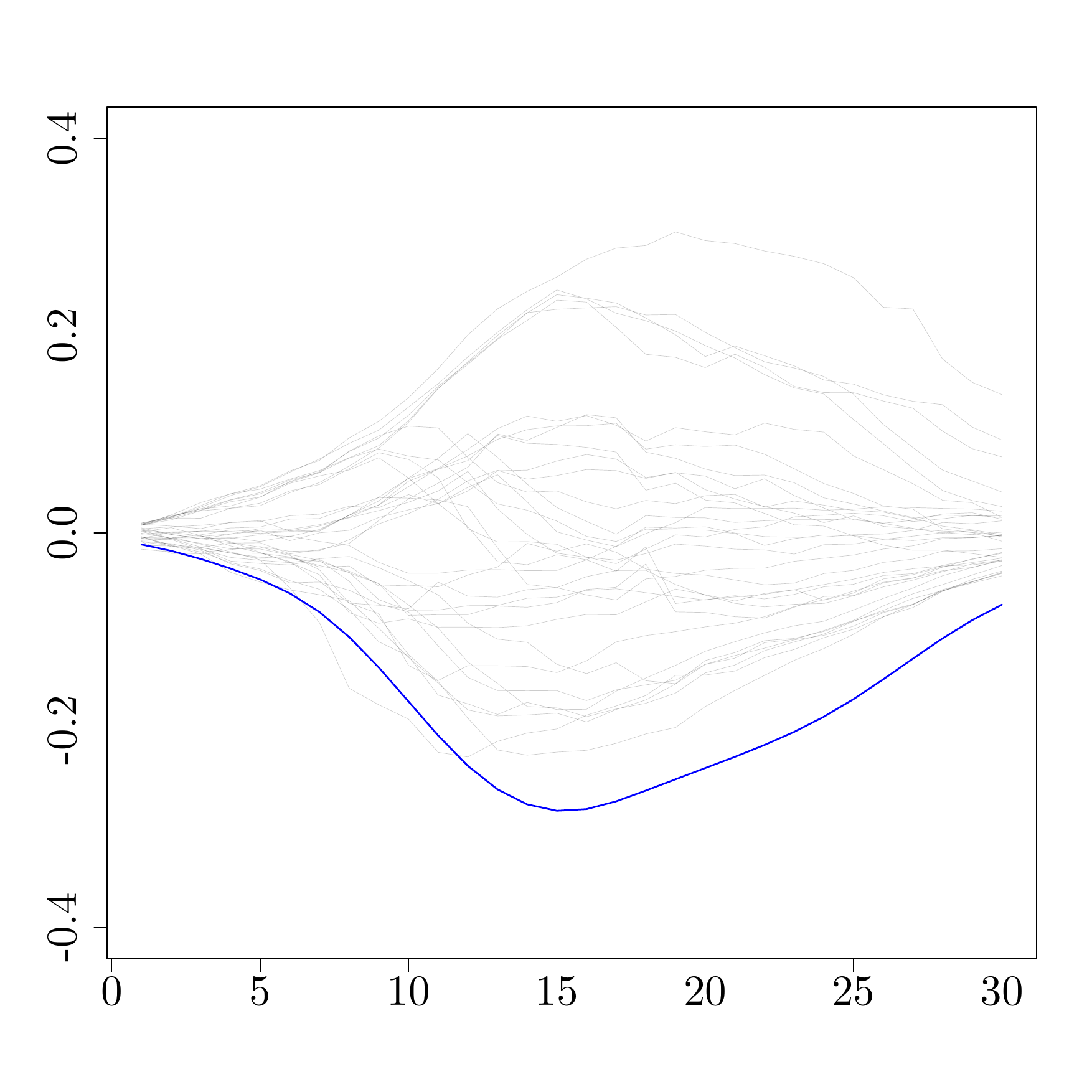}
		\caption{}
		\label{fig:rfig4}
	\end{subfigure}
	\begin{subfigure}{.24\textwidth}
		\centering
		\includegraphics[width=.98\linewidth]{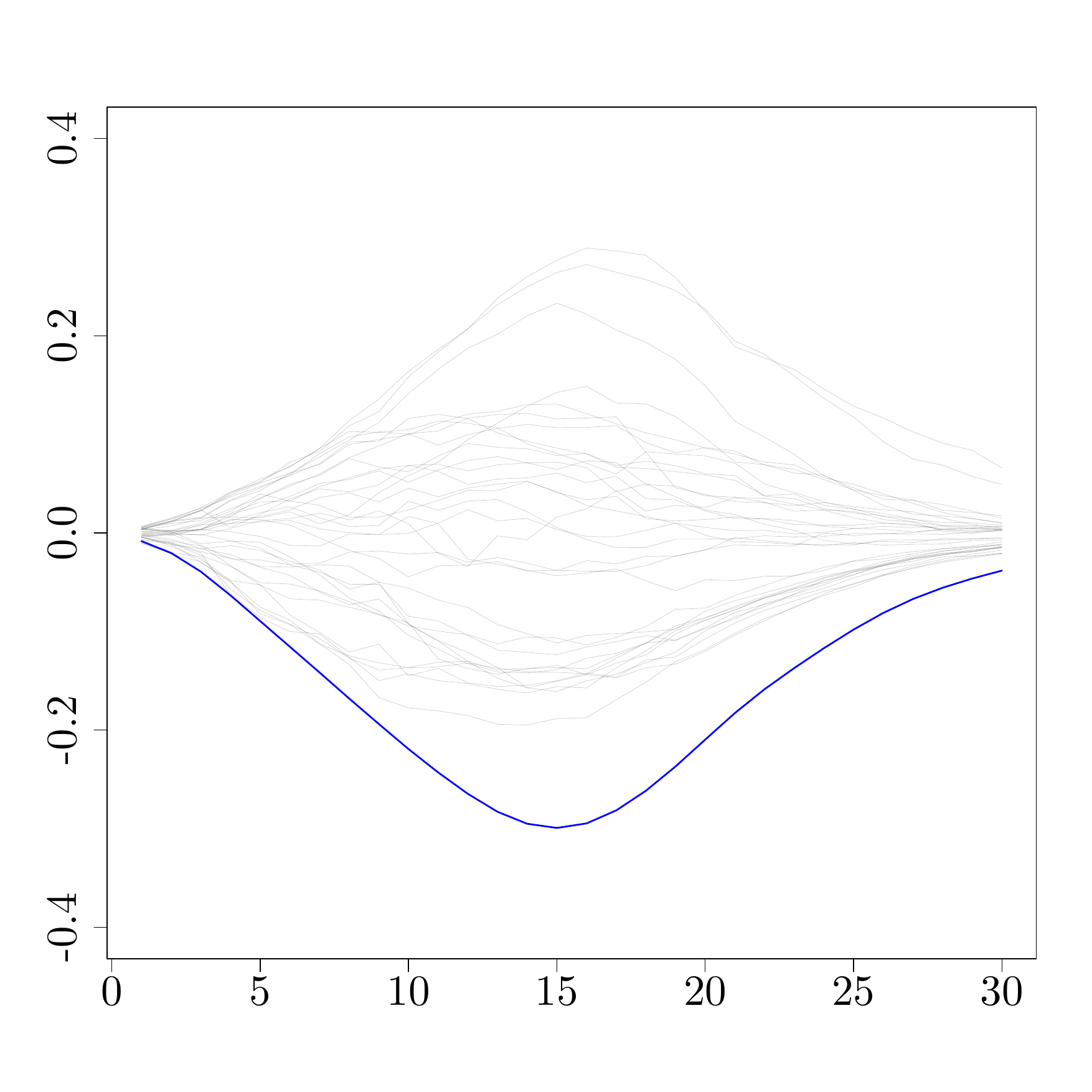}
		\caption{}
		%\label{fig:rfig1}
	\end{subfigure}
	\begin{subfigure}{.24\textwidth}
		\centering
		\includegraphics[width=.98\linewidth]{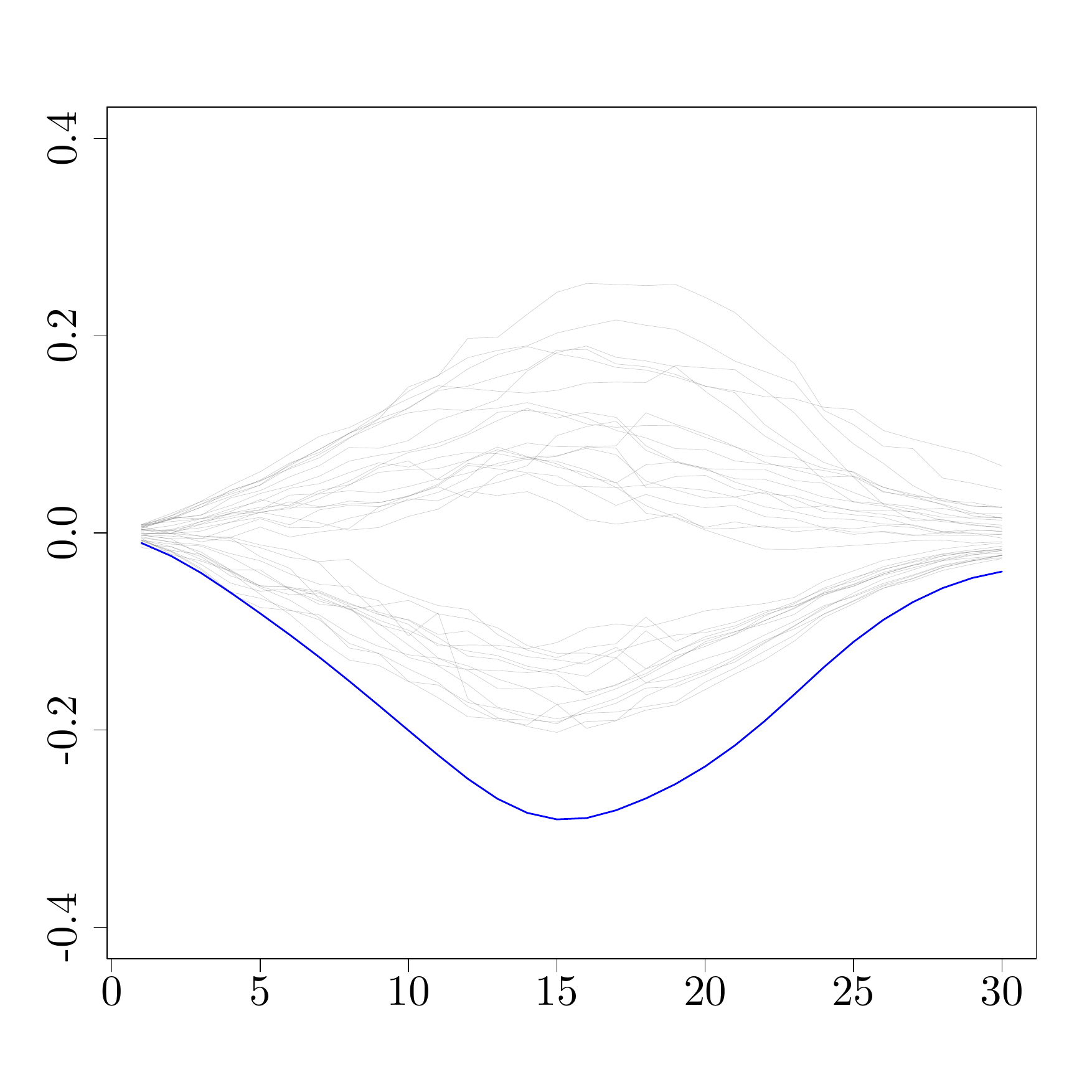}
		\caption{}
		%\label{fig:rfig2}
	\end{subfigure}
	\caption{Survival functions of cohorts of medflies:
		(a) Females on a sugar diet, (b) Females on a protein plus sugar diet, (c) Males on a sugar diet, and (d) Males on a protein plus sugar diet; (Top row) Gray lines: survival functions of samples, Red line: mean function, 
		(Bottom row) Gray lines: deviation of samples from the mean function, Blue line: first eigenfunction of covariance operators, which accounts for 
		89.3\%, 94.5\%, 96.4\% and 97.8\% of the variation of survival functions in each of four groups respectively.}
	\label{fig:rfig}
\end{figure}
\begin{figure}[bh!]
	\begin{subfigure}{.23\textwidth}
		\centering
		\includegraphics[width=.95\linewidth]{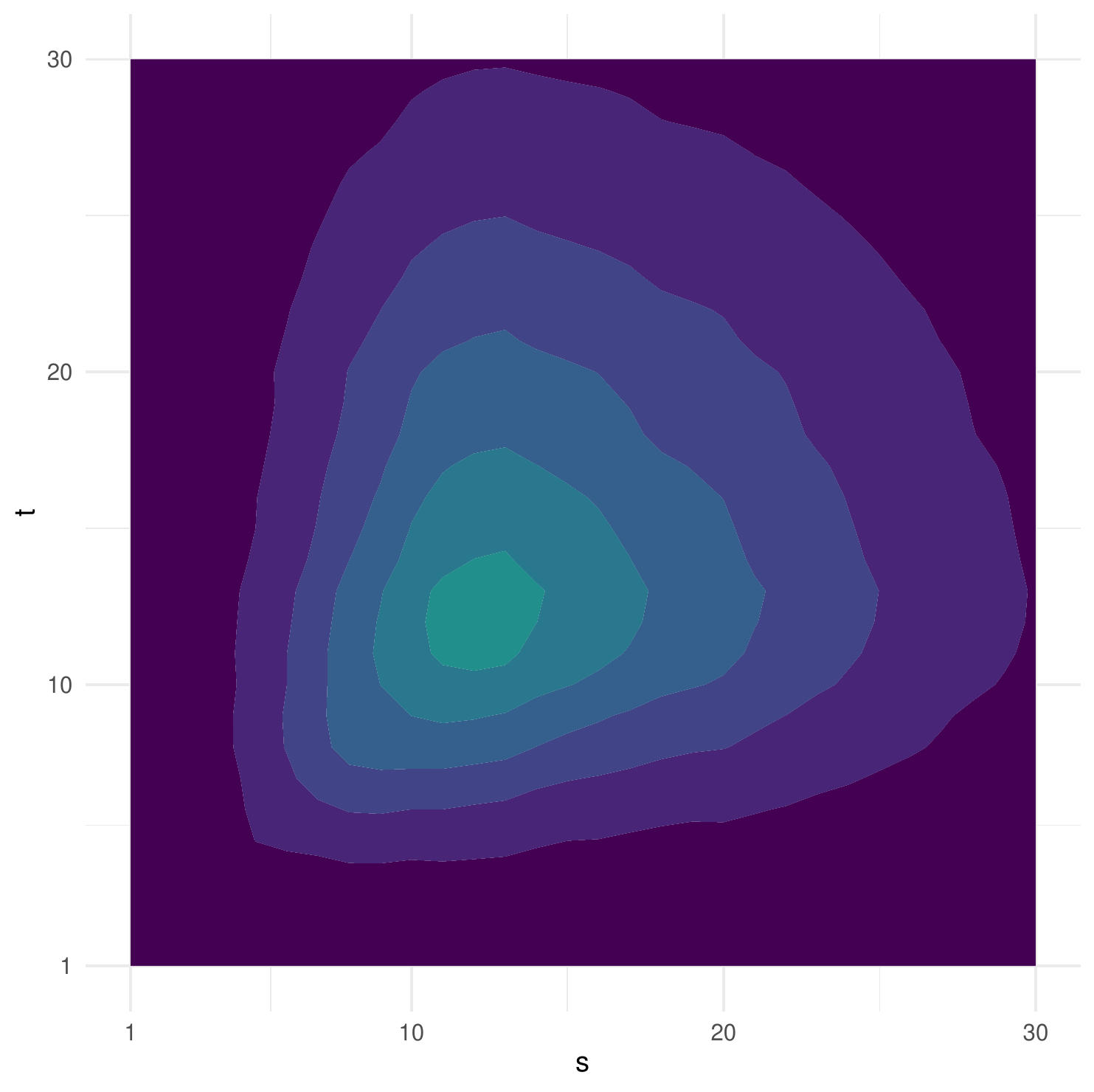}
		\caption{}
		\label{fig:sfig1}
	\end{subfigure}%
	\begin{subfigure}{.23\textwidth}
		\centering
		\includegraphics[width=.95\linewidth]{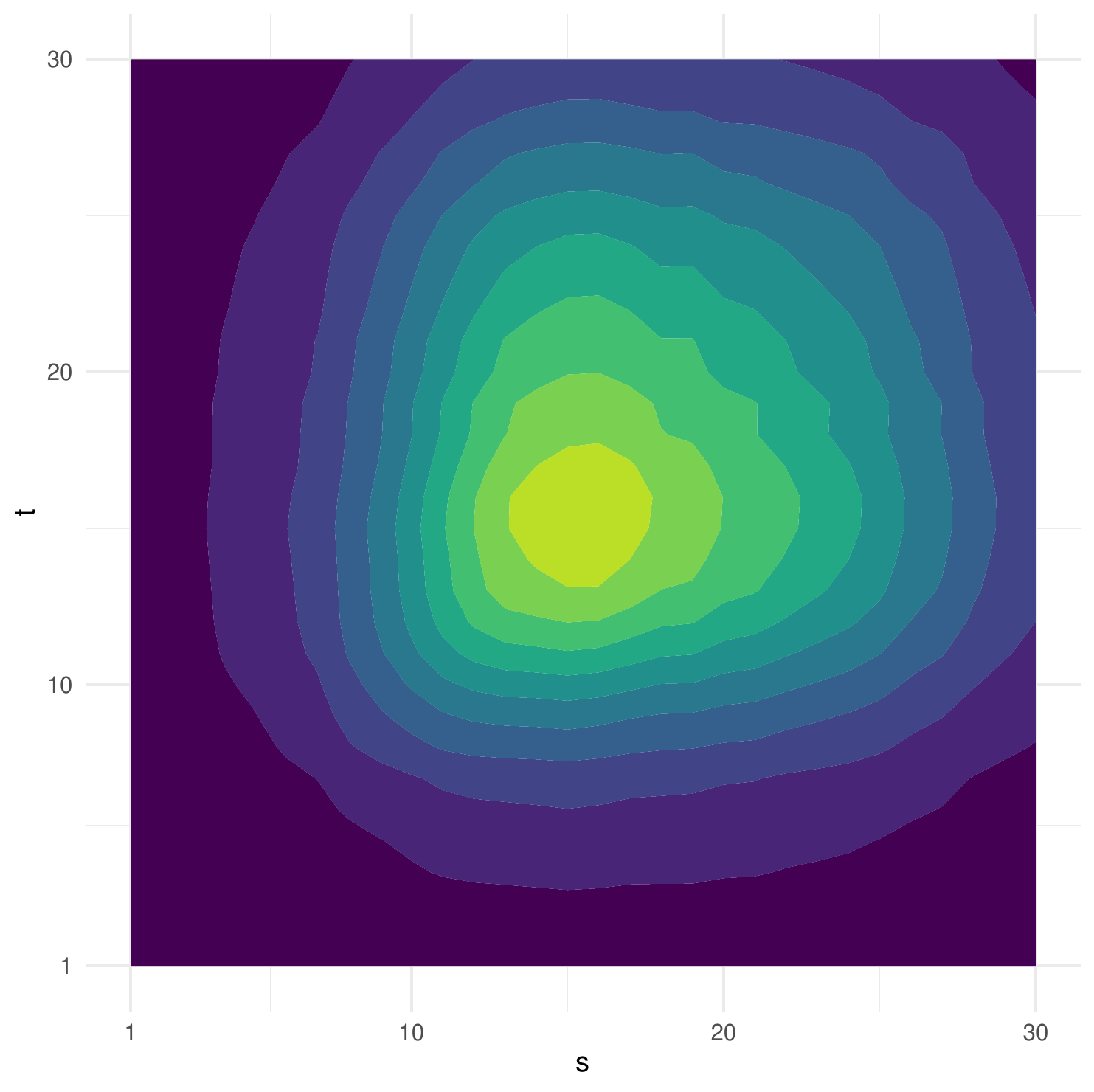}
		\caption{}
		\label{fig:sfig2}
	\end{subfigure}
	\begin{subfigure}{.23\textwidth}
		\centering
		\includegraphics[width=.95\linewidth]{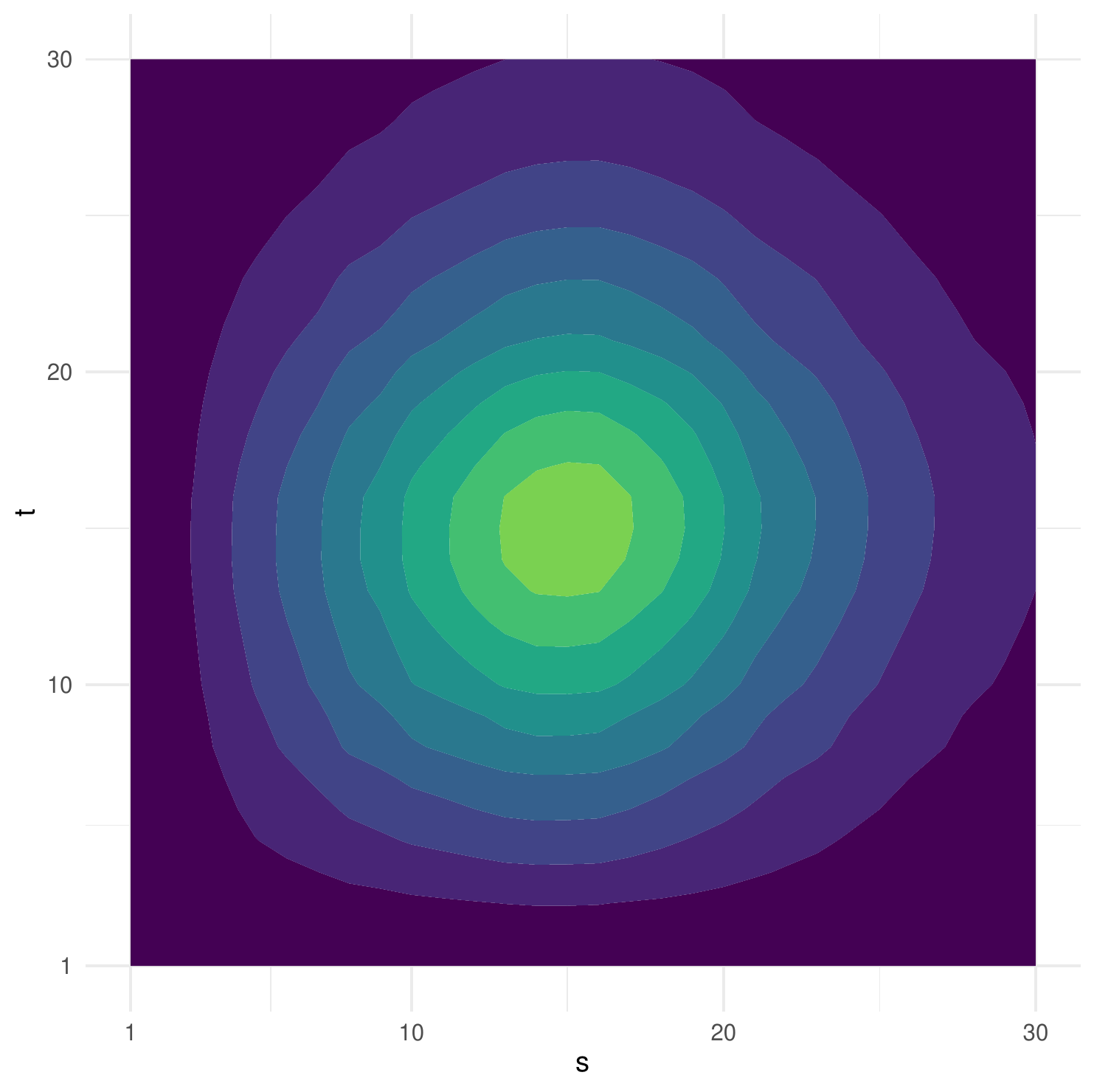}
		\caption{}
		\label{fig:sfig3}
	\end{subfigure}
	\begin{subfigure}{.285\textwidth}
		\centering
		\includegraphics[width=.95\linewidth]{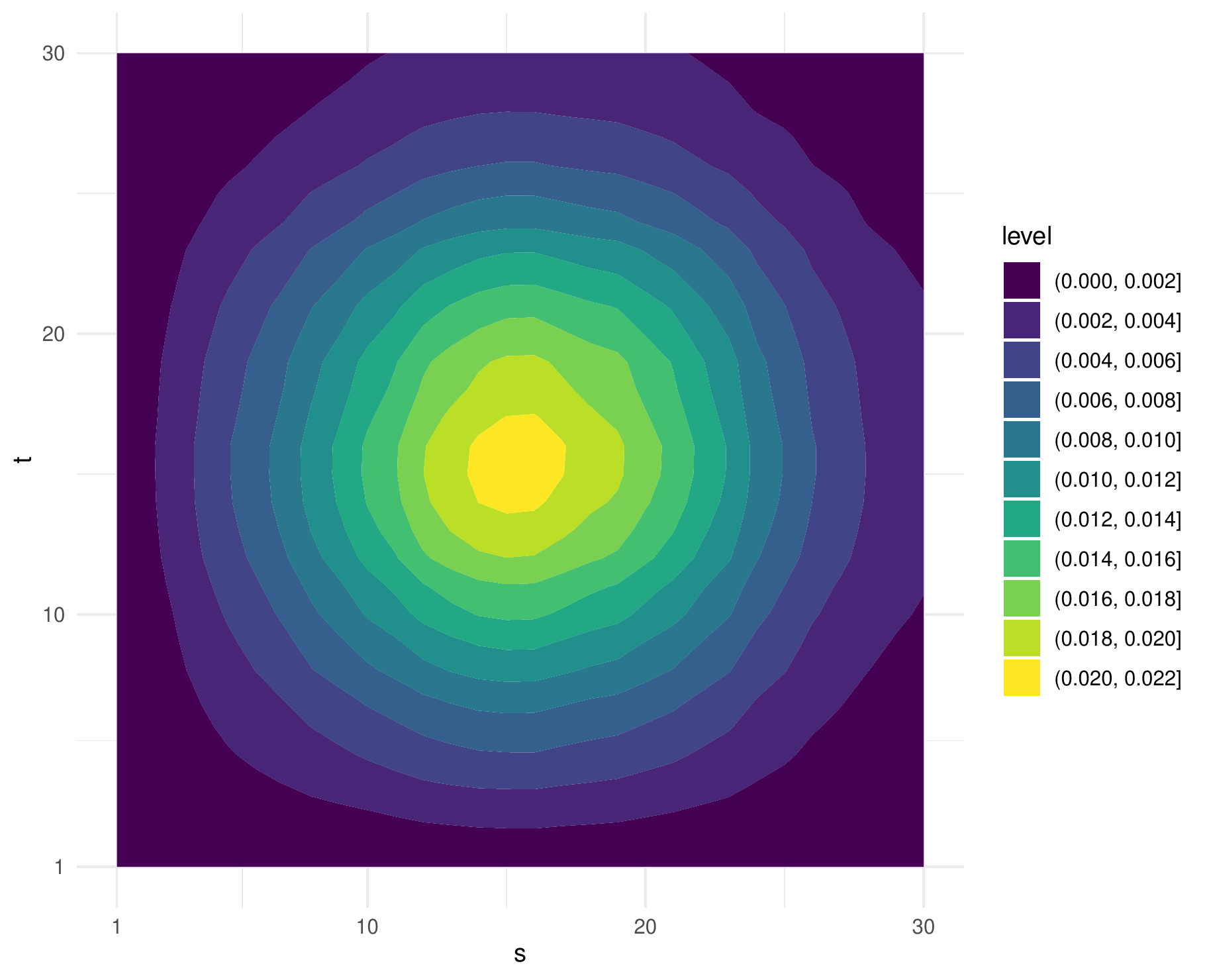}
		\caption{}
		\label{fig:sfig4}
	\end{subfigure}
	\caption{Estimated covariance functions of the four groups of medflies: 
		(a) Females on a sugar diet, (b) Females on a protein plus sugar diet, (c) Males on a sugar diet, and (d) Males on a protein plus sugar diet.}
	\label{fig:covfig}
\end{figure}
The panels in the first row demonstrate 33 sample functions in each group as well as the mean functions, and each sample function is the survival rate of medflies in one cage. The panels in the second row demonstrate deviation of samples from the group's mean function as well as the first eigenfunction of the covariance operator. The first eigenfunction explains the major variation of functional samples within each group (89.3\%, 94.5\%, 96.4\%, and 97.8\% in each group respectively). It could be noticed that there is a slight difference in eigenvalues and eigenfunctions of covariance operators between groups. The kernel functions of the covariance operators in the four groups are depicted in Figure \ref{fig:covfig}, which magnifies the between-groups difference of covariance operators. 
%\begin{figure}[th]
%	\begin{subfigure}{.416\textwidth}
%		\centering
%		\includegraphics[width=.95\linewidth]{images/C3.pdf}
%		\caption{}
%		\label{fig:sfig1}
%	\end{subfigure}%
%	\begin{subfigure}{.416\textwidth}
%		\centering
%		\includegraphics[width=.95\linewidth]{images/C4.pdf}
%		\caption{}
%		\label{fig:sfig2}
%	\end{subfigure}
%\begin{subfigure}{.07\textwidth}
%	\centering
%	\hspace{1em}
%	%\includegraphics[width=.95\linewidth]{images/C4.pdf}
%%	\caption{}
%%	\label{fig:sfig200}
%\end{subfigure}
%	\begin{subfigure}{.416\textwidth}
%		\centering
%		\includegraphics[width=.95\linewidth]{images/C1.pdf}
%		\caption{}
%		\label{fig:sfig3}
%	\end{subfigure}
%	\begin{subfigure}{.5\textwidth}
%		\centering
%		\includegraphics[width=.95\linewidth]{images/C2_leg.pdf}
%		\caption{}
%		\label{fig:sfig4}
%	\end{subfigure}
%	\caption{Estimated covariance functions of the four groups of medflies: 
%		(a) Females on a sugar diet, (b) Females on a protein plus sugar diet, (c) Males on a sugar diet, and (d) Males on a protein plus sugar diet.}
%	\label{fig:covfig}
%\end{figure}
MMD and the four other tests including $L^2$, $T_{\max}$, $\text{GPF}$ and $F_{\max}$ were employed to test the equality of covariance operators. The results are presented in Table \ref{tbl:tblCovRealData1}.

% It can be understood that the results of the $\text{MMD}$ test are similar to those of the $F_{\max}$ test.
It can be understood that the p-values of all pair-wise comparisons are generally smaller than of the other four tests.
%, indicating that MMD test may be more powerful in detecting the between group difference of covariance function.
As described in \citet{guo2019new}, it was expected that $F_{\max}$ test to have higher power than $GPF$ in this data set. However, as it is shown in simulation studies, $\text{MMD}$ has higher power than both $F_{\max}$ and $GPF$ in all of the scenarios.
\begin{table}[tbh]
	\caption{p-Values (in percent) of $L^{2}$, $T_{\max}$, GPF, $F_{\max}$ and MMD tests\newline applied to compare covariance operators of survival functions of the four groups
		of medflies.}
	\centering{}%
	\begin{tabular}{c|ccccc}
		& $L^{2}$ & $T_{\max}$ & $\text{GPF}$ & $F_{\max}$ & $\text{MMD}$\tabularnewline
		\hline 
		(a) vs (b) & \phantom{0}1.4 & \phantom{0}0.4 & \phantom{0}0.6 & 0.1 & 0.1\tabularnewline
		(a) vs (c) & \phantom{0}6.2 & \phantom{0}0.6 & \phantom{0}3.4 & 0.1 & 0.1\tabularnewline
		(a) vs (d) & \phantom{0}0.1 & \phantom{0}0.1 & \phantom{0}0.1 & 0.1 & 0.1\tabularnewline
		(b) vs (c) & 21.2 & 29.2 & 12.2 & 3.6 & 0.1\tabularnewline
		(b) vs (d) & 33.8 & 47.0 & 13.8 & 1.2 & 0.2\tabularnewline
		(c) vs (d) & 24.0 & 27.8 & 21.2 & 6.2 & 2.6\tabularnewline
		All Groups & \phantom{0}2.6 & \phantom{0}01.8 & \phantom{0}0.4 & 0.1 & 0.1\tabularnewline
		\hline 
	\end{tabular}
	\label{tbl:tblCovRealData1}
\end{table}

\section{Conclusions and Discussion}\label{sec:discuss}
This study explored kernel methods for probability measures and their applications to functional data analysis.  We derived conditions of kernels that are characteristic for infinite-dimensional separable Hilbert spaces.  We also derived a framework for introducing a pseudo-likelihood function over infinite-dimensional separable Hilbert spaces. It is shown that the MKM estimators for location and covariance operator obtained by maximizing this pseudo-likelihood function coincide with ordinary least square estimators, which is the same as what we observe in finite-dimensional spaces where ordinary least square estimators coincide with MLE in the case of Gaussian distribution. We also used Maximum Mean Discrepancy as a distance over the space of Gaussian probability measures induced by functional response models and derived new powerful tests for the problems of functional one-way ANOVA and homogeneity of covariance operators. An important question which we have not covered in this paper is how to choose the Gaussian kernel bandwidth parameter $\sigma$. As it is also proposed in \citet{sriperumbudur2010hilbert}, one may choose a family of characteristic kernels $\left\{k_\sigma\left(x,y\right):=exp\left(-\sigma\left\Vert x-y\right\Vert^2\right),~\sigma>0\right\}$ and use the maximal RKHS distance  $\gamma\left(\mathbb{P},\mathbb{Q}\right)=\sup_{\sigma>0}{\gamma_\sigma\left(\mathbb{P},\mathbb{Q}\right)}$ where $\gamma_\sigma$ is the MMD metric defined by characteristic kernel $k_\sigma$. $\gamma$ is a stronger metric than $\gamma_\sigma$, so new tests derived from $\gamma$ must have a better performance than those introduced in section \ref{sec::aplications}.\\

\appendix
\section{Appendix}
\subsection{Proof of Theorem \ref{thm:thm1} and Corollary \ref{cor:smallBallProb}}
To provide the proof of Theorem \ref{thm:thm1} we need the following lemma:
\begin{lem}\label{lem:sumajbj}
	Let $\{b_{j}\}$ be a descending sequence of positive real numbers
	and $\{a_{j}\}$ be a series of real numbers such that $\sum_{j\geq1}\left|a_{j}\right|<\infty$
	and $\sum_{j\geq1}a_{j}b_{j}>0$. Then there exists a finite $N\in\mathbb{N}$
	such that $\sum_{j=1}^{N}a_{j}>0$.
\end{lem}

\begin{proof}
	Let $P=\left\{ n_{1},n_{2},\ldots\right\} \subseteq\mathbb{N}$ be the set
	of indices for which $a_{j}>0$, define $n_{0}=0$ and for any $n_{i}\in P$,
	let $T_{n_{i}}=\mathbb{N}\cap(n_{i-1},n_{i}]$. Then for any $i\geq1$,
	we have 
	$
	b_{n_{i}}\sum_{j\in T_{n_{i}}}a_{j}\geq\sum_{j\in T_{n_{i}}}b_{j}a_{j}.
	$
	Let $n_{k}\in P$ be the first index such that $\sum\limits _{j=1}^{n_{k}}a_{j}b_{j}>0$. If $k=1$, the proof is straightforward. If $k>1$, then
	\begin{align*}
	\sum_{j=1}^{n_{k}}a_{j}=\sum_{i=1}^{k}\sum_{j\in T_{n_{i}}}\hspace{-0.2em}a_{j}&\geq\sum_{i=1}^{k}\frac{1}{b_{n_{i}}}\sum_{j\in T_{n_{i}}}b_{j}a_{j}  \geq\frac{1}{b_{n_{k-1}}}\sum_{i=1}^{k-1}\sum_{j\in T_{n_{i}}}b_{j}a_{j}+\frac{1}{b_{n_{k}}}\sum_{j\in T_{n_{i}}}b_{j}a_{j}\\
	& \geq\frac{1}{b_{n_{k-1}}}\sum_{i=1}^{k}\sum_{j\in T_{n_{i}}}b_{j}a_{j}=\frac{1}{b_{n_{k-1}}}\sum_{j=1}^{n_{k}}b_{j}a_{j}>0.
	\end{align*}
\end{proof}

\begin{proof}[Proof of Theorem \ref{thm:thm1}:]
	Suppose $m_{P_{2}}(y)-m_{P_{1}}(y)=\delta>0$. There exists $r>0$
	big enough such that $\sup\limits _{x\in B_{r}(y)^{c}}\psi(\left\Vert x-y\right\Vert _{\mathbb{H}})\leq\delta/2$,
	in which $B_{r}(y)=\{x\in\mathbb{H}\:\text{s.t}\:\left\Vert x-y\right\Vert _{\mathbb{H}}<r\}$.
	Then, we have
	
	\begin{align*}
	0<\delta & =\int\limits _{\mathbb{H}}\psi(\left\Vert x-y\right\Vert _{\mathbb{H}})(P_{2}-P_{1})(dx)\\
	&=\int\limits _{B_{r}(y)}\psi(\left\Vert x-y\right\Vert _{\mathbb{H}})(P_{2}-P_{1})(dx)+\int\limits _{B_{r}(y)^{c}}\psi(\left\Vert x-y\right\Vert _{\mathbb{H}})(P_{2}-P_{1})(dx)\\
	& \leq\int\limits _{B_{r}(y)}\psi(\left\Vert x-y\right\Vert _{\mathbb{H}})(P_{2}-P_{1})(dx)+\delta/2
	\end{align*}
	and thus
	\begin{equation}
	\int\limits _{B_{r}(y)}\psi(\left\Vert x-y\right\Vert _{\mathbb{H}})(P_{2}-P_{1})(dx)\geq\frac{\delta}{2}>0.\label{eq:q11}
	\end{equation}
	Let define
	\begin{itemize}
		\item $r_{i,L}=(1-L^{i})r$ ; $i\geq1$ , $L\in(0,1)$
		\item $B_{i,L}=B_{r_{i,L}}(y)$ ; $i\geq1$
		\item $B_{1,L}^{'}=B_{1,L}$ , $B_{i,L}^{'}=B_{i,L}\backslash B_{\left(i-1\right),L}$
		; $i\geq2$
	\end{itemize}
	thus from (\ref{eq:q11}) we have
	\begin{align*}
	0<\frac{\delta}{2}&\leq\sum_{i\geq1}\int\limits _{B_{i,L}^{'}}\psi(\left\Vert x-y\right\Vert _{\mathbb{H}})(P_{2}-P_{1})(dx)\\
	& \leq\sum_{i\geq1}m_{i,L}(P_{2}-P_{1})(B_{i,L}^{'})+\sum_{i\geq1}\gamma_{i,L}P_{2}(B_{i,L}^{'})\\
	& \leq\sum_{i\geq1}m_{i,L}(P_{2}-P_{1})(B_{i,L}^{'})+\sup_{i\geq1}\gamma_{i,L}P_{2}\left(B_{r}(y)\right),
	\end{align*}
	where $m_{i,L}=\inf\limits _{x\in B_{i,L}^{'}}\psi(\left\Vert x-y\right\Vert _{\mathbb{H}})$
	and $M_{i,L}=\sup\limits _{x\in B_{i,L}^{'}}\psi(\left\Vert x-y\right\Vert _{\mathbb{H}})$
	and $\gamma_{i,L}=M_{i,L}-m_{i,L}$. Because $\psi$ is a bounded non-negative continuous and strictly decreasing function, we can choose $L\in\left(0,1\right)$ such that $\sup_{i\geq1}\gamma_{i,L}P_{2}\left(B_{r}(y)\right)<\frac{\delta}{4}$ or $\sup_{i\geq1}\gamma_{i,L}<\frac{\delta}{4P_{2}\left(B_{r}(y)\right)}$
	and thus $\sum_{i\geq1}m_{i,L}(P_{2}-P_{1})(B_{i,L}^{'})>0$. By lemma \ref{lem:sumajbj} there exists $N<\infty$ such that $\sum_{i=1}^{N}\left(P_{2}-P_{1}\right)\left(B_{i,L}^{'}\right)>0$,
	which immediately follows that $\left(P_{2}-P_{1}\right)\left(B_{r^{*}}(y)\right)>0$, where $r^{*}=\left(1-L^{N}\right)r$. 
	%To choose a proper $L\in\left(0,1\right)$ such that $\sup_{i\geq1}\gamma_{i,L}$
	%bounded above by $\frac{\delta}{4P_{2}\left(B_{r}(y)\right)}$, consider
	%that
	%\begin{align*}
	%\sup_{i\geq1}\gamma_{i,L} & =\sup_{i\geq1}\left\{ \sup_{x\in B_{i,L}^{'}}\psi(\left\Vert x-y\right\Vert _{\mathbb{H}})-\inf_{x\in %B_{i,L}^{'}}\psi(\left\Vert x-y\right\Vert _{\mathbb{H}})\right\} =\max\left\{ \psi\left(r_{1,L}\right),\sup_{i\geq2}\left\{ %\psi\left(r_{i,L}\right)-\psi\left(r_{\left(i-1\right),L}\right)\right\} \right\} \\
	%& =\max\left\{ \psi\left((1-L)r\right),\sup_{i\geq2}\left\{ \psi\left((1-L^{i})r\right)-\psi\left((1-L^{i-1})r\right)\right\} \right\}. 
	%\end{align*}
	%Since $\psi$ is a Lipschitz continuous function, there exists a positive constant $c$, such that $\left|\psi'\left(x\right)\right|\leq %c\left|x\right|$ and thus 
	%\begin{equation*}
	%\psi\left((1-L^{i})r\right)-\psi\left((1-L^{i-1})r\right)  =rL\psi'\left((1-L^{i-1})r\right)+o\left(L^{i-1}\left(1-L\right)r\right)\leq %rLc+cL\left(1-L\right)r\leq2rLc.
	%\end{equation*}
	%So it is sufficient to take $L\leq\min\left\{ %1-\frac{1}{r}\psi^{-1}\left(\frac{\delta}{4P_{2}\left(B_{r}(y)\right)}\right),\frac{\delta}{8rcP_{2}\left(B_{r}(y)\right)}\right\} $. Thus, %$\sup_{i\geq1}\gamma_{i,L^*}P_{2}\left(B_{r}(y)\right)<\frac{\delta}{4}$, and consequently $\sum_{i\geq1}m_{i,L^*}(P_{2}-P_{1})(B_{i,L^*}^{'})>0$.
\end{proof}

\begin{proof}[Proof of Corollary \ref{cor:smallBallProb}:]~	
	
	Let assume $\int_{\mathbb{H}}\psi(\left\Vert x-y_{2}\right\Vert _{\mathbb{H}})P(dx)>\int_{\mathbb{H}}\psi(\left\Vert x-y_{1}\right\Vert _{\mathbb{H}})P(dx)$
	and let $P_{-a}$ be the push forward of $P$ by the map $x\mapsto x+a$, which translates
	$x$ to $x+a$, then we have $\int_{\mathbb{H}}\psi(\left\Vert x\right\Vert _{\mathbb{H}})P_{-y_{2}}(dx)>\int_{\mathbb{H}}\psi(\left\Vert x\right\Vert _{\mathbb{H}})P_{-y_{1}}(dx)$
	and thus by the same argument as stated in the proof of Theorem \ref{thm:thm1}
	there exists $r>0$ big enough such that $\left(P_{-y_{2}}-P_{-y_{1}}\right)\left(B_{r}\left(0\right)\right)>0$
	and consequently $P\left(B_{r}(y_{2})\right)>P\left(B_{r}(y_{1})\right)$.
\end{proof}

\subsection{Proof of Theorem \ref{thm:existanceOfCharKer_l2}}
The existence proof of a continuous characteristic kernel for $\ell_2$ relies on the following theorem  by \citet{steinwart2019strictly}. We also need lemma \ref{lem:erinftyContl2Cont} to complete the proof.
\begin{thm}
	\textbf{\label{lem:Steinwart2017Lemma}\citep[Theorem 3.14]{steinwart2019strictly}}
	For a compact topological Hausdorff space $\left(X,\tau\right)$, the
	following statements are equivalent:
\end{thm}

\begin{enumerate}
	\item There exists a universal kernel $k$ on $X$.
	\item There exists a continuous characteristic kernel $k$ on $X$.
	\item $X$ is metrizable, i.e. there exists a metric generating the topology
	$\tau$.
\end{enumerate}
\begin{lem}
	\label{lem:erinftyContl2Cont}Let $\overline{\mathbb{R}}$ be the
	extended real line, and $\overline{\mathbb{R}}^{\infty}$ and
	$\mathbb{R}^{\infty}$ be the countable products of $\overline{\mathbb{R}}$
	and $\mathbb{R}$ respectively, which are equipped with the product topologies,
	and let $\ell_{2}\subset\overline{\mathbb{R}}^{\infty}$ be the Hilbert
	space of square summable sequences. If the function $f:\overline{\mathbb{R}}^{\infty}\to\mathbb{R}$
	is continuous, then $f|_{\ell^2}$, which is restriction
	of $f$ to $\ell_{2}$, is continuous with respect to the norm of
	$\ell_{2}$.
\end{lem}

\begin{proof}
	Assume that $\varphi : \overline{\mathbb{R}}\to [-1, 1]$ is defined as follows:
	$$\varphi(x)= \frac{x}{1+|x|}, ~~~~\forall x\in \mathbb{R}~~\mbox{and}~~~ \varphi(-\infty)=-1,~~~  \varphi(+\infty)=1.$$
	It is clear that $\varphi$ is a homeomorphic  and order-preserving. Consider  $$\rho(x,y) :=|\varphi(x)-\varphi(y)|,  ~~~ \mbox{for~all}~x, y\in \overline{\mathbb{R}}.$$ Then $\rho$ is a metric and the topology induced by this metric is equivalent to the order topology.  On the other hand, it is well  known  that the product topology on $\overline{\mathbb{R}}^\infty$ can be generated by the following metric \citep[][Theorem 2.6.6]{conwey},
	$$d(x, y):=\sum_{k=1}^\infty\frac{\rho(x_k, y_k)}{2^k\left(1+\rho(x_k, y_k)\right)}, \quad \forall x=(x_k), y=(y_k)\in  \overline{\mathbb{R}}^\infty.$$
	Now, we show that the topology induced by the metric $d$ on $\ell^2$ is weaker than the norm topology. For this purpose, let $x_n=(x_n^k), x=(x^k)\in \ell^2$ and $\|x_n-x\|\to 0$. Therefore, $|x_n^k-x^k|\to 0$ for all $k\in \mathbb{N}$. Because $\varphi$ is a homeomorphic, $\rho(x_n^k, x^k)\to 0$ for any $k\in \mathbb{N}$ and hence $d(x_n, x)\to 0$. Thus, if $f$ is continuous with respect to the product topology, then $f|_{\ell^2}$ is continuous with respect to the norm topology.
\end{proof}
\begin{proof}[Proof of Theorem \ref{thm:existanceOfCharKer_l2}:]
	Without
	loss of generality assume $\mathbb{H}$ to be the space of square
	summable sequences $\ell_{2}$, which is a subset of $\mathbb{R}^{\infty}$,
	and let $\mathcal{B}\left(\ell_{2}\right)$ be Borel sigma-algebra
	generated by the open sets of $\ell_{2}$. $\mathbb{R}$ is a one-dimensional
	locally compact Hausdorff space, and the extended real line $\overline{\mathbb{R}}$
	equipped with order topology is a metrizable Hausdorff and compact topological space.
	Equip both $\mathbb{R}^{\infty}$ and $\overline{\mathbb{R}}^{\infty}$
	with the product topologies and let $\mathcal{B}\left(\mathbb{R}^{\infty}\right)$
	and $\mathcal{B}\left(\overline{\mathbb{R}}^{\infty}\right)$ be the
	Borel sigma-algebra generated by the open sets of these topologies. Consider
	that $\mathcal{B}\left(\ell_{2}\right)=\left\{ A\cap\ell_{2}:A\in\mathcal{B}\left(\mathbb{R}^{\infty}\right)\right\} $,
	and we have $\mathcal{B}\left(\ell_{2}\right)\subseteq\mathcal{B}\left(\mathbb{R}^{\infty}\right)\subseteq\mathcal{B}\left(\overline{\mathbb{R}}^{\infty}\right)$.
	Note that $\mathcal{B}\left(\mathbb{R}^{\infty}\right)\subseteq\mathcal{B}\left(\overline{\mathbb{R}}^{\infty}\right)$
	Because we equipped extended real line $\overline{\mathbb{R}}$ with
	order topology, which includes the bases for the natural topology of $\mathbb{R}$.
	Let $\iota:\ell_{2}\to\overline{\mathbb{R}}^{\infty}$ be the usual
	inclusion map, then for every $A\in\mathcal{B}\left(\overline{\mathbb{R}}^{\infty}\right)$
	we have $\iota^{-1}\left(A\right)=A\cap\ell_{2}\in\mathcal{B}\left(\ell_{2}\right)$,
	so $\iota$ is a $\mathcal{B}\left(\ell_{2}\right)-\mathcal{B}\left(\overline{\mathbb{R}}^{\infty}\right)$
	measurable map and hence every $\ell_{2}$-valued random element is
	an $\overline{\mathbb{R}}^{\infty}$-valued random element and thus
	the space of Borel probability measures on $\left(\ell_{2},\mathcal{B}\left(\ell_{2}\right)\right)$
	is a subset of the space of Borel probability measures on $\left(\overline{\mathbb{R}}^{\infty},\mathcal{B}\left(\overline{\mathbb{R}}^{\infty}\right)\right)$.
	$\overline{\mathbb{R}}^{\infty}$ itself is a metrizable compact topological
	Hausdorff space, thus by invoking Lemma \ref{lem:Steinwart2017Lemma},
	there exists a continuous characteristic kernel $k\left(\cdot,\cdot\right)$
	on $\overline{\mathbb{R}}^{\infty}$, which by employing Lemma \ref{lem:erinftyContl2Cont},
	its restriction to $\ell_{2}$ is also continuous with respect to
	the norm of $\ell_{2}$.
\end{proof}
\subsection{Proof of Theorem \ref{thm:GausCahr_c00}}
Let $\ell_{2}$
be the space of square summable sequences with inner product $\left\langle \cdot,\cdot\right\rangle $
and norm $\left\Vert \cdot\right\Vert $, and let $\Lambda_{\theta}$
be the infinite-dimensional Gaussian measure on the measurable space
$\left(\mathbb{R}^{\infty},\mathcal{B}\left(\mathbb{R}^{\infty}\right)\right)$
defined as the product of countably many copies of normal distribution
with mean zero and variance $\theta$. The dual space of $\mathbb{R}^{\infty}$
is $c_{00}$, so characteristic function of Gaussian measure, for
any $x\in c_{00}$ equals to
\begin{equation}
\psi\left(x\right):=\int\limits _{\mathbb{R}^{\infty}}e^{-i\left\langle \omega,x\right\rangle }\Lambda_{2\sigma}\left(d\omega\right)=e^{-\sigma\left\Vert x\right\Vert ^{2}}.\label{eq:GausKerChar}
\end{equation}
Let $\mathbb{P}$ and $\mathbb{Q}$ be two arbitrary probability measures
over $c_{00}$ such that $\gamma_{k}\left(\mathbb{P},\mathbb{Q}\right)=0$,
then
\begin{align}
0=\gamma_{k}\left(\mathbb{P},\mathbb{Q}\right)^{2} & =\int\limits _{c_{00}}\int\limits _{c_{00}}e^{-\sigma\left\Vert x-y\right\Vert ^{2}}\left(\mathbb{P}-\mathbb{Q}\right)\left(dx\right)\left(\mathbb{P}-\mathbb{Q}\right)\left(dy\right)\nonumber \\
& =\int\limits _{c_{00}}\int\limits _{c_{00}}\left(\int\limits _{\mathbb{R}^{\infty}}e^{-i\left\langle \omega,x-y\right\rangle }\Lambda_{2\sigma}\left(d\omega\right)\right)\left(\mathbb{P}-\mathbb{Q}\right)\left(dx\right)\left(\mathbb{P}-\mathbb{Q}\right)\left(dy\right)\nonumber \\
& \stackrel{(a)}{=}\int\limits _{\mathbb{R}^{\infty}}\left(\int\limits _{c_{00}}\int\limits _{c_{00}}e^{-i\left\langle \omega,x-y\right\rangle }\left(\mathbb{P}-\mathbb{Q}\right)\left(dx\right)\left(\mathbb{P}-\mathbb{Q}\right)\left(dy\right)\right)\Lambda_{2\sigma}\left(d\omega\right)\nonumber \\
& =\int\limits _{\mathbb{R}^{\infty}}\left(\int\limits _{c_{00}}e^{-i\left\langle \omega,x\right\rangle }\left(\mathbb{P}-\mathbb{Q}\right)\left(dx\right)\int\limits _{c_{00}}e^{i\left\langle \omega,y\right\rangle }\left(\mathbb{P}-\mathbb{Q}\right)\left(dy\right)\right)\Lambda_{2\sigma}\left(d\omega\right)\nonumber \\
& =\int\limits _{\mathbb{R}^{\infty}}\left(\phi_{\mathbb{P}}\left(\omega\right)-\phi_{\mathbb{Q}}\left(\omega\right)\right)\left(\overline{\phi_{\mathbb{P}}\left(\omega\right)}-\overline{\phi_{\mathbb{Q}}\left(\omega\right)}\right)\Lambda_{2\sigma}\left(d\omega\right)\nonumber \\
& =\int\limits _{\mathbb{R}^{\infty}}\left|\phi_{\mathbb{P}}\left(\omega\right)-\phi_{\mathbb{Q}}\left(\omega\right)\right|^{2}\Lambda_{2\sigma}\left(d\omega\right).\label{eq:eqGausChar}
\end{align}
In the above equation, Fubini-Tonneli's theorem is invoked in (a).
Dual of $c_{00}$ with norm $\left\Vert \cdot\right\Vert $, is the
space of square summable sequences $\ell_{2}$. So to show that $\mathbb{P}=\mathbb{Q}$,
it is enough to show that $\phi_{\mathbb{P}}=\phi_{\mathbb{Q}}$ agrees
on $\ell_{2}$. By (\ref{eq:eqGausChar}) and by definition of the integral
and the fact that $\text{supp}\left(\Lambda_{2\sigma}\right)=\mathbb{R}^{\infty}$,
for any open set $B$ we have, 
\[
\inf_{\omega\in B}\left|\phi_{\mathbb{P}}\left(\omega\right)-\phi_{\mathbb{Q}}\left(\omega\right)\right|^{2}=0.
\]
Fix $\omega_{0}\in c_{00}$, and for any $m\in\mathbb{N}$ define
$$B_{m}:=\left\{ x\in\mathbb{R}^{m}:\sum_{i=1}^{m}\left(x_{i}-\omega_{0i}\right)^{2}<\frac{1}{m^2}\right\} \times\mathbb{R}^{\infty},$$ which is an open set in $\mathbb{R}^{\infty}$. Thus for each $m\in\mathbb{N}$,
we have, $$\inf_{\omega\in B_{m}}\left|\phi_{\mathbb{P}}\left(\omega\right)-\phi_{\mathbb{Q}}\left(\omega\right)\right|^{2}=0,$$
and so there exists $\omega_{m}\in B_{m}$ such that, $\left|\phi_{\mathbb{P}}\left(\omega_{m}\right)-\phi_{\mathbb{Q}}\left(\omega_{m}\right)\right|^{2}<\frac{1}{m}.$
Confirm that the sequence $\omega_{m}$ converges in the metric of
$\mathbb{R}^{\infty}$ to $\omega_{0}$, since 
\begin{align*}
d\left(\omega_{m},\omega_{0}\right) & =\sum_{k\geq1}2^{-k}\frac{\left|\omega_{mk}-\omega_{0k}\right|}{1+\left|\omega_{mk}-\omega_{0k}\right|}\leq\sum_{k=1}^{m}2^{-k}\frac{\nicefrac{1}{m}}{1+\nicefrac{1}{m}}+\sum_{k>m}2^{-k}\\
& \leq\frac{1}{m+1}\left(1-2^{-m}\right)+2^{-m}\to0.
\end{align*}
So $\left\langle \omega_{m},x\right\rangle \to\left\langle \omega_{0},x\right\rangle $
for any $x\in c_{00}$. By a simple application of Bounded Convergence
Theorem, we have
\begin{align*}
\lim_{m\to\infty}\left|\phi_{\mathbb{P}}\left(\omega_{m}\right)-\phi_{\mathbb{Q}}\left(\omega_{m}\right)\right|^{2} & =\lim_{m\to\infty}\left|\int\limits _{c_{00}}e^{-i\left\langle \omega_{m},x\right\rangle }\mathbb{P}\left(dx\right)-\int\limits _{c_{00}}e^{-i\left\langle \omega_{m},x\right\rangle }\mathbb{Q}\left(dx\right)\right|^{2}\\
& =\left|\int\limits_{c_{00}}\hspace{-0.5em}\lim_{m\to\infty}\hspace{-0.42em}e^{-i\left\langle \omega_{m},x\right\rangle }\mathbb{P}\left(dx\right)\hspace{-0.1em}-\hspace{-0.5em}\int\limits _{c_{00}}\hspace{-0.5em}\lim_{m\to\infty}\hspace{-0.42em}e^{-i\left\langle \omega_{m},x\right\rangle }\mathbb{Q}\left(dx\right)\right|^{2}\\
& =\left|\int\limits _{c_{00}}e^{-i\left\langle \omega_{0},x\right\rangle }\mathbb{P}\left(dx\right)-\int\limits _{c_{00}}e^{-i\left\langle \omega_{0},x\right\rangle }\mathbb{Q}\left(dx\right)\right|^{2}\\
& =\left|\phi_{\mathbb{P}}\left(\omega_{0}\right)-\phi_{\mathbb{Q}}\left(\omega_{0}\right)\right|^{2}
\end{align*}
and thus
\[
\left|\phi_{\mathbb{P}}\left(\omega_{0}\right)-\phi_{\mathbb{Q}}\left(\omega_{0}\right)\right|^{2}=\lim_{m\to\infty}\left|\phi_{\mathbb{P}}\left(\omega_{m}\right)-\phi_{\mathbb{Q}}\left(\omega_{m}\right)\right|^{2}\leq\lim_{m\to\infty}\frac{1}{m}\to0.
\]
\emph{So $\phi_{\mathbb{P}}=\phi_{\mathbb{Q}}$ on $c_{00}$. The
	space $c_{00}$ is dense in $\ell_{2}$, so $\phi_{\mathbb{P}}=\phi_{\mathbb{Q}}$
	agrees on $\ell_{2}$ and thus $\mathbb{P}=\mathbb{Q}$.}

\subsection{Proof of Proposition \ref{prop:GausKMF}}
Before providing the proof we need some tools, which are provided in the upcoming theorems and lemmas. The next theorem is a generalization of Ky Fan's inequality, which 
is useful to show convexity of the map $A\mapsto\left|I+A\right|^{-\nicefrac{1}{2}}$
on the convex set of positive trace-class operators that is crucial to prove Gaussian kernel is characteristic for the family of  Gaussian  distributions. The following
theorem is a special case of \citet[Theorem 1]{minh2017infinite}
when $\mu=\gamma=1$.
\begin{thm}
	\label{thm:(minh2017thm1)} Let
	$\mathbb{H}$ be an infinite-dimensional separable Hilbert space,
	and $A$, $B$ two arbitrary positive trace-class operators, for
	$0\leq\alpha\leq1$
	\[
	\left|\alpha\left(I+A\right)+\left(1-\alpha\right)\left(I+B\right)\right|\geq\left|I+A\right|^{\alpha}\left|I+B\right|^{1-\alpha}.
	\]
	For $0<\alpha<1$, equality occurs if and only if $A=B$.
\end{thm}

\begin{lem}
	\label{lem:logdetconcavity}Let $\mathbb{H}$ be a separable Hilbert
	space, and let $\left|\cdot\right|$ be the determinant of a non-negative
	symmetric operator on $\mathbb{H}$. $A\mapsto\left|I+A\right|^{-\nicefrac{1}{2}}$
	is a convex function over the convex set of positive trace-class operators
	on $\mathbb{H}$, and for any two arbitrary positive trace-class operators
	$A$ and $B$,
	\[
	2\left|I+\frac{A+B}{2}\right|^{-\nicefrac{1}{2}}\leq\left|I+A\right|^{-\nicefrac{1}{2}}+\left|I+B\right|^{-\nicefrac{1}{2}}
	\]
	and $2\left|I+\frac{A+B}{2}\right|^{-\nicefrac{1}{2}}=\left|I+A\right|^{-\nicefrac{1}{2}}+\left|I+B\right|^{-\nicefrac{1}{2}}$
	if and only if $A=B$.
\end{lem}

\begin{proof}
	By Theorem \ref{thm:(minh2017thm1)} we have 
	\[
	\log\left|I+\left(\alpha A+\left(1-\alpha\right)B\right)\right|\geq\alpha\log\left|I+A\right|+\left(1-\alpha\right)\log\left|I+B\right|,
	\]
	so $A\mapsto\log\left|I+A\right|$ is a concave function on the convex
	set of positive trace-class operators, and thus $A\mapsto\log\left|I+A\right|^{\nicefrac{-1}{2}}$
	is a convex function and also is $A\mapsto\left|I+A\right|^{\nicefrac{-1}{2}}$,
	since $x\mapsto e^{x}$ is a non-decreasing convex function. Consequently
	\[
	\left|I+\left(\frac{1}{2}A+\frac{1}{2}B\right)\right|^{\nicefrac{-1}{2}}\leq\frac{1}{2}\left|I+A\right|^{-\nicefrac{1}{2}}+\frac{1}{2}\left|I+B\right|^{-\nicefrac{1}{2}}
	\]
	and thus 
	\[
	2\left|I+\frac{A+B}{2}\right|^{-\nicefrac{1}{2}}\leq\left|I+A\right|^{-\nicefrac{1}{2}}+\left|I+B\right|^{-\nicefrac{1}{2}}.
	\]
	By invoking Theorem (\ref{thm:(minh2017thm1)}), equality occurs
	if and only if $A=B$.
\end{proof}
%The Woodbury matrix identity is useful for solving a variety of problems
%in linear algebra. The operator version is stated like this:
%\begin{thm}
%\textbf{\citep[Theorem 3.5.6]{hsing2015theoretical}} For operators $\mathscr{\mathcal{S}},\mathscr{T},\mathscr{U}$
%and $\mathscr{V}$ from $\mathbb{H}$ to $\mathbb{H}$, with $\mathscr{\mathcal{S}}$
%and $\mathscr{T}$ invertible, 
%\[
%\left(\mathscr{T}+\mathscr{U}\mathcal{S}^{-1}\mathscr{V}\right)^{-1}=\mathscr{T}^{-1}-\mathscr{T}^{-1}\mathscr{U}\left(\mathcal{S}+\mathscr{V}\mathscr{T%}^{-1}\mathscr{U}\right)^{-1}\mathscr{V}\mathscr{T}^{-1}
%\]
%\end{thm}
\begin{lem}
	\label{lem:Maniglia2004prop}\textbf{\citep[Proposition 1.2.8]{maniglia2004gaussian}} Let
	$\mathbb{H}$ be a separable Hilbert space and $\mathcal{N}\left(\mu,C\right)$
	be a Gaussian probability measure on $\mathbb{H}$ with mean function $\mu$
	and covariance operator $C$. For any $\sigma>0$ 
	\[
	\int\limits _{\mathbb{H}}e^{-\sigma\left\Vert x\right\Vert _{\mathbb{H}}^{2}}\mathcal{N}\left(\mu,C\right)\left(dx\right)=\left|I+2\sigma C\right|^{-\nicefrac{1}{2}}e^{-\sigma\left\langle \left(I+2\sigma C\right)^{-1}\mu,\mu\right\rangle }.
	\]
\end{lem}
\begin{proof}[Proof of Proposition \ref{prop:GausKMF}:]
	If $Y\sim\mathcal{N}\left(\mu,\boldsymbol{C}\right)$ then by lemma
	\ref{lem:Maniglia2004prop} we have
	
	\begin{align*}
	m_{P}(x) & =\int\limits _{\mathbb{H}}e^{-\sigma\left\Vert y-x\right\Vert _{\mathbb{H}}^{2}}\mathcal{N}\left(\mu,\boldsymbol{C}\right)\left(dy\right)=\int\limits _{\mathbb{H}}e^{-\sigma\left\Vert z\right\Vert _{\mathbb{H}}^{2}}\mathcal{N}\left(x-\mu,\boldsymbol{C}\right)\left(dz\right)\\
	& =\left|I+2\sigma\boldsymbol{C}\right|^{-1/2}e^{-\sigma\left\langle \left(I+2\sigma\boldsymbol{C}\right)^{-1}\left(x-\mu\right),\left(x-\mu\right)\right\rangle }.
	\end{align*} 
	Let $T_{1}=I+2\sigma C_{1}$, then
	\begin{align*}
	\left\langle m_{P_{1}},m_{P_{2}}\right\rangle _{\mathcal{H}_{k}} & =\int\limits _{\mathbb{H}}\int\limits _{\mathbb{H}}e^{-\sigma\left\Vert x-y\right\Vert _{\mathbb{H}}^{2}}\mathcal{N}\left(\mu_{1},C_{1}\right)\left(dx\right)\mathcal{N}\left(\mu_{2},C_{2}\right)\left(dy\right)\\
	& =\int\limits _{\mathbb{H}}\left|T_{1}\right|^{-\nicefrac{1}{2}}e^{-\sigma\left\langle T_{1}^{-1}\left(y-\mu_{1}\right),\left(y-\mu_{1}\right)\right\rangle }\mathcal{N}\left(\mu_{2},C_{2}\right)\left(dy\right)\\
	& =\left|T_{1}\right|^{-\nicefrac{1}{2}}\int\limits _{\mathbb{H}}e^{-\sigma\left\langle T_{1}^{\nicefrac{-1}{2}}\left(y-\mu_{1}\right),T_{1}^{\nicefrac{-1}{2}}\left(y-\mu_{1}\right)\right\rangle }\mathcal{N}\left(\mu_{2},C_{2}\right)\left(dy\right)\\
	& =\left|T_{1}\right|^{-\nicefrac{1}{2}}\int\limits _{\mathbb{H}}e^{-\sigma\left\Vert z\right\Vert _{\mathbb{H}}^{2}}\mathcal{N}\left(T_{1}^{\nicefrac{-1}{2}}\left(\mu_{2}-\mu_{1}\right),T_{1}^{\nicefrac{-1}{2}}C_{2}T_{1}^{\nicefrac{-1}{2}}\right)\left(dz\right)\\
	& =\left|T_{1}\right|^{-\nicefrac{1}{2}}\left|I+2\sigma T_{1}^{\nicefrac{-1}{2}}C_{2}T_{1}^{\nicefrac{-1}{2}}\right|^{\nicefrac{-1}{2}}\\
	& \qquad\qquad\qquad e^{-\sigma\left\langle \left(I+2\sigma T_{1}^{\nicefrac{-1}{2}}C_{2}T_{1}^{\nicefrac{-1}{2}}\right)^{-1}T_{1}^{\nicefrac{-1}{2}}\left(\mu_{2}-\mu_{1}\right),T_{1}^{\nicefrac{-1}{2}}\left(\mu_{2}-\mu_{1}\right)\right\rangle }\\
	& =\left|T_{1}\right|^{-\nicefrac{1}{2}}\left|I+2\sigma T_{1}^{-1}C_{2}\right|^{\nicefrac{-1}{2}}\\
	& \qquad\qquad\qquad e^{-\sigma\left\langle T_{1}^{\nicefrac{-1}{2}}\left(I+2\sigma T_{1}^{\nicefrac{-1}{2}}C_{2}T_{1}^{\nicefrac{-1}{2}}\right)^{-1}T_{1}^{\nicefrac{-1}{2}}\left(\mu_{2}-\mu_{1}\right),\left(\mu_{2}-\mu_{1}\right)\right\rangle }\\
	& =\left|I+2\sigma\left(C_{1}+C_{2}\right)\right|^{\nicefrac{-1}{2}}e^{-\sigma\left\langle \left(I+2\sigma\left(C_{1}+C_{2}\right)\right)^{-1}\left(\mu_{2}-\mu_{1}\right),\left(\mu_{2}-\mu_{1}\right)\right\rangle },
	\end{align*}
	and thus 
	\begin{align*}
	\left\Vert m_{P_{1}}-m_{P_{2}}\right\Vert _{\mathcal{H}_{k}}^{2} & =\left\Vert m_{P_{1}}\right\Vert _{\mathcal{H}_{k}}^{2}+\left\Vert m_{P_{2}}\right\Vert _{\mathcal{H}_{k}}^{2}-2\left\langle m_{P_{1}},m_{P_{2}}\right\rangle _{\mathcal{H}_{k}}\\
	& =\left|I+4\sigma C_{1}\right|^{-\nicefrac{1}{2}}+\left|I+4\sigma C_{2}\right|^{-\nicefrac{1}{2}}\\
	&\phantom{=}-2\left|I+2\sigma\left(C_{1}+C_{2}\right)\right|^{\nicefrac{-1}{2}}e^{-\sigma\left\langle \left(I+2\sigma\left(C_{1}\hspace{-0.2em}+\hspace{-0.2em}C_{2}\right)\right)^{-1}\left(\mu_{2}-\mu_{1}\right),\left(\mu_{2}-\mu_{1}\right)\right\rangle }.
	\end{align*}
	By invoking lemma \ref{lem:logdetconcavity} we have 
	$$\left|I+4\sigma C_{1}\right|^{-\nicefrac{1}{2}}+\left|I+4\sigma C_{2}\right|^{-\nicefrac{1}{2}}\geq2\left|I+2\sigma\left(C_{1}+C_{2}\right)\right|^{\nicefrac{-1}{2}},$$ 
	and the equality occurs if and only if $C_{1}=C_{2}$. So $\left\Vert m_{P_{1}}-m_{P_{2}}\right\Vert _{\mathcal{H}_{k}}^{2}=0$
	if and only if $\mu_{1}=\mu_{2}$ and $C_{1}=C_{2}$. Hence, Gaussian
	kernel is characteristic for the family of Gaussian distributions.
\end{proof}

\subsection{Proof of Proposition \ref{thm:tensorProdCharKer}}
\begin{proof}
	We first give a proof for the product-kernel. A proof for the sum-kernel follows the same approach. Let $\mathscr{P}$ be the collection of probability measures on a
	separable Hilbert space $\mathbb{H}$, and $k\left(\cdot,\cdot\right):\mathbb{H}\times\mathbb{H}\longrightarrow\mathbb{R}$
	a characteristic kernel on $\mathbb{H}$. Consider the kernel mean with product-kernel 
	
	\[
	m_{k^{n}}:\mathscr{P}^{n}\rightarrow\mathbb{\mathcal{H}}_{k^{n}}\qquad\otimes_{j=1}^{n}P_{j}\mapsto m_{\otimes_{i=1}^{n}P_{j}}(x_{1},\ldots,x_{n})
	\]
	such that for any $x_{1},\ldots,x_{n}\in\mathbb{H}$,
	\begin{align*}
	m_{\otimes_{j=1}^{n}P_{j}}(x_{1},\ldots,x_{n}):&=\int\limits _{\mathbb{H}^{n}}\left(\prod_{i=1}^{n}k(x_{i},y_{i})\right)\otimes_{j=1}^{n}P_{j}(dy_{j})\\
	&=\prod_{i=1}^{n}\int\limits _{\mathbb{H}}k(x_{i},y_{i})P_{i}(dy_{i})=\prod_{i=1}^{n}m_{P_{i}}\left(x_{i}\right).
	\end{align*}
	Let $\mathbb{P},\mathbb{Q}\in\mathscr{P}^{n}$ i.e. $\mathbb{P}=\otimes_{j=1}^{n}P_{j}$
	, $\mathbb{Q}=\otimes_{j=1}^{n}Q_{j}$ such that $\mathbb{P}\neq\mathbb{Q}$.
	Given $k$ is characteristic on $\mathbb{H}$, there exists $1\leq i\leq n$
	such that $P_{i}\neq Q_{i}$ and $m_{P_{i}}\left(\cdot\right)\neq m_{Q_{i}}\left(\cdot\right)$, thus there exists $\left(x_{n}\right)\in\mathbb{H}^{n}$ such that $\prod_{i=1}^{n}m_{P_{i}}\left(x_{i}\right)\neq\prod_{i=1}^{n}m_{Q_{i}}\left(x_{i}\right)$. Similarly let
	
	\[
	m_{k^{n}}:\mathscr{P}^{n}\rightarrow\mathbb{\mathcal{H}}_{k^{n}}\qquad\otimes_{j=1}^{n}P_{j}\mapsto m_{\otimes_{i=1}^{n}P_{j}}(x_{1},\ldots,x_{n})
	\]
	such that for any $x_{1},\ldots,x_{n}\in\mathbb{H}$,
	\begin{align*}
	m_{\otimes_{i=1}^{n}P_{j}}(x_{1},\ldots,x_{n}):&=\int\limits _{\mathbb{H}^{n}}\left(\sum_{i=1}^{n}k(x_{i},y_{i})\right)\otimes_{j=1}^{n}P_{j}(dy_{j})\\
	&=\sum_{i=1}^{n}\int\limits _{\mathbb{H}}k(x_{i},y_{i})P_{i}(dy_{i})=\sum_{i=1}^{n}m_{P_{i}}\left(x_{i}\right).
	\end{align*}
	Let $\mathbb{P},\mathbb{Q}\in\mathscr{P}^{n}$ i.e. $\mathbb{P}=\otimes_{j=1}^{n}P_{j}$
	, $\mathbb{Q}=\otimes_{j=1}^{n}Q_{j}$ and $\mathbb{P}\neq\mathbb{Q}$.
	Given $k$ is characteristic on $\mathbb{H}$, there exists $1\leq i\leq n$
	such that $P_{i}\neq Q_{i}$ and $m_{P_{i}}\left(\cdot\right)\neq m_{Q_{i}}\left(\cdot\right)$,
	thus there exists $\left(x_{n}\right)\in\mathbb{H}^{n}$ such that
	$\sum_{i=1}^{n}m_{P_{i}}\left(x_{i}\right)\neq\sum_{i=1}^{n}m_{Q_{i}}\left(x_{i}\right)$.
\end{proof}
%\subsection{Proof of Proposition \ref{Prop:GausCase_SufficientStatistics}}
%Based on the logarithm of the kernel mean function given in (\ref{eq:modelReg}), we have
%\begin{align*}
%\log~ & m_{\otimes_{i=1}^{n}\mathcal{N}\left(\mu,C\right)}(y_{1},\ldots,y_{n})=\sum_{i=1}^{n}\sum_{j\geq1}\frac{-\sigma}{1+2\sigma\lambda_{j}}\langle y_{i}-\mu,\psi_{j}\rangle^{2}-\frac{n}{2}\sum_{j\geq1}\log\left(1+2\sigma\lambda_{j}\right)\\
%& =\sum_{j\geq1}\frac{-\sigma}{1+2\sigma\lambda_{j}}\sum_{i=1}^{n}\langle y_{i}-\bar{y}+\bar{y}-\mu,\psi_{j}\rangle^{2}-\frac{n}{2}\sum_{j\geq1}\log\left(1+2\sigma\lambda_{j}\right)\\
%& =\sum_{j\geq1}\frac{-\sigma}{1+2\sigma\lambda_{j}}\sum_{i=1}^{n}\left[\langle y_{i}-\bar{y},\psi_{j}\rangle^{2}+\langle\bar{y}-\mu,\psi_{j}\rangle^{2}-2\langle y_{i}-\bar{y},\psi_{j}\rangle\langle\bar{y}-\mu,\psi_{j}\rangle\right]-\frac{n}{2}\sum_{j\geq1}\log\left(1+2\sigma\lambda_{j}\right)\\
%& =\sum_{j\geq1}\frac{-\sigma}{1+2\sigma\lambda_{j}}\sum_{i=1}^{n}\left[\langle y_{i}-\bar{y},\psi_{j}\rangle^{2}+\langle\bar{y}-\mu,\psi_{j}\rangle^{2}\right]-\frac{n}{2}\sum_{j\geq1}\log\left(1+2\sigma\lambda_{j}\right)\\
%& =\sum_{j\geq1}\frac{-n\sigma}{1+2\sigma\lambda_{j}}\left[\langle\hat{C}_{Y}\psi_{j},\psi_{j}\rangle+\langle\bar{y}-\mu,\psi_{j}\rangle^{2}\right]-\frac{n}{2}\sum_{j\geq1}\log\left(1+2\sigma\lambda_{j}\right)
%\end{align*}

\section*{Acknowledgements}
The first author is grateful to the Graduate office of the University of Isfahan for their support. Part of this work was done while Saeed Hayati was visiting in the Institute of Statistical Mathematics under the support by the Research Organization of Information and Systems.  KF has been supported in part by JSPS KAKENHI 18K19793.
Afshin Parvardeh gratefully thanks Professor Victor Panaretos and EPFL in Switzerland for the kind hospitality  that received during spending his sabbatical leave at EPFL, in which this work, in part, was prepared.
%\printbibliography

\bibliography{bibtex_MKM}

\end{document}